\documentclass[openright,a4paper,11pt]{amsart}

\usepackage[utf8]{inputenc}
\usepackage[T1]{fontenc}
\usepackage{amsmath}
\usepackage{amssymb}

\usepackage[frenchb]{babel}

\usepackage{amsfonts}
\input xy
\xyoption{all}
\usepackage[dvips]{graphicx}
\usepackage{boxedminipage}
\usepackage{color}

\newcommand{\ZZ}{{\mathbb Z}}
\newcommand{\NN}{{\mathbb N}}

\newcommand{\KK}{{\mathbb K}}
\newcommand{\CC}{{\mathbb C}}
\newcommand{\RR}{{\mathbb R}}

\newcommand{\TT}{{\mathbb T}}

\newtheorem{thm}{Théoreme}[section]
\newtheorem{defi}[thm]{Définition}
\newtheorem{prop}[thm]{Proposition}
\newtheorem{lemma}[thm]{Lemme}
\newtheorem{cor}[thm]{Corollaire}

\newtheorem{exo}[thm]{Exercice}
          {\theoremstyle{remark}
\newtheorem{exe}[thm]{Exemple}}
          {\theoremstyle{definition}
\newtheorem{rem}[thm]{Remarque}}
          \newtheorem{pb}[]{Problème}

\begin{document}
\title{Géométries énumératives complexe, réelle et tropicale}
\author{Erwan Brugallé}
\address{Université Pierre et Marie Curie,  Paris 6, 175 rue du Chevaleret, 75 013 Paris, France}
\email{brugalle@math.jussieu.fr}
\date{\today}

\maketitle
\hspace{80ex} \textit{À Seal}

\begin{abstract}
 Ce texte est une introduction à la géométrie algébrique énumérative
 et aux applications de la géométrie tropicale en géométrie
 classique, basée sur un cours 
 dispensé pendant les Journées Mathématiques X-UPS des 14 et 15 mai
 2008 à l'École Polytechnique. Le public de ces journées est composé
 de professeurs de mathématiques en classes
 préparatoires, et cette  introduction se veut accessible à un
 étudiant en première année de master de mathématiques.

Ce texte  est  extrait du livre \cite{BIT} regroupant les 3 cours
donnés lors de ces journées.
J'y ai simplement ajouté
quelques modifications mineures afin de le rendre indépendant des autres
chapitres. 
 \end{abstract}\tableofcontents

\section{Introduction}
Notre objectif ici est d'expliquer comment résoudre
simplement, par 
des méthodes combinatoires, un grand nombre de problèmes énumératifs
 grâce à la géométrie tropicale. 

La géométrie énumérative est la branche des mathématiques qui tente de
répondre à
des questions du style

\hspace{3ex}  "Combien de droites passent par 2 points dans le plan?" (facile),

\hspace{3ex} "Combien de coniques passent par 5 points dans le plan?" (facile),

\hspace{3ex} "Combien de cubiques (i.e. courbes définies par un polynôme de degré
3) passent par 9 points dans le plan?" (facile),

\hspace{3ex} "Combien de cubiques se recoupant une fois passent par 8 points dans le plan?" (moins
facile), 

\hspace{3ex} ...

\vspace{1ex}
Les problèmes énumératifs que nous considérons dans ce
texte ont été posés en toute généralité au XIXème siècle.

Si nous  comptons les courbes définies par des polynômes à coefficients
complexes, alors le nombre de courbes ne dépend pas de la
configuration de points choisie, tout comme le nombre de racines
complexes d'un polynôme en une variable à coefficients complexes est
toujours égal à 
son degré. Cela facilite quelque peu  l'énumération des courbes complexes, 
mais il aura tout de même fallu 
 attendre 
les années 1990 
et
les travaux de Kontsevich puis
de Caporaso et Harris
pour obtenir une réponse complète.

En revanche, si l'on compte les courbes définies par des
polynômes à coefficients réels, ce nombre dépend fortement des points
choisis, ce qui complique singulièrement le problème... Nous pouvons toujours
affirmer que le nombre de courbes réelles est plus petit que le nombre
de courbes complexes, tout comme le nombre de racines réelles d'un
polynôme réel en une variable est plus petit que le nombre de ses
racines complexes, cela ne nous avance finalement pas beaucoup.
Au début des années 2000, Welschinger a montré qu'en comptant avec un
signe $+1$ ou $-1$ les courbes
réelles se recoupant un nombre maximum de fois, nous obtenons un nombre
indépendant de la 
configuration de points choisie.  De plus, ce nombre, appelé
\textit{invariant de Welschinger}, nous donne une borne inférieure sur
le nombre de courbes réelles quelle que soit la configuration de
points! Ainsi, grâce à ces invariants, Itenberg, Kharlamov
et Shustin ont pu démontrer que par $3d-1$ points du plan, passait
\textit{toujours} une courbe algébrique réelle de degré $d$
(i.e. définie par un polynôme en deux variables de degré $d$ à
coefficients réels) 
se recoupant un nombre maximum de fois.

La géométrie énumérative est un très joli domaine des mathématiques où
de nombreux problèmes de base peuvent être résolus à l'aide de quelques idées
astucieuses. 
Cependant, les techniques requises pour traiter un problème
énumératif complexe ou réel plus général deviennent rapidement sophistiquées, et
leur maîtrise 
demande beaucoup de temps et d'investissement.

Parallèlement, une nouvelle géométrie a émergée au début de ce
troisième millénaire, la \textit{géométrie  tropicale}. Les objets de
cette géométrie sont linéaires par morceaux ce qui simplifie
considérablement leur étude! En particulier, 
la géométrie énumérative tropicale est beaucoup plus simple que les
géométries énumératives complexe ou réelle.
De plus,  un théorème très profond
de  Mikhalkin nous dit que l'on peut compter des courbes complexes
ou réelles simplement en comptant des courbes tropicales. Ainsi, 
grâce à la 
géométrie tropicale et à un algorithme astucieux,  Mikhalkin a pu calculer 
pour la première fois les invariants de Welschinger.

\vspace{2ex}
Ce texte est une introduction à la géométrie énumérative que j'espère
compréhensible au  niveau de première année de Master. Quelques notes
de bas de page donnent des précisions sur certains termes 
employés, mais leur compréhension n'est absolument pas nécessaire  à
la compréhension 
du texte.

\vspace{2ex}

La section \ref{alg cplx} est consacrée à l'initiation à la géométrie
énumérative complexe et à la généralisation de la question
``Combien de droites passent par 2 points du plan?''. Dans la section
\ref{enum reel}, nous verrons que les choses se compliquent lorsque
nous nous intéressons aux courbes réelles, et nous définirons les
invariants de Welschinger.
Les courbes tropicales et leur géométrie énumérative seront traitées à
la section \ref{trop}, puis nous expliquerons à la section \ref{etage}
comment passer des courbes tropicales à des objets encore plus
simples, les \textit{diagrammes en étages}. Nous utiliserons ces
diagrammes à la section \ref{appl} pour résoudre quelques problèmes
énumératifs complexes et réels.

\vspace{3ex}
\textbf{Remerciements :} Je tiens à remercier chaleureusement Lucia
Lopez de Medrano, Assia
Mahboubi, Gurvan Mével, Nicolas Puignau, Emmanuel Rey, Jean Jacques Risler et Claude
Sabbah pour
leur relecture 
attentive et leurs 
critiques constructives. Je remercie aussi le public des Journées
Mathématiques X-UPS dont les commentaires ont contribué à améliorer le
texte initial.

\section{Géométrie énumérative complexe}\label{alg cplx}

\subsection{Échauffement}

C'est un fait admis de tout le monde, par 2 points distincts du plan
 passe une unique droite. Attardons nous un instant sur la
 démonstration de cette proposition évidente.

Une droite du plan est donnée par un équation de la forme $aX+bY+c=0$
où $a,b$ et $c$ sont 3 nombres (réels ou complexes, peu importe
ici). Un point $p$ du plan est donné par 2 coordonnées
$(x_p,y_p)$, et ce point est sur la droite d'équation $aX+bY+c=0$ si
et seulement si ses coordonnées satisfont son équation, c'est-à-dire
si et seulement si 
$ax_p+by_p+c=0$. 
Si $q$ est  un deuxième point
du plan, alors la droite d'équation $aX+bY+c=0$ passe par $p$ et $q$
si et seulement si les 3 nombres $a,b$ et $c$ sont solutions du
système d'équation
\begin{equation}\label{syst 1}
\left\{ \begin{array}{ccc}
ax_p+by_p+c&=&0
\\ax_q+by_q+c&=&0
\end{array}\right. 
\end{equation}

Les points $p$ et $q$ étant donnés, nous avons donc  un système linéaire de 2
équations en les variables $a,b$ et 
$c$. De plus, si les points $p$ et $q$ sont distincts, alors ce
système est de rang 2. L'ensemble des triplets $(a,b,c)$
solutions est donc infini, et l'ensemble des droites passant par $p$ et $q$
à l'air infini... En fait, ces solutions forment une droite
vectorielle (2 équations et 3 inconnues), c'est-à-dire que tous les
triplets solutions sont colinéaires. Or, si $(a,b,c)=\lambda
(a',b',c')$ et si $(a,b,c)\ne(0,0,0)$, alors les deux
équations $aX+bY+c=0$ et $a'X+b'Y+c'=0$ définissent la même droite. En
d'autres termes, tous les triplets solutions du système (\ref{syst 1})
définissent la même droite!
Et  par deux points distincts du plan passe donc une
unique droite.

Qu'avons nous utilisé dans notre preuve? 
En passant de la droite à son équation, nous avons tout d'abord traduit
 un problème 
géométrique en un problème algébrique. Puis, pour montrer que
 2 points sont nécessaires et suffisants pour déterminer une
droite, nous avons simplement calculé la dimension
de l'espace des droites du
plan, et nous avons trouvé 2. En effet, une droite est déterminée
par les 3 coefficients de son équation, mais 3 coefficients
colinéaires définissent la 
même droite. La dimension de l'espace des droites est donc $3-1=2$.

Nous venons de résoudre un problème de géométrie énumérative à propos
de polynômes en deux variables de degré 1. Mais en y réfléchissant
bien, où avons nous utilisé que nos polynômes étaient de
degré 1? Cette  hypothèse ne nous a finalement servi que pour
identifier clairement l'espace dans lequel nous avons travaillé, afin
de pouvoir calculer sa dimension.
Ainsi, 
modulo quelques définitions, la même méthode doit marcher pour des
polynômes en deux variables de 
n'importe quel degré.

\subsection{Un problème simple de géométrie énumérative}\label{enum simple}

Rappelons tout d'abord la définition du degré d'un polynôme en deux
variables.

\begin{defi}
Soit $P(X,Y)=\sum a_{i,j}X^iY^j$ un polynôme dans $\CC[X,Y]$. Le degré
de $P(X,Y)$ est le maximum de la somme $i+j$   lorsque $a_{i,j}$
est non nul.
\end{defi}
\begin{rem}
Bien sûr, le corps de base ne joue aucun rôle dans cette définition
qui est la même pour un polynôme dans $\KK[X,Y]$, quel que soit le
corps $\KK$.
\end{rem}

Comme d'habitude, nous noterons $\CC_d[X,Y]$ l'espace vectoriel des
polynômes de degré au plus $d$. Une droite du plan est donnée par
une équation de la 
forme $aX+bY+c=0$, c'est donc 
l'ensemble solution d'un polynôme de degré 1. Plus généralement, un
sous ensemble de 
$\CC^2$ définit par une équation polynomiale en deux variables est
appelé une \textit{courbe algébrique}.

\begin{defi}
Soit $P(X,Y)$ un polynôme dans $\CC[X,Y]$ de degré au moins 1. Alors
l'ensemble $C$ des 
solutions de l'équation $P(X,Y)=0$ dans $\CC^2$ est appelé une courbe algébrique
complexe. 

On dit que la courbe $C$ est irréductible si le polynôme $P$ est
irréductible. Le degré de $C$ est le degré de $P(X,Y)$.
\end{defi}
Afin d'alléger un peu le texte, nous utiliserons dans la suite
l'expression ``courbe algébrique'' plutôt que  ``courbe 
algébrique complexe'' lorsque cela ne prêtera pas à confusion.

Une courbe algébrique irréductible est une courbe ``minimale'', en ce
sens qu'elle
n'est pas l'union de deux courbes algébriques. En effet, si
$P(X,Y)=P_1(X,Y)P_2(X,Y)$, alors la courbe définie par $P(X,Y)$ est
l'union des courbes définies par $P_1(X,Y)$ et $P_2(X,Y)$\footnote{On pourrait
  avoir $P_1(X,Y)=P_2(X,Y)$, on aurait dans ce cas là une courbe
  \textit{multiple}.}. Nous avons déjà étudié le cas des courbes algébriques de
degré 1, regardons les courbes algébriques de degré 2 et 3.

\begin{exe}
Les courbes algébriques de degré 2 sont appelées \textit{coniques},
et à changement de coordonnées affine près de $\CC^2$, il n'existe que 5
coniques dont les équations sont $X^2+Y^2-1=0$, $X^2-Y=0$, 
$X^2-Y^2=0$, $Y^2-1=0$
et $Y^2=0$. Les deux premières coniques sont irréductibles, mais pas
les trois dernières. Ces cinq coniques sont représentées sur la Figure
\ref{coniques}. 
\end{exe}

\begin{figure}[h]
\begin{center}
\begin{tabular}{ccccccccc}
\includegraphics[width=2cm, angle=0]{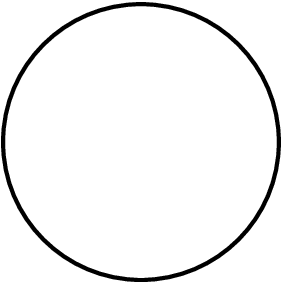}&\hspace{2ex}&
\includegraphics[width=2cm, angle=0]{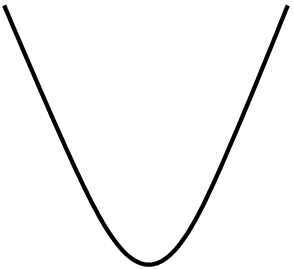}&\hspace{2ex}&
\includegraphics[width=2cm, angle=0]{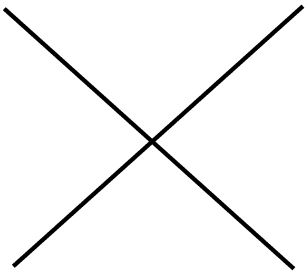}&\hspace{2ex}&
\includegraphics[width=2cm, angle=0]{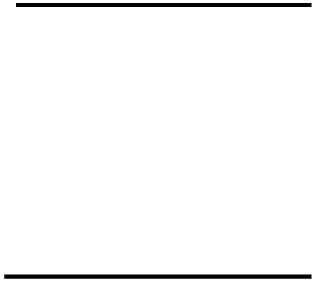}&\hspace{2ex}&
\includegraphics[width=2cm, angle=0]{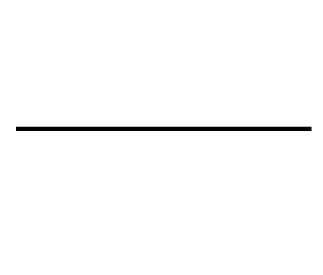}
\\ a) $X^2+Y^2-1=0$ && b) $X^2-Y=0$&& c) $X^2-Y^2=0$ &&d) $Y^2-1=0$
&& e) $Y^2=0$
\end{tabular}
\end{center}
\caption{Classification des coniques dans $\CC^2$}
\label{coniques}
\end{figure}

Évidemment, nous trichons un peu lorsque nous dessinons des courbes
dans $\CC^2$, puisque nous ne dessinons en fait que ce qui se passe
dans $\RR^2$... 

\begin{exe}
Les courbes
algébriques de degré 3 sont appelées \textit{cubiques}. La cubique
d'équation  $Y^2 - X(X^2+1)=0$ est
représentée 
sur la Figure \ref{cub}. Une grande majorité des cubiques possède une
propriété assez extraordinaire : elles peuvent être munies d'une loi de
groupe abélien! Cela fait des cubiques  des objets très appréciés des
géomètres algébristes et des  cryptographes. 
\end{exe}

\begin{figure}[h]
\begin{center}
\begin{tabular}{c}
\includegraphics[width=1cm, angle=0]{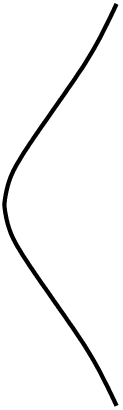}
\end{tabular}
\end{center}
\caption{La cubique dans $\CC^2$ d'équation $Y^2 - X(X^2+1)=0$}
\label{cub}
\end{figure}

Nous pouvons déjà faire deux remarques intéressantes. Tout d'abord,
comme $\CC$ est 
algébriquement clos, une courbe
algébrique $C$ n'est jamais vide : pour tout nombre complexe $x_0$, le
polynôme $P(x_0,Y)$ est maintenant un polynôme en une variable à
coefficients complexes et donc
admet au moins une racine dans $\CC$. Il existe donc un point de $C$ de la forme
$(x_0,y_0)$ pour tout $x_0$ dans $\CC$. 
Ensuite, comme dans le cas des droites, 
deux
polynômes non nuls multiples l'un de l'autre définissent la même courbe
algébrique\footnote{On peut donc voir l'espace des courbes algébriques comme
le projectivisé de l'espace des polynômes.}. 

Combien est-il nécessaire de fixer de points
 pour caractériser une courbe algébrique de
degré $d$? La réponse est encore une fois donnée par un calcul de
dimension.
Nous voyons apparaître dans la proposition suivante le mot
\textit{générique}. Ce mot reviendra constamment dans la suite de ce
texte, sans être pourtant jamais vraiment défini.
En effet, définir rigoureusement le mot \textit{générique} nous
demanderait un travail long et technique nous éloignant du sujet
traité ici\footnote{De manière générale, \textit{générique} signifie
  ``en dehors d'une sous variété algébrique d'une certaine variété
  algébrique''.}. Heureusement, le sens 
``intuitif'' du mot \textit{générique}  devrait  parfaitement suffire à
 la compréhension de 
ce texte. 

\begin{prop}\label{enum1}
Par $\frac{d(d+3)}{2}$ points génériques de $\CC^2$ passe une unique
courbe algébrique de degré $d$.
\end{prop}
\begin{proof}
Calculons la dimension de l'espace $\CC_d[X,Y]$. Un élément de
cet espace s'écrit
$$\sum_{i+j\le d} a_{i,j}X^iY^j $$
la dimension de $\CC_d[X,Y]$ est donc le nombre de couples $(i,j)$ dans
$\NN^2$ vérifiant $i+j\le d$. En faisant varier $i$ de $0$ à $d$ et en
comptant le nombre de $j$ possibles nous obtenons
$$\dim (\CC_d[X,Y])= (d+1) + d +\ldots + 2+1=\frac{(d+2)(d+1)}{2}  $$
Chercher les polynômes  $P(X,Y)$ qui s'annulent en un
point $(x_p,y_p)$ donné revient à résoudre l'équation $P(x_p,y_p)=0$
qui est linéaire en les coefficients $a_{i,j}$ de $P(X,Y)$.  Ainsi, si
nous cherchons les polynômes  $P(X,Y)$ qui s'annulent en
$\frac{d(d+3)}{2}$ points fixés, nous nous ramenons à résoudre un système de
$\frac{d(d+3)}{2}$ équations linéaires à $\frac{(d+2)(d+1)}{2}  $
inconnues. Si les points 
sont en position générique dans le plan, alors ce système est de rang
maximal et l'ensemble solution est donc un sous espace vectoriel de
$\CC_d[X,Y]$ de dimension 
$\frac{(d+2)(d+1)}{2}- \frac{d(d+3)}{2}=1 $. Or deux polynômes
multiples l'un de l'autre définissent la même courbe, la
courbe algébrique de degré $d$ passant par nos points est donc unique. 
\end{proof}

\begin{exe}
D'après la Proposition \ref{enum1}, il existe une unique conique
passant par 5 points en position générique dans le plan. Par exemple,
la conique  représentée sur la Figure \ref{coniques 2}b est l'unique
conique passant
par les 5 points représentés sur la Figure \ref{coniques 2}a. 
\end{exe}

\begin{figure}[h]
\begin{center}
\begin{tabular}{ccc}
\includegraphics[width=2cm,
  angle=0]{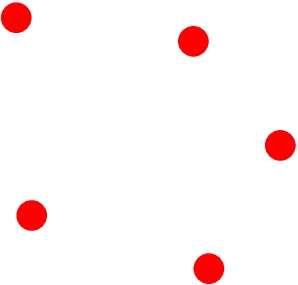}& \hspace{10ex} &
\includegraphics[width=2cm, angle=0]{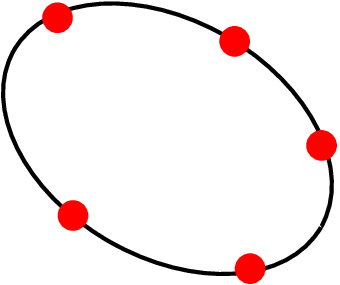}
\\ a) && b)
\end{tabular}
\end{center}
\caption{Une unique conique par 5 points}
\label{coniques 2}
\end{figure}

En regardant de près la démonstration du fait que par deux points
passe une unique droite, nous avons compris que nous pouvions
généraliser sans trop d'efforts le problème et sa solution aux courbes
algébriques de n'importe quel degré.
Maintenant que nous avons résolu ce problème plus général, observons
de plus près encore nos courbes algébriques.
Puisque nous avons choisi des points génériques, les courbes
solutions vont aussi être génériques : elles ne
se ``recoupent'' jamais. Or, il existe des courbes que se recoupent,
mais ces courbes  ne seront \textit{jamais}
solution... La raison à cela est que l'espace des courbes
qui se recoupent est beaucoup plus petit que l'espace de toutes les courbes
de degré $d$. Plus précisément, il a une dimension de moins. Mais
alors, si nous fixons un point de moins, pouvons nous faire passer une
courbe par les points restant et qui se recouperait?

Comme on peut s'en douter à la lecture du paragraphe précédent, nous
devons préalablement aller un peu plus loin dans 
l'étude des courbes 
algébriques 
avant de pouvoir poser rigoureusement notre nouveau problème. 

\subsection{Un problème énumératif plus général... et plus compliqué}
Une courbe algébrique $C$ est définie par une équation polynomiale
$P(X,Y)=0$. Comme toute courbe définie par une équation implicite, les
points de $C$ sont  naturellement séparés en deux 
ensembles. Ceux pour lesquels la différentielle de
$P(X,Y)$ ne s'annule pas, et les autres. 

\subsubsection{Courbes algébriques nodales}

Les points $p$ de $C$ pour lesquels la
différentielle de $P(X,Y)$ est non nulle en $p$ sont dits \textit{non
  singuliers}. Ce sont les points de $C$ les plus simples possible. 
D'après le théorème
des fonctions implicites, la courbe $C$ ressemble
à un graphe 
de fonction dans des coordonnées adéquates au voisinage d'un tel point. 
Les points $p$ de $C$ pour lesquels la
différentielle de $P(X,Y)$ est nulle en $p$ sont appelés les points
\textit{singuliers} de $C$. Au voisinage d'un tel point, la courbe $C$
peut prendre des formes variées et être extrêmement compliquée! Le plus
simple des points singuliers 
est le \textit{point double}.

\begin{defi}
Un point $p$ d'une courbe algébrique $C$ d'équation $P(X,Y)=0$
est appelé  point double de 
$C$ si la différentielle de $P(X,Y)$ en $p$ est nulle et si la
différentielle seconde de $P(X,Y)$ en $p$ est une forme quadratique
non dégénérée. 
\end{defi}

\begin{rem}\label{rem diff}
La différentielle et la
différentielle seconde d'un polynôme en $(0,0)$ se calculent très
facilement : si $P(X,Y)=P_0(X,Y) + P_1(X,Y) + \ldots + P_d(X,Y)$ où 
$P_i(X,Y)$ est un polynôme dont tous les monômes sont
\textit{exactement} de degré $i$, alors la différentielle de $P(X,Y)$
en l'origine est le polynôme $P_1(X,Y)$ et sa différentielle seconde
est le polynôme $P_2(X,Y)$. Pour calculer les différentielles d'un
polynôme en un point $p$ quelconque du plan, il suffit de  faire un
changement de variables affine pour ramener $p$ à l'origine et
appliquer la recette précédente.
\end{rem}

\begin{exe}
La conique d'équation $X^2+Y^2-1=0$ (voir Figure \ref{coniques}a) n'a
que des points non singuliers. Par contre, 
la conique d'équation $X^2-Y^2=0$ (voir Figure \ref{coniques}c) a un
point double à l'origine. 
\end{exe}

D'après le Lemme de Morse, au voisinage d'un point double, une courbe
algébrique ressemble à la courbe définie par une forme quadratique non
dégénérée. Le corps $\CC$ étant algébriquement clos, toutes les formes
quadratiques 
non dégénérées sur $\CC^2$ sont équivalentes à  $X^2-Y^2$.
En particulier, il existe
un unique modèle local pour un point double d'une courbe algébrique
complexe : la courbe ressemble à l'union des droites d'équation
$X-Y=0$ et $X+Y=0$. 
Ainsi, $p$ est un point
d'intersection de deux branches non singulières de $C$  à tangentes
distinctes. Au voisinage d'un point double, une courbe algébrique
ressemble donc à la Figure \ref{coniques}c.

\begin{defi}\label{def ns}
Une courbe algébrique dont tous les points sont non singuliers est
appelée courbe non singulière.

Une courbe algébrique dont les seuls points singuliers sont des points
doubles est appelée une courbe nodale. 
\end{defi}

La majorité des courbes sont non singulières. Plus précisément, une
courbe $C$ \textit{générique} est non 
singulière.
Le mot \textit{générique}  veut dire ici \textit{en dehors d'un
fermé d'intérieur vide et de mesure nulle}\footnote{Plus précisément, 
il faut prendre $C$ en dehors d'une hypersurface algébrique stricte dans
l'espace des courbes, appelée \textit{hypersurface discriminante}.}. Si l'on
considère maintenant l'espace des courbes nodales, alors une courbe
nodale générique n'a qu'un seul point double. Plus généralement, une
courbe générique dans l'espace des courbes algébriques avec au moins
$n$ points doubles a exactement $n$ points doubles.

Puisque les courbes génériques sont non singulières,  lorsque nous
cherchions à la section \ref{enum simple} les courbes 
algébriques de degré $d$ passant par $\frac{d(d+3)}{2}$ points en
position générique, nous trouvions  toujours une courbe non singulière. 
Fixons maintenant 1 point de moins, c'est-à-dire prenons
$\frac{d(d+3)}{2}-1$ points en 
position générique. Par le même raisonnement qu'à la section \ref{enum
simple}, il existe un espace de dimension 1\footnote{C'est en fait une
droite dans l'espace de courbes.} de courbes
algébriques de degré $d$ passant par ces points. Génériquement, ces
courbes sont non singulières, mais nous pouvons raisonnablement
nous attendre à ce que certaines d'entre elles soient nodales. De plus,
par généricité, ces courbes nodales auront \textit{exactement} un point double.
Si nous fixons maintenant $\frac{d(d+3)}{2}-2$ points en 
position générique, alors nous obtenons un espace de dimension 1 de courbes
nodales passant par ces points, et ces courbes auront génériquement
exactement 1 point double. Encore une fois, nous pouvons nous attendre à ce
que quelques unes de ces courbes aient au moins 2 points doubles, et  
par généricité, ces courbes auront exactement 2 points doubles.

En continuant ainsi, nous voyons que moins nous fixons de points dans $\CC^2$,
 plus nous pouvons imposer de points doubles sur les courbes passant
par ces points. De plus, une fois fixés un certain nombre de points,
on voit que le nombre maximum de points doubles est aussi fixé. Nous
pouvons même espérer trouver un nombre fini de
courbes ayant ce 
nombre maximum de points doubles et passant par les points
fixés. Demandons nous \textit{``Combien?''}, et voilà notre problème
énumératif posé.

Mais avant d'aller plus loin, nous
devons savoir combien de points 
doubles une courbe  algébrique peut avoir. Ce nombre est toujours
fini, et la proposition suivante nous donne le nombre maximum.

\begin{prop}\label{max pt2}
Une courbe algébrique complexe nodale $C$ de degré
$d$ a au plus 
$\frac{d(d-1)}{2}$ points doubles. De plus, si $C$ est irréductible,
alors elle ne peut avoir plus de
 $\frac{(d-1)(d-2)}{2}$ points doubles.
\end{prop}

\begin{exe} 
Une droite est toujours non singulière, donc n'a jamais de point
  double. 
\end{exe}

\begin{exe} 
Prouvons la Proposition \ref{max pt2} dans le cas du degré 2. Nous
devons  montrer  qu'une conique a au maximum un unique
point double et que toute conique nodale est réductible.

Soit $C$ une conique nodale. Alors d'après la Remarque \ref{rem diff},
quitte à faire un changement de
coordonnées affine, nous pouvons supposer qu'un
 point double de $C$ est le point $(0,0)$, et  que $C$ est définie par
 l'équation $X^2-Y^2=0$. La conique $C$ est donc réductible, plus
 précisément est 
 l'union des deux 
 droites d'équation $X-Y=0$ et $X+Y=0$(voir Figure \ref{coniques}c). Comme une droite est non
 singulière, la conique $C$ ne peut pas avoir  d'autre point double. 
\end{exe}

\begin{exe} 
L'union de 3 droites non concourantes et non parallèles est une
cubique avec 3 points doubles. L'union d'une conique et d'une droite
est une cubique avec deux points doubles. Une cubique avec un point
double est donnée par l'équation $Y^2 - X^2(X + 1)=0$, et une cubique
sans point double 
est donnée par l'équation $Y^2 - X(X^2+1)=0$. Toutes ces cubiques sont
représentées à la Figure \ref{cub 2}.
\end{exe}

\begin{figure}[h]
\begin{center}
\begin{tabular}{cccc}
\includegraphics[width=3cm, angle=0]{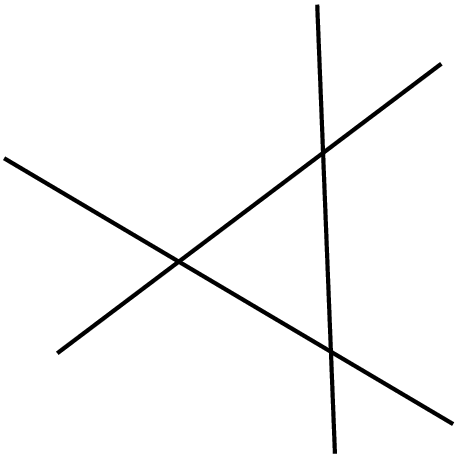}&
\includegraphics[width=3cm, angle=0]{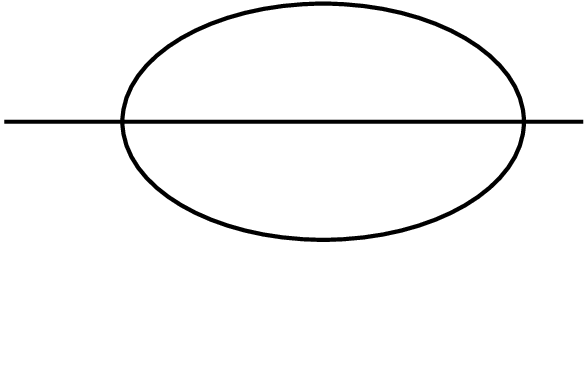}&
\includegraphics[width=2cm, angle=0]{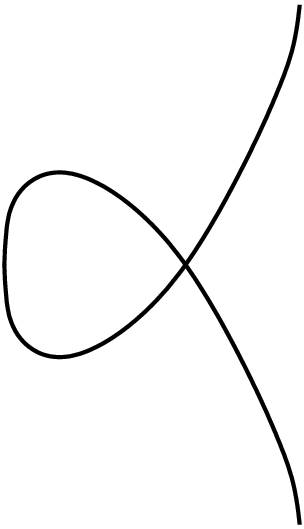}&
\includegraphics[width=1cm, angle=0]{Figures/Cubique1.eps}
\\ a) Union de 3 droites & b) Union d'une droite &
c) $Y^2 - X^2(X + 1)=0$  & d) $Y^2 - X(X^2+1)=0$
\\ & et d'une conique
\end{tabular}
\end{center}
\caption{Cubiques avec 3, 2, 1 et 0 points doubles}
\label{cub 2}
\end{figure}

\begin{exe} 
Plus généralement, il est facile de voir qu'il existe effectivement
une courbe algébrique 
nodale réductible de degré $d$ avec $\frac{d(d-1)}{2}$ points doubles
: il suffit de prendre l'union de $d$ droites dont 3 ne sont 
jamais concourantes.  En perturbant un peu l'équation de ces courbes,
on peut construire des courbes algébriques 
nodales irréductibles de degré $d$ avec $\frac{(d-1)(d-2)}{2}$ points
doubles. Un exemple en degré 4 est représenté sur les Figures
\ref{quart}a et b\footnote{\label{brusot}Voici une explication ``avec les mains''
  du fait que la courbe de la Figure \ref{quart}b est irréductible.
  On part de l'union des 4 droites de la Figure \ref{quart}a, puis on
perturbe un point double. On obtient ainsi l'union d'une courbe
irréductible de degré 2 et de deux droites. Puis, en perturbant un point
d'intersection de la conique et d'une des deux droites, on obtient l'union d'une
courbe  irréductible de degré 3 et d'une droite. Pour finir, on perturbe un
point d'intersection de la cubique et de la droite restante, et on
obtient une courbe de degré 4 irréductible avec $6-3=3$ points
doubles. Pour pouvoir  perturber n'importe
quel point double d'une courbe algébrique tout en préservant d'autres
points doubles éventuels, nous avons utilisé le Théorème de Brusotti qui
  affirme que les \textit{strates} du discriminant correspondant aux
  courbes nodales s'intersectent \textit{transversalement}.}.
\end{exe}
\begin{figure}[h]
\begin{center}
\begin{tabular}{ccccc}
\includegraphics[height=3.5cm, angle=0]{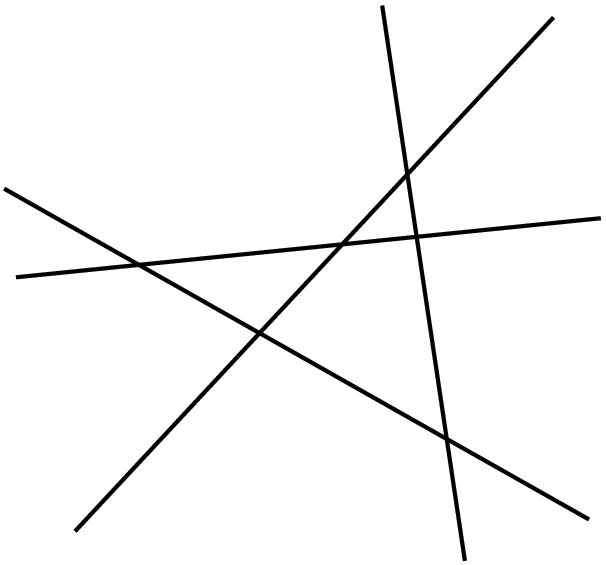}&
 \hspace{4ex} &
\includegraphics[height=3.5cm, angle=0]{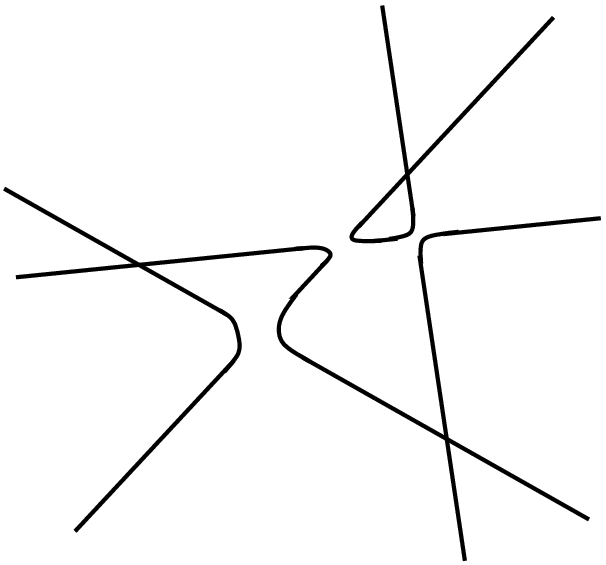}&
 \hspace{4ex} &
\includegraphics[height=3.5cm, angle=0]{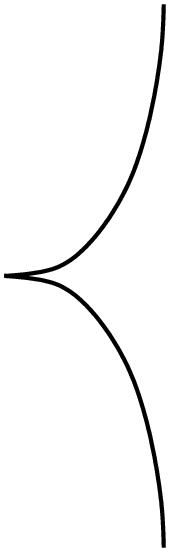}
\\ a) Union de 4 droites && b) Une courbe de degré 4&&
c) Un cusp
\\ &&  avec 3 points doubles
\end{tabular}
\end{center}
\caption{D'autres exemples de courbes algébriques dans $\CC^2$}
\label{quart}
\end{figure}

Précisons qu'il existe des courbes algébriques qui ne sont pas
nodales, c'est-à-dire possédant des points singuliers qui ne sont pas
des points 
doubles. Par exemple, la courbe d'équation $Y^2-X^3=0$ possède un
point singulier en $(0,0)$ qui n'est pas un point double, mais un
$\textit{cusp}$ (voir Figure \ref{quart}c).

Pour des raisons qu'il est difficile d'expliquer dans ce texte, les
géomètres préfèrent parler du nombre de points doubles qu'une courbe
algébrique \textit{n'a pas}, plutôt que du nombre de points doubles qu'elle a. 
Ce nombre est appelé \textit{genre} de la courbe. 

\begin{defi}
Soit $C$ une courbe algébrique nodale irréductible de degré $d$ avec $r$   
points doubles. On définit le genre\footnote{Des points doubles
  peuvent se trouver ``à l'infini'', et en toute rigueur il faudrait
  en tenir compte dans notre définition du genre. Cependant, pour les 
 problèmes énumératifs discutés dans ce texte, 
toutes les
  courbes solutions seront lisses à l'infini.}
 de $C$, noté 
$g(C)$, par
$$g(C)=\frac{(d-1)(d-2)}{2}-r $$
\end{defi}

Le genre d'une courbe nodale irréductible de degré $d$ est donc
toujours un nombre 
compris entre 0 et $\frac{(d-1)(d-2)}{2}$. Les courbes non singulières
correspondent exactement aux courbes de genre $\frac{(d-1)(d-2)}{2}$.
À l'opposé, nous verrons à la section
\ref{courbe rat} que les courbes de genre 0 sont aussi très particulières.

Nous avons déjà vu à la section \ref{enum simple} que  pour pouvoir
poser un problème énumératif, nous devons préalablement savoir
 la dimension de l'espace des courbes qui nous
intéressent. Nous connaissons déjà 
 la dimension de l'espace des
courbes non singulières de degré $d$ qui est
$\frac{d(d+3)}{2}$. D'après la discussion qui suit la Définition
\ref{def ns}, il semble qu'imposer un point double fasse baisser la
dimension de 1. Cela est effectivement vrai, et en remplaçant le nombre de
points doubles par le genre, nous obtenons la proposition
qui suit\footnote{C'est une conséquence du Théorème de Brusotti, voir
  note de bas de page \ref{brusot}.}.

\begin{prop}
L'espace\footnote{C'est une sous variété algébrique irréductible de
  l'espace des courbes algébriques de degré $d$.} des courbes
algébriques nodales irréductibles de 
degré $d$ et 
de genre $g$ est de dimension $3d-1+g$. 
\end{prop}

\subsubsection{Et bien, comptez maintenant!}\label{enum cplx}
Nous voici enfin prêt à poser notre problème énumératif dans toute sa
généralité.
Fixons nous un degré $d\ge 1$, un genre $g\ge 0$, et $\omega=\{p_1,\ldots,
p_{3d-1+g}\}$ une configuration  de $3d-1+g$ points dans
$\CC^2$. Considérons alors l'ensemble 
$\mathcal C (d,g,\omega)$ de toutes les courbes  algébriques
irréductibles nodales de  degré 
$d$,  de genre  $g$,  passant par  tous les points de $\omega$. 

\begin{prop}\label{inv cplx}
Pour une configuration $\omega$ générique, le cardinal de
$\mathcal C (d,g,\omega)$ est fini et ne dépend pas de $\omega$. 
\end{prop}

Essayons d'expliquer grossièrement pourquoi ce cardinal 
est indépendant de $\omega$.
 En mettant ce problème géométrique sous une forme
algébrique, on s'aperçoit que les
courbes de $\mathcal C (d,g,\omega)$  correspondent aux
racines d'un certain polynôme
dans $ \CC[X]$. De plus,  le degré  
de ce polynôme est le même pour toute configuration générique 
$\omega$. Comme $ \CC$
est algébriquement clos, 
 le nombre de solutions est égal à ce degré, c'est-à-dire que le
 nombre de courbes dans $\mathcal C (d,g,\omega)$ est constant.

On pose alors 
$$N(d,g)=\mathcal Card\left( \mathcal C (d,g,\omega)\right).$$

Le nombre $N(d,g)$ est donc le nombre de  courbes  algébriques
irréductibles nodales\footnote{Il existe une notion de genre pour
  n'importe quelle 
  courbe algébrique, non nécessairement nodale, et on peut montrer que
toutes les courbes algébriques
irréductibles de  degré 
$d$,  de genre  $g$ et  passant par $3d-1+g$ point génériques
donnés sont toutes des courbes nodales. Ainsi, on peut oublier de
préciser nodale dans l'énoncé du problème.} de  degré 
$d$,  de genre  $g$,  passant par une configuration générique de
$3d-1+g$ points. 

\begin{exe}
Si $g=\frac{(d-1)(d-2)}{2}$, alors nous calculons le nombre de
courbes non singulières (avec 0 points doubles) passant par
$\frac{d(d+3)}{2}$ points du plan. Nous retombons donc sur le problème
traité à la section \ref{enum simple}:
\end{exe}

\begin{prop}\label{enum ns}
Pour tout $d\ge 1$ on a $N(d, \frac{(d-1)(d-2)}{2})=1$.
\end{prop}

Cependant, calculer les nombres $N(d,g)$ en toute généralité s'avère
être un problème beaucoup plus difficile que le cas
$g=\frac{(d-1)(d-2)}{2})$...  
Nous savons, d'après la ``démonstration'' de 
la Proposition \ref{inv cplx}, que le nombre $N(d,g)$ peut
s'interpréter comme le degré d'un certain polynôme. Mais c'est
seulement en théorie que nous connaissons l'existence de ce polynôme! En 
particulier, nous ne connaissons pas a priori son degré. En fait,
savoir que ce polynôme existe ne nous aide pas beaucoup à
déterminer son degré, et il faudra trouver d'autres moyens pour
calculer les nombres $N(d,g)$.

\begin{pb}\label{pb1}
Comment calculer les nombres $N(d,g)$?
\end{pb}

Sans aucun doute, les origines de la géométrie énumérative remontent à
fort loin. Par exemple, les grecs
savaient certainement déjà que par 5 points passe une unique conique.
 Cependant, à ma connaissance 
le Problème \ref{pb1} n'est posé en toute généralité que
depuis le  XIXème siècle\footnote{Les problèmes énumératifs posés
  alors sont même encore plus généraux! Voici un exemple dont nous ne parlons
  pas ici : combien y a t-il de droites dans l'espace intersectant 4
  autres droites données? La réponse est 2, et il existe une bien jolie
démonstration due à Schubert.} alors  que la géométrie énumérative connaît un
véritable essor  
grâce notamment aux travaux de Chasles, De Joncquieres, Schubert,
Zeuthen ... En 1900, Hilbert propose dans son 15ème problème (voir
\cite{Hil}) de
travailler à des bases rigoureuses du calcul énumératif de Schubert. 
Malgré les progrès de la géométrie énumérative au XIXème siècle, peu
des nombres $N(d,g)$ étaient finalement connus en 1900. Le tableau
\ref{nombre 1} 
résume à peu près ce que l'on  savait à cette époque du Problème
\ref{pb1}. Pour information, le nombre $N(4,0)=620$ a été calculé pour
la première fois 
par Zeuthen  (voir \cite{Zeu1}). 

\begin{table}[!h]
$$\begin{array}{|c|c|c|c|c|}
\hline d\setminus g &\hspace{1.5ex} 0 \hspace{1.5ex} &\hspace{1.5ex}
1 \hspace{1.5ex} &\hspace{1.5ex} 2 \hspace{1.5ex} &\hspace{1.5ex}
3 \hspace{1.5ex} 
\\ \hline 1 & 1 &0  &0 & 0
\\ \hline 2 & 1 & 0 &0  & 0
\\ \hline 3 & 12 & 1 & 0 & 0
\\ \hline 4 & 620 & 225  &27 & 1
\\ \hline
\end{array}
$$
\caption{Premiers  nombres $N(d,g)$}
\label{nombre 1}
\end{table}

En plus, de ces valeurs particulières de $N(d,g)$, la réponse au
Problème \ref{pb1} pour les courbes avec un unique point double était
connue depuis longtemps. Nous donnerons une preuve tropicale de la
formule suivante à
la section \ref{appl 
  cplx}. 

\begin{prop}\label{enum 1n}
Pour tout $d\ge 3$ on a 
$$N(d, \frac{(d-1)(d-2)}{2}-1)=3(d-1)^2$$
\end{prop}

A partir de années 70, les progrès de la géométrie algébrique contribuèrent
à réveiller l'intérêt pour les questions énumératives, et de nouveaux
nombres $N(d,g)$ furent calculés.
 Par exemple, la géométrie énumérative des courbes
avec 2 points doubles était comprise.
Une démonstration possible de la Proposition \ref{enum 2n}
est proposée à la section \ref{exos}.

\begin{prop}\label{enum 2n}
Pour tout $d\ge 4$ on a 
$$N(d,\frac{(d-1)(d-2)}{2}-2)=\frac{3}{2}(d-1)(d-2)(3d^2-3d-11) $$
\end{prop}

Le Lecteur intéressé pourra trouver d'autres formules du même 
tonneau dans \cite{DiFrIt}. 
Dans les années 90, une formule générale calculant tous les nombres
$N(d,g)$ a été donnée par Caporaso et Harris (voir
\cite{CapHar1}). Le problème \ref{pb1} était résolu.

La formule de 
Caporaso et Harris est assez compliquée et nous ne la donnerons pas
ici. En revanche, nous proposons au Lecteur  d'écrire lui même 
 cette formule 
à l'Exercice
\ref{exo CH}. 

Avant l'apparition de la \textit{géométrie tropicale}, les calculs des nombres
$N(d,g)$ étaient, à part quelques exceptions, assez compliqués et
demandaient un bagage technique assez important. La géométrie
  tropicale
a fourni  une nouvelle approche de ce domaine, simplifiant
considérablement les calculs précédents. En effet, les objets
tropicaux ont une nature combinatoire, beaucoup plus simple que les
courbes algébriques!
Le but de ce
texte est de  convaincre le Lecteur dans les sections \ref{trop},
\ref{etage} et \ref{appl} que 
tout un chacun peut calculer 
n'importe quel nombre $N(d,g)$ et établir d'intéressantes relations
entre ces nombres par ses propres moyen s'il dispose de
suffisamment de temps et de
patience.

\subsubsection{Courbes rationnelles}\label{courbe rat}
Les formules \ref{enum ns}, \ref{enum 1n} et \ref{enum 2n} dénombrent
 des courbes ayant peu de points doubles. À l'inverse, les
courbes ayant le nombre maximum de points doubles sont aussi très spéciales
et peuvent être étudiées avec des techniques particulières. Précisons
au passage que
ces courbes intéressent aussi beaucoup les physiciens. Nous
retrouverons les courbes de genre 0 à la section \ref{welschi}, car ce
sont quasiment les seules courbes que l'on sache étudier en profondeur
si nous prenons $\RR$ comme corps de base à la place de $\CC$.

Les courbes de genre 0 sont
aussi appelées \textit{courbes rationnelles}, car elles sont
\textit{paramétrées} par $\CC$.

\begin{thm}\label{rat}
Si $C$ est une courbe algébrique de degré $d$ et de genre 0, alors il
existe trois polynômes  en une variable et de degré $d$, $F(T), G(T)$ et
$H(T)$ tels que $C$ soit l'image de l'application rationnelle
$$\begin{array}{ccc}
\CC\setminus\{\textit{pôles de }\frac{F(t)}{H(t)} \textit{ et }
\frac{G(t)}{H(t)} \}& \longrightarrow &\CC^2 
\\ t&\longmapsto &\left(\frac{F(t)}{H(t)},  \frac{G(t)}{H(t)} \right)
\end{array} $$

\end{thm}

Quelques années avant Caporaso et Harris, Kontsevich a donné une
formule récursive 
calculant tous les nombres $N(d,0)$.

\begin{thm}[Kontsevich \cite{KonMan1}]\label{Kont}
Les nombre $N(d,0)$ sont données par la relation 
$$ N(d,0)=\sum_{\begin{array}{c}d_1+d_2=d\\ d_1,\ d_2\ge
      1\end{array}}  
N(d_1,0)N(d_2,0)\left( d_1^2d_2^2\left(\begin{array}{c}3d-4
  \\ 3d_1-2\end{array}  \right)
-d_1^3d_2 \left(\begin{array}{c}3d-4
  \\ 3d_1-1\end{array}  \right)  \right)$$
si $d\ge 2$, et par la valeur initiale  $N(1,0)=1$.
\end{thm}

Ainsi, tous les nombres $N(d,0)$ se retrouvent à partir du fait que par 2 points
passe une unique droite!
Notons que les approches de Kontsevich, et de Caporaso et Harris sont
différentes. Ainsi, les formules obtenues sont différentes, en ce sens
que la formule de Kontsevich \textit{n'est pas} une sous formule de la formule
de Caporaso et Harris. 
Grâce au Théorème \ref{Kont}, nous pouvons calculer tous les nombres
$N(d,0)$. À peu près un siècle après Zeuthen, Vainsencher (voir
\cite{Vain1}) a calculé 
 $N(5,0)=87304$ pour la première fois. Quelques temps après, les
nombres $N(d,0)$ pour 
$d\ge 6$ sont calculés pour la première fois grâce à la formule de
Kontsevich. 
\begin{table}[!h]
$$\begin{array}{|c|c|c|c|c|c|c|c|c|}
\hline d &\hspace{4ex} 1\hspace{4ex} & \hspace{4ex}
2 \hspace{4ex}& \hspace{4ex}3\hspace{4ex}  &\hspace{4ex}
4 \hspace{4ex}&\hspace{4ex} 5\hspace{4ex}  &\hspace{4ex} 6\hspace{4ex}
&\hspace{4ex}  7\hspace{4ex}
\\ \hline
N(d,0) & 1 &1  & 12&
620
&  87304 &26312976 &14616808192
\\\hline
\end{array}$$
\caption{Premières  valeurs de $N(d,0)$}
\label{tab2}
\end{table}

À la lumière du tableau \ref{tab2}, nous pouvons constater que les
nombres $N(d,0)$ 
ont l'air de grandir assez vite.
De fait nous avons la proposition suivante.
\begin{prop}[Di Francesco - Itszykson \cite{DiFrIt}]\label{asy N}
L'asymptotique de la suite $(\ln N(d,0))_{d\ge 1}$ lorsque $d$ tend
vers l'infini est donnée par
$$\ln N(d,0)\sim 3d\ln d  $$
\end{prop}

\section{Géométrie énumérative réelle}\label{enum reel}
À la section \ref{alg cplx}, nous avons utilisé à deux endroits que
notre corps de base $\CC$ était algébriquement clos. 
Nous avons d'abord invoqué le fait que toutes les formes quadratiques
sur $\CC$ non dégénérées sont équivalentes  
pour écrire l'équation locale d'un point double d'une courbe
algébrique. Puis, pour démontrer que les nombres $N(d,g)$ ne
dépendaient pas de la configuration de points $\omega$ choisie, nous
avons utilisé le fait qu'un polynôme générique de degré $n$ dans
$\CC[X]$ a exactement $n$ 
racines.

Ces deux résultats sont faux sur $\RR$, ce qui  rend la géométrie
  réelle paradoxalement plus complexe que la géométrie complexe. Avant
  de nous lancer 
  dans la géométrie énumérative réelle proprement dite, penchons nous
  d'abord sur 
  les courbes algébriques réelles.

\subsection{Courbes algébriques réelles}\label{reel}

Si $P(X,Y)$ est un polynôme dans $\RR[X,Y]$, alors 
dans l'esprit la section \ref{alg cplx}, nous pouvons regarder
l'ensemble des 
points $(x,y)$ dans 
$\RR^2$ tels que $P(x,y)=0$. Maintenant, le fait que $\RR$ ne soit
pas algébriquement clos change beaucoup de choses. En particulier, la
courbe définie par $P(X,Y)$ peut être vide. Par exemple, aucun
point de $\RR^2$ ne vérifie $X^2+Y^2+1=0$. Nous arrivons ainsi à une petite
contradiction : nous appelons courbe l'ensemble vide...  
Comme souvent, le problème vient de ce que nous ne regardons pas le bon objet.
En effet, il ne faut pas voir une courbe algébrique réelle uniquement
comme l'ensemble des solutions dans $\RR^2$ de l'équation $P(X,Y)=0$,
mais toujours comme l'ensemble des solutions dans $\CC^2$ de
l'équation $P(X,Y)=0$. Que vient donc faire le mot ``réel'' ici
puisque nous considérons toujours les solutions d'une équation dans
$\CC^2$? Le fait que le polynôme $P$ soit réel entraîne que si
$P(x,y)=0$, alors $P(\overline x,\overline y)=0$, où $\overline x$
désigne le nombre complexe conjugué à $x$. Ainsi, notre courbe
algébrique $C$ définie par $P(X,Y)$ est beaucoup plus ``spéciale'' qu'une
courbe algébrique complexe quelconque, puisqu'elle admet une
involution\footnote{Plus précisément 
  une involution antiholomorphe.} 
$$\begin{array}{ccc}
C&\longrightarrow& C
\\(x,y) &\longmapsto& (\overline x,\overline y)
\end{array} $$
De plus, les zéros de $P(X,Y)$ dans $\RR^2$, c'est-à-dire les points réels de
$C$, ne sont autres que les points fixes de cette involution.
Ceci motive la définition suivante.

\begin{defi}
Soit $P(X,Y)$ un polynôme dans $\RR[X,Y]$ de degré au moins 1. Alors
l'ensemble $C$ des 
solutions de $P(X,Y)=0$ dans $\CC^2$ est appelé une courbe algébrique
réelle. De plus, l'ensemble $\RR C= C\cap \RR^2$ est appelé partie
réelle de $C$. 
\end{defi}

\begin{exe}
Une droite dont l'équation est à coefficients réels est un courbe
algébrique réelle, et sa partie réelle est une droite telle que nous
les
 dessinons depuis tout petits.
\end{exe}

À la section \ref{alg cplx}, nous trichions un peu en dessinant les
courbes algébriques dans $\CC^2$. Comme il est peu aisé de représenter
$\CC^2\simeq\RR^4$, nous considérions en fait des courbes algébriques réelles,
et nous dessinions leur partie réelle. Mais ici, puisque nous parlons
justement de courbes algébriques réelles, les dessins seront fidèles à
la réalité\footnote{Pour autant qu'une telle phrase ait un sens...}.
Les courbes dessinées dans les
Figures \ref{coniques}, \ref{cub}, \ref{coniques 2}, \ref{cub 2} et
\ref{quart} sont donc des 
exemples de parties réelles de courbes algébriques réelles.

\begin{exe}
La Figure \ref{crb reel} contient d'autres exemples de partie réelle
de courbes algébriques réelles de degré 2, 3 et 4. Les équation des
courbes représentée sur les Figures  \ref{crb reel}c et d sont un
peu longues, nous ne les écrivons pas. 
\end{exe}
\begin{figure}[h]
\begin{center}
\begin{tabular}{ccccc}
&
\includegraphics[height=3cm, angle=0]{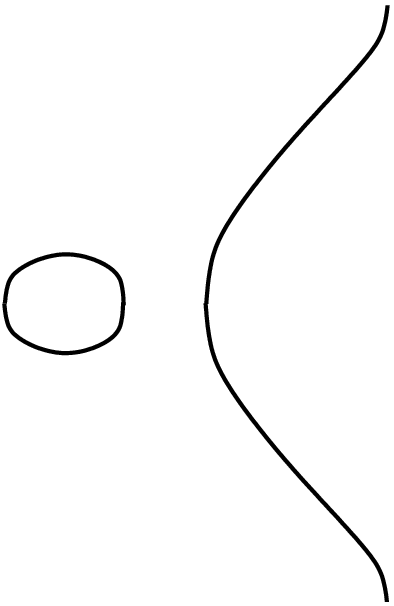}&
\includegraphics[height=3cm, angle=0]{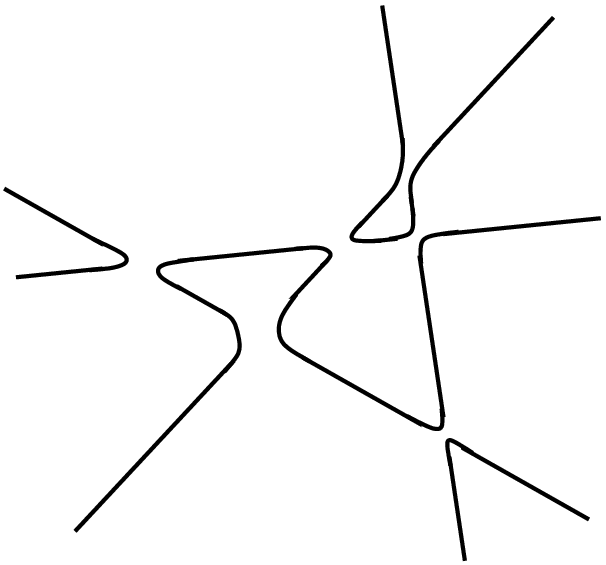}&
\hspace{0ex}&
\includegraphics[height=3cm, angle=0]{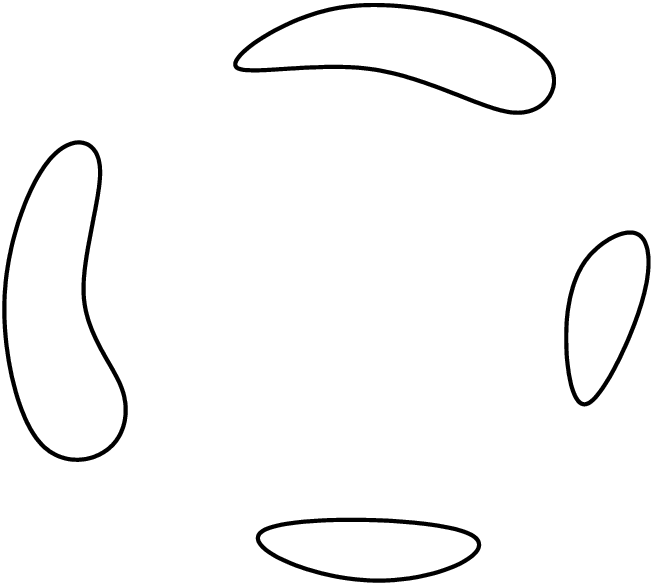}
\\ a) $X^2+Y^2+1=0$ & b) $Y^2-X(X^2-1)$ &
c) Une courbe de degré 4 && d) Une autre courbe 
\\ &&&&de degré 4
\end{tabular}
\end{center}
\caption{Courbes algébriques réelles de degré 2,3 et 4}
\label{crb reel}
\end{figure}

Pourquoi considérer tous les points dans $\CC^2$ d'une courbe
algébrique réelle au lieu de se contenter des points dans $\mathbb
R^2$? 
 Nous avons vu plus haut que considérer uniquement la partie réelle implique
 d'étudier des objets potentiellement vides. Mais il y a plus grave!
Certaines informations concernant  
 une courbe algébrique réelle peuvent être cachées dans ses points
 complexes. Par exemple, les points doubles d'une algébrique réelle
 nodale n'ont aucune raison d'être tous réels. Ainsi, si nous voulons
 définir la bonne notion de genre d'une courbe algébrique réelle, nous
 devons prendre en compte \textit{tous} les points doubles de la
 courbe, pas uniquement ceux qui se trouvent dans $\RR^2$.

Puisqu'une courbe algébrique réelle est avant tout une courbe
algébrique complexe, nous pouvons parler de courbes algébriques
réelles irréductibles, singulières ...

\begin{defi}
Une  courbe algébrique réelle $C$ est irréductible (respectivement
non singulière, nodale) si 
$C$ est irréductible (respectivement 
non singulière, nodale) en tant que courbe algébrique complexe.

Le genre d'une courbe algébrique réelle irréductible nodale $C$ est celui de $C$
vue comme courbe algébrique complexe. 
\end{defi}
\begin{rem}
Une courbe algébrique réelle définie par un polynôme $P(X,Y)$ dans
$\RR[X,Y]$ est 
irréductible si le polynôme $P(X,Y)$ est
irréductible dans $\CC[X,Y]$. 
\end{rem}

\begin{exe}
Toutes les courbes algébriques réelles représentées à la Figure
\ref{crb reel}
sont irréductibles et non singulières.
\end{exe}

\begin{exe}
Terminons de passer en revue toutes les coniques réelles possibles à
changement de variables affine de $\RR^2$ près. Les formes quadratiques 
$X^2+Y^2$ et $X^2-Y^2$ n'étant pas équivalentes sur $\RR$, il y a plus
de coniques réelles que de coniques complexes. Pour avoir toutes les
coniques réelles en plus de celles dessinées sur les Figures
\ref{coniques} et \ref{crb reel}a, nous devons ajouter les coniques
d'équations $X^2-Y^2 +1 =0$, $X^2+Y^2=0$ et $X^2+1=0$. Ces trois coniques
sont représentées à la Figure \ref{con reel}. Notons que les
polynômes $X^2+Y^2$ et $X^2+1$ sont irréductibles dans $\RR[X,Y]$, mais pas
dans $\CC[X,Y]$ et définissent donc des courbes algébriques réelles réductibles.
\end{exe}

\begin{figure}[h]
\begin{center}
\begin{tabular}{ccccc}
\includegraphics[height=3cm, angle=0]{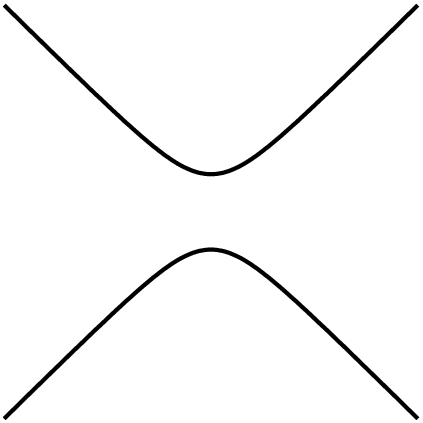}&
\hspace{5ex}  &
\includegraphics[height=3cm, angle=0]{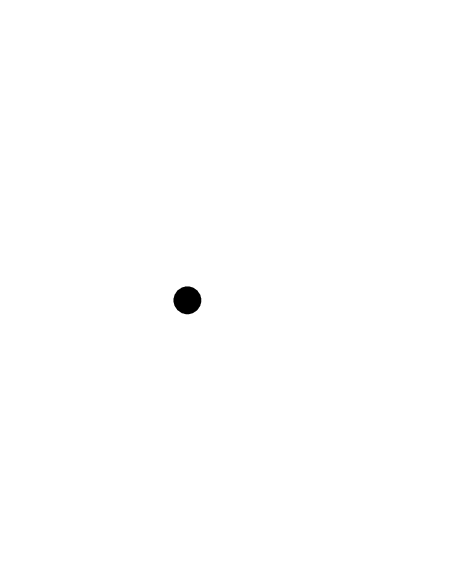}&
\hspace{5ex}  &

\\ a) $X^2-Y^2+1=0$ && b) $X^2+Y^2=0$ &&
c)  $X^2+1=0$
\end{tabular}
\end{center}
\caption{Fin de la classification des coniques dans $\RR^2$}
\label{con reel}
\end{figure}

En regardant la classification des coniques réelles, nous remarquons un
nouveau phénomène : il  y a deux types de points doubles réels
possibles pour une courbe algébrique réelle! En effet, une conique
réelle nodale est soit l'union de deux droites réelles, soit l'union
de deux droites complexes conjuguées. 
Encore une fois, cela est
dû au fait qu'à changement de variables
affine de $\RR^2$ et à l'application $P(X,Y)\to -P(X,Y)$ près, il y a
deux formes 
quadratiques non dégénérées distinctes  : $X^2+Y^2$ et
$X^2-Y^2$. Comme les formes 
quadratiques sont les équations locale d'un point double, nous voyons alors
qu'il existe effectivement deux sortes de points doubles réels pour une courbe
algébrique réelle. Nous aurons l'occasion de reparler de ces deux
types de points doubles réels à la section \ref{welschi}.

\begin{defi}
Soit $C$ une courbe algébrique réelle, et $p\in \RR C$ un point double
de $C$. On dit que $p$ est un point double réel isolé de $\RR C$ si la
différentielle seconde de $C$ en $p$ est équivalente sur $\RR$ à
$X^2+Y^2$. On dit que $p$ est un point double réel non isolé de $\RR C$ si la
différentielle seconde de $C$ en $p$ est équivalente sur $\RR$ à
$X^2-Y^2$. 
\end{defi}

Si $p$ est un point double réel isolé de $\RR C$, alors au voisinage de
$p$, $C$ ressemble à la courbe définie par l'équation $X^2+Y^2=0$. En
particulier, $p$ est le seul point de $\RR C$ dans ce voisinage, d'où
le nom de point double isolé. De plus, comme $X^2+Y^2=(X+iY)(X-iY)$,
au voisinage de 
$p$ la courbe $C$ ressemble à l'union de deux droites complexes conjuguées,
c'est-à-dire que $p$ est le point d'intersection de 
deux branches complexes conjuguées de $C$. Cette situation est
représentée à la Figure \ref{pt 2 reel}a où les pointillés
représentent les 
deux branches complexes conjuguées de $C$.

Au voisinage d'un point double réel non isolé, la courbe $C$ ressemble  à
la courbe définie par l'équation $X^2-Y^2=0$, et  donc à
l'union de deux droites réelles (voir la Figure \ref{pt 2 reel}b).

\begin{figure}[h]
\begin{center}
\begin{tabular}{ccccc}
\includegraphics[height=2cm, angle=0]{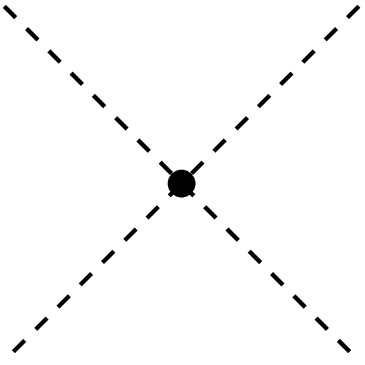}&
\hspace{4ex}  &
\includegraphics[height=2cm, angle=0]{Figures/Conique3.eps}&
\hspace{4ex}  &
\includegraphics[height=3cm, angle=0]{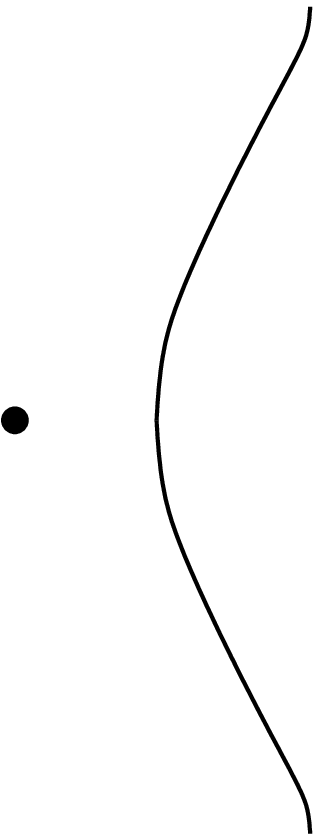}

\\ a) Un point double isolé && b) Un point double non isolé &&
c) $Y^2 - X^2(X-1)=0$
\end{tabular}
\end{center}
\caption{Points doubles réels}
\label{pt 2 reel}
\end{figure}

\begin{exe}
La cubique d'équation $Y^2 -X^2(X+1)=0$ (voir Figure \ref{cub 2}c) a un
point double réel non isolé en $(0,0)$. La cubique d'équation   $Y^2 -
X^2(X-1)=0$ (voir Figure \ref{pt 2 reel}c) a un  
point double réel isolé en $(0,0)$.
\end{exe}

\subsection{Le réel, c'est compliqué}
Maintenant que nous avons  défini les courbes algébriques
réelles nodales irréductibles, nous pouvons jouer au même jeu qu'à la
section \ref{enum cplx}. 
Fixons nous un degré $d\ge 1$, un genre $g\ge 0$, et
une configuration générique\footnote{\emph{générique} comme configuration de points dans
$\CC^2$. De telles configurations de points réels existent. En effet, il
existe un polynôme $f$ à coefficients complexes non tous nuls tel que les
coordonnées des points des configurations $\omega$ non génériques satisfont à
$f=0$. Lorsque ces coordonnées sont toutes réelles, elles satisfont aux deux
équations réelles obtenues en remplaçant les coefficients de $f$ par leur
partie réelle, resp. imaginaire. L'une des deux est nécessairement non
triviale, donc les configurations de points réels ne peuvent pas toutes être
non génériques.}
 $\omega=\{p_1,\ldots, p_{3d-1+g}\}$ de
$3d-1+g$ points dans 
$\RR^2$. Considérons 
alors l'ensemble 
$\RR \mathcal C (d,g,\omega)$ de toutes les courbes  algébriques  réelles
irréductibles 
nodales de  degré 
$d$,  de genre  $g$,  passant par  tous les points de $\omega$. Comme
$\RR^2 \subset \CC^2$, il est clair que  $\RR \mathcal C
(d,g,\omega)\subset \mathcal 
C (d,g,\omega)$, où $\mathcal C (d,g,\omega)$ est l'ensemble des courbes
algébriques \textit{complexes} de degré $d$ 
et de genre $g$ passant par les points de $\omega$.

Nous avons  vu à la section \ref{enum cplx} que pour une configuration
de points
$\omega$ dans 
$\CC^2$, les courbes de $\mathcal C (d,g,\omega)$
sont données par les racines d'un polynôme dans $\CC[X]$, dont le
degré, et donc le nombre de racines, ne dépend pas du choix de la
configuration générique $\omega$. Si les points 
$p_i$ sont dans $\RR^2$, alors ce polynôme est dans $\RR[X]$. Maintenant,
le nombre de racines réelles d'un polynôme réel ne
dépend pas uniquement du degré du polynôme, mais aussi de ses
coefficients. C'est-à-dire que la Proposition \ref{inv cplx} a peu de
chance d'être vraie
 dans le cas réel : le cardinal de $\RR \mathcal C (d,g,\omega)$
dépend de la configuration 
$\omega$ choisie! Par exemple, on peut facilement
trouver une configuration de 8 points dans $\RR^2$ par laquelle passent
12 cubiques réelles rationnelles, et une autre par laquelle passent
seulement 8 cubiques réelles rationnelles (voir Proposition \ref{W3}). 

Les racines d'un polynôme dans $\RR[X]$ sont soient réelles, soit
naturellement associées en paires de racines complexes
conjuguées. Puisque les solutions d'un problème énumératif réel sont
les racines d'un polynôme réel, les solutions sont soit des courbes
réelles, soit 
naturellement associées en paires de courbes complexes conjuguées. 
Reprenons l'exemple des 8 points de $\RR^2$ par lesquels passent
seulement 8 cubiques réelles rationnelles. Nous savons que 
 par  ces 8 points
 passent
12 cubiques complexes rationnelles, donc 4 de ces cubiques complexes
rationnelles ne sont pas réelles et sont regroupées en
 deux paires de cubiques complexes conjuguées.

La géométrie énumérative réelle semble donc plus compliquée que la
géométrie énumérative complexe... Il y a cependant un cas où la
méthode employée pour compter les courbes complexes marche tout aussi
bien pour
compter les courbes réelles.

\begin{prop}\label{lisse reel}
Pour tous $d\ge 1$ et $\omega$ générique, on a $\mathcal Card\left( \RR \mathcal C(d,
\frac{(d-1)(d-2)}{2},\omega) \right)=1$. 
\end{prop}
\begin{proof}
Nous pouvons reprendre mot pour mot la démonstration de la section
\ref{enum simple} en remplaçant $\CC$ par $\RR$. En effet, ce problème
énumératif 
se traduit en un problème d'algèbre linéaire dont la résolution ne
dépend pas du corps de base. 
\end{proof}

Mais  pour l'énumération générale des courbes algébriques
réelles nodales, savoir 
compter les courbes complexes ne nous aide pas beaucoup...

\begin{pb}\label{pb reel}
Que peut-on dire du cardinal de $\RR \mathcal C (d,g,\omega)$?
\end{pb}

Puisque ce cardinal est le nombre de racines réelles d'un polynôme
dans $\RR[X]$, on peut immédiatement dire deux choses : ce nombre est plus petit
que le  degré du polynôme et lui est
 congru modulo 2

\begin{prop}\label{encadr triv}
Pour toute configuration de points $\omega$, on a les inégalités
suivantes
$$0,1 \le \mathcal Card\left(\RR \mathcal C (d,g,\omega) \right) \le N(d,g) $$
la borne inférieure étant 1 (resp. 0) si $N(d,g)$ est impair
(resp. pair).
De plus on a 
$$\mathcal Card\left(\RR \mathcal C (d,g,\omega) \right)=N(d,g)\ mod\  2$$
\end{prop}

Maintenant, quelles sont les valeurs possibles du cardinal de $\RR \mathcal C
(d,g,\omega)$ entre 0 et $N(d,g)$? Si tous les polynômes de $\RR[X]$
représentaient toutes les solutions de nos problèmes énumératifs,
alors la réponse 
serait évidente : le cardinal de $\RR \mathcal C
(d,g,\omega)$ pourrait être n'importe quel nombre plus petit et de même
parité que $N(d,g)$. 

Cependant, nos polynômes à une variable proviennent de problèmes
géométriques, et ne sont donc a priori pas du tout quelconques. Nous
avons donc bien 
peu d'espoir de prouver l'énoncé ci-dessus. Et pour cause, la
Proposition \ref{W3} ci-dessous nous
montre que cet énoncé est faux! 

À l'heure actuelle, peu de choses sont connues dans l'étude du Problème
\ref{pb reel} comparé à ce que l'on sait du Problème \ref{pb1}. 
 Si l'on connaît plusieurs moyens de calculer tous les 
nombres $N(d,g)$, on ne sait à peu près rien des valeurs
possibles du cardinal de $\RR \mathcal C
(d,g,\omega)$ à part dans quelques  cas particuliers. En fait, on ne sait
répondre en toute généralité à cette question qu'en degré plus petit
que 3... Le cas des degrés 1 et 2 est couvert par la Proposition
\ref{lisse reel}, ainsi que le cas des cubiques réelles non singulières de
degré 3. Le cas des cubiques réelles rationnelles a été traité par
Degtyarev et Kharlamov il y a seulement quelques années.

\begin{prop}[Degtyarev - Kharlamov \cite{DK}]\label{W3}
Les valeurs prises par le cardinal de $\RR \mathcal C(3,0,\omega)$ quand $\omega$
parcourt les configurations génériques de 8 points dans $\RR^2$ sont
$8,10$ et $12$.
\end{prop}

La Proposition \ref{W3} fait apparaître le  phénomène 
diablement intéressant suivant : quel
que soit le choix de nos 8 points dans $\RR^2$, il existe toujours au
moins 8 cubiques rationnelles réelles passant par ces 8 points! Ainsi,
la borne inférieure 0 de la Proposition \ref{encadr triv} est loin
d'être optimale. Par contre, on constate aussi qu'il existe des
configurations de 8 points pour lesquelles les 12 cubiques rationnelles
complexe sont en faite réelles. Est ce là un fait général? Peut on
améliorer les deux bornes de la Proposition \ref{encadr triv}? Les
problèmes suivants sont  des versions simplifiées du problème \ref{pb reel}

\begin{pb}\label{min reel}
Quelle est la valeur minimale du cardinal de $\RR \mathcal C(d,g,\omega)$ 
quand $\omega$ varie?
\end{pb}

\begin{pb}\label{max reel}
Quelle est la valeur maximale du cardinal de $\RR \mathcal C(d,g,\omega)$ 
quand $\omega$ varie?
\end{pb}

Pour l'instant, ces deux problèmes sont encore largement ouverts. Par
exemple, personne à ma connaissance ne sait s'il existe une
configuration de 11 points dans $\RR^2$ telle que les 620 courbes
algébriques complexes rationnelles de degré 4 passant par ces points
soient toutes réelles.
Vers 2002, Welschinger a proposé une recette pour associer un signe
$\pm 1$ à toute courbe
algébrique 
réelle 
 et a montré un résultat surprenant : si on
compte les courbes rationnelles de cette manière, alors le résultat ne
dépend pas de la configuration $\omega$ choisie. Encore plus intéressant, cet
invariant donne une nouvelle borne inférieure pour le cardinal de $\RR \mathcal
C(d,0,\omega)$!

\subsection{Invariants de Welschinger}\label{welschi}

Nous avons déjà mentionné à la section \ref{courbe rat} que les
courbes rationnelles sont très spéciales parmi toutes les courbes
algébriques, et que l'on disposait de  techniques
particulières pour les étudier. Nous allons ici expliquer comment compter
``correctement'' les courbes réelles rationnelles pour que le résultat ne
dépende pas de la configuration $\omega$ choisie. Nous obtiendrons
ainsi les \textit{invariants de Welschinger}.

Nous avons vu à la section \ref{reel} qu'il
existe deux types  de points doubles réels pour une
courbe algébrique réelle, les points doubles isolés et les points
doubles non isolés. Un point double isolé est le point d'intersection de deux
branches complexes conjuguées de la courbe. Un point double non isolé
est le point 
d'intersection de deux branches réelles. 
\begin{defi}
La masse d'une courbe algébrique
réelle $C$, notée $m(C)$, est le nombre de points doubles isolés de $C$.
\end{defi}

\begin{exe}
La cubique d'équation $Y^2-X^2(X+1)=0$ (voir Figure \ref{cub 2}c) est
de masse 0, et la cubique 
d'équation $Y^2-X^2(X+1)=0$ (voir Figure \ref{pt 2 reel}c) est de masse 1. 
\end{exe}

Soit $\omega$ une configuration générique de $3d-1$ points dans
$\RR^2$ et posons
$$W(d)= \sum_{C\in \RR \mathcal C(d,0,\omega)}
(-1)^{m(C)}$$

Le résultat suivant est un des théorèmes majeurs en géométrie
énumérative réelle de ces dernières années.

\begin{thm}[Welschinger \cite{Wel1}]
Pour tout $d\ge 1$, le nombre $W(d)$ ne dépend pas de $\omega$.
\end{thm}

Pour nous, l'intérêt principal  de l'invariant $W(d)$ est qu'il donne une borne
inférieure pour le cardinal de $\RR \mathcal C(d,0,\omega)$ potentiellement
meilleure que celle de la Proposition \ref{encadr triv}.

\begin{prop}
Pour tout $d\ge 1$ et toute configuration $\omega$ générique de $3d-1$
points dans $\RR^2$, on a
$$|W(d)|\le \mathcal Card\left(\RR \mathcal C(d,0,\omega)\right) $$ 
\end{prop}
\begin{proof}
La proposition découle immédiatement du fait que les courbes de $\RR
\mathcal C(d,0,\omega)$ sont comptées avec un signe $\pm 1$.
\end{proof}

Donc si nous arrivons à montrer que les nombres $W(d)$ ne sont pas
nuls, nous aurons accompli un progrès dans l'étude du Problème
\ref{min reel}! On voit facilement avec la Formule de Kontsevich
(Théorème \ref{Kont}) que les nombres $N(d,0)$
sont pairs dès que $d\ge 3$, de sorte que nous ne pouvons même pas affirmer
immédiatement que par $3d-1$ point génériques de $\RR^2$ passe au
moins une courbe algébrique réelle rationnelle de degré $d$.

\begin{pb}
Les nombres $W(d)$ sont ils nuls?
\end{pb}

Welschinger a montré que les nombres $W(d)$ ne dépendent pas de la
configuration $\omega$ choisie, mais il restait à les calculer. Dans
son article original, Welschinger a seulement pu démontrer que $W(d)$
était non nul pour $d\le 5$.
 Ce n'est que grâce à la géométrie
tropicale que les
invariants de Welschinger ont pu être calculés pour la première fois
par Mikhalkin dans \cite{Mik1}. En se basant sur les travaux de Mikhalkin, 
 Itenberg, Kharlamov et Shustin (voir \cite{IKS1}) ont  donné une
 borne inférieur non triviale pour les nombres $W(d)$ et
 ont ainsi établi l'existence d'une courbe
algébrique réelle rationnelle de degré $d$ passant par n'importe quelle
configuration de $3d-1$ points réels. Un peu plus tard, ils ont montré
que les invariants de Welschinger satisfont une formule de type
Caporaso-Harris (voir \cite{IKS3}). Encore une fois, cette formule est
assez compliquée, nous ne la donnerons pas ici. Nous proposons
toutefois au Lecteur d'écrire lui même cette formule à l'Exercice
\ref{exo CH reel}.

Les premières valeurs de $W(d)$ sont données dans le tableau
\ref{W}. Notons que pour démontrer la Proposition \ref{W3}, Degtyarev et
Kharlamov avaient en fait calculé $W(3)$, et ce avant que
Welschinger ait défini ses invariants. 
Les valeurs de $W(d)$ pour $d\ge 4$ sont dues
à Itenberg, Kharlamov et Shustin. Ces derniers ont aussi étudié
quelques propriétés de 
la suite  $(W(d))_{d\ge 1}$. Nous verrons ces
résultats plus en détail à la section \ref{appl reel}.

La Proposition \ref{W3} nous dit que $W(3)$ est la borne inférieure
optimale pour $Card\left(\RR \mathcal C(3,0,\omega)\right) $.
En utilisant  les diagrammes en
étages (voir section \ref{etage}), Rey a montré que l'invariant de
Welschinger est aussi optimal en degré 4.

\begin{table}[!h]
$$\begin{array}{|c|c|c|c|c|c|c|c|c|}
\hline d &\hspace{4ex} 1\hspace{4ex} & \hspace{4ex}
2 \hspace{4ex}& \hspace{4ex}3\hspace{4ex}  &\hspace{4ex}
4 \hspace{4ex}&\hspace{4ex} 5\hspace{4ex}  &\hspace{4ex} 6\hspace{4ex}
&\hspace{4ex}  7\hspace{4ex}
\\ \hline
W(d) & 1 &1  & 8&
240
&  18 264 &2845440 & 792731520
\\\hline
\end{array}$$
\caption{Premières  valeurs de $W(d)$}
\label{W}
\end{table}

\begin{prop}[Rey]
Il existe une configuration générique de 11 points dans $\RR^2$  par
laquelle passent
exactement $240$ courbes algébriques réelles rationnelles de degré 4.
\end{prop}

À partir du degré 5, le problème de l'optimalité de $W(d)$ est toujours ouvert.

\begin{pb}
Existe-t-il toujours une configuration générique $\omega$ de $3d-1$
points dans $\RR^2$ telle que le cardinal de $\RR \mathcal C(d,0,\omega)$ soit exactement $W(d)$? 
\end{pb}

\begin{rem}
Rien ne nous empêche de compter les courbes de  $\RR \mathcal
C(d,g,\omega)$ pour tout genre $g$ avec les mêmes signes que pour les
courbes rationnelles. Malheureusement, les nombres obtenus ne sont pas
invariants 
par rapport à $\omega$ dès que $g\ge 1$ (voir \cite{Wel1} ou \cite{IKS1})...
\end{rem}

\section{Géométrie énumérative  tropicale}\label{trop}
La géométrie tropicale est une géométrie algébrique comme aux sections
\ref{alg cplx} et \ref{enum reel}, mais au lieu de se fixer un corps
de base, on se 
fixe à la place un \textit{semi-corps} (c'est-à-dire que nous ne
disposons pas de la soustraction) de base. Les objets tropicaux que l'on
obtient ont alors une nature combinatoire, ce qui facilite la
résolution de certains problèmes. Par exemple, la géométrie
énumérative tropicale est beaucoup plus simple que ses homologues
réelle ou complexe! La suite de ce texte est destinée à convaincre le
Lecteur que la géométrie
énumérative tropicale, c'est facile.

L'intérêt majeur de la géométrie tropicale est qu'elle permet de
résoudre des problèmes 
de géométrie classique. Les Théorèmes de Correspondance de Mikhalkin
que nous verrons à la section \ref{corr} en sont de beaux exemples :
compter des courbes tropicales revient à compter des courbes algébriques.

\subsection{Courbes tropicales}
Par soucis de concision, nous n'adopterons pas ici le point de vue
\textit{géométrie algébrique sur le semi-corps tropical}.
Nous nous bornerons à donner les définitions ad hoc dont nous aurons
besoin dans la suite 
de ce texte.  
Pour nous,  une courbe tropicale est un graphe équilibré dans $\RR^2$
dont les arêtes sont des intervalles à pente rationnelle et auxquelles
nous attribuons un nombre entier appelé \textit{poids}.

Il est possible d'unifier les concepts de cette section avec ceux de la
section \ref{alg cplx} où les courbes algébriques sont définies
comme le lieu d'annulation de polynômes. Il existe une notion de
polynôme
tropical et de son lieu d'annulation, et on peut ainsi
 définir une courbe tropicale comme à la section \ref{alg cplx}. 
Cette définition ``algébrique''  de courbe tropicale est
équivalente à celle que nous donnons ici,
nous renvoyons le Lecteur au cours d'Ilia Itenberg dans \cite{BIT} ou
aux textes \cite{Mik3} et \cite{Mik9}.

\begin{defi}\label{rect}
Un graphe rectiligne à pentes rationnelles $\Gamma$ dans $\RR^2$ est la donnée d'un ensemble
fini de points $\Gamma_0$ de $\RR^2$, et d'un ensemble fini $\Gamma_1$
d'intervalles (non nécessairement bornés) de $\RR^2$  tels que
\begin{itemize}
\item[$\bullet$] chaque intervalle de $\Gamma_1$ est contenu dans
  une droite d'équation $aX+bY+c=0$ avec $a$ et $b$ dans $\ZZ$, 

\item[$\bullet$] les extrémités de chaque intervalle de $\Gamma_1$
  sont dans $\Gamma_0$,
\item[$\bullet$] l'intersection de deux éléments de $\Gamma_1$ est soit
  vide, soit un point de $\Gamma_0$,

\item[$\bullet$] tout point dans $\Gamma_0$ est une extrémité d'au
  moins 3 intervalles de $\Gamma_1$.
\end{itemize}

Les points de $\Gamma_0$ sont appelés les sommets de $\Gamma$,  les
intervalles de $\Gamma_1$ sont appelés les arêtes de $\Gamma$. Si un
sommet $s$ est une extrémité d'une arête $a$, on dit que le sommet
$s$ est adjacent à l'arête $a$. Un sommet adjacent à 3 (resp. 4)
arêtes est appelé un sommet trivalent (resp. quadrivalent).

\end{defi}

\begin{exe}
Des exemples de graphes rectilignes à pentes
rationnelles sont représentés sur la Figure \ref{graph}. Le graphe de
la Figure \ref{graph}a (resp. \ref{graph}b, c et d) a 5 (resp. 3,
1 et 4)
sommets et 8 (resp. 10, 3 et 9) arêtes. 
\end{exe}

\begin{figure}[h]
\begin{center}
\begin{tabular}{ccc}
\includegraphics[height=3cm, angle=0]{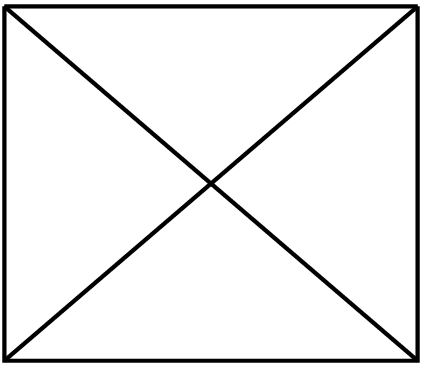}&
\hspace{10ex}  &
\includegraphics[height=3cm, angle=0]{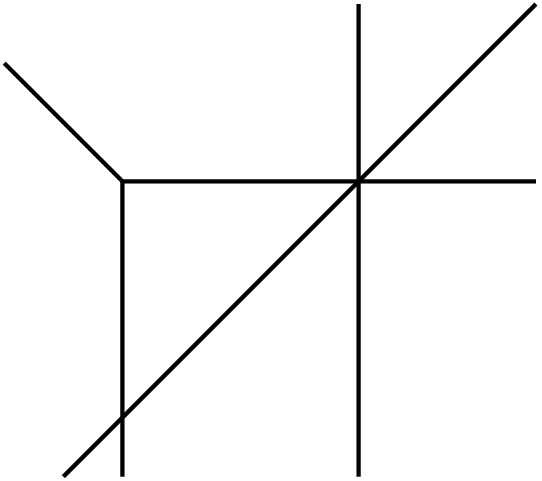}
\\ a) && b)

\\
\\\includegraphics[height=4cm, angle=0]{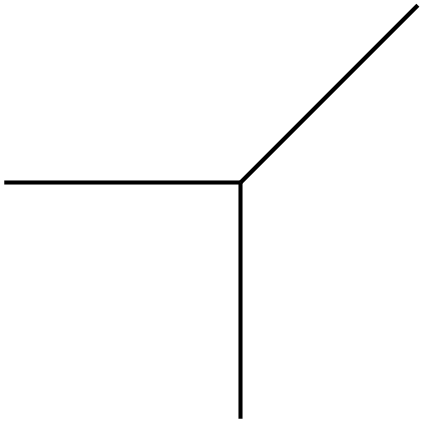}&
\hspace{10ex}  &
\includegraphics[height=4cm, angle=0]{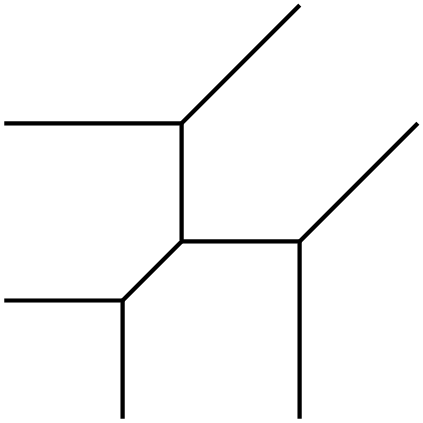}

\\c)&&d)
\end{tabular}
\end{center}
\caption{Graphes rectilignes à pentes
rationnelles}
\label{graph}
\end{figure}

Puisque les sommets d'un graphe rectiligne à pentes rationnelles $
\Gamma$ dans $\RR^2$
sont tous adjacents à au moins trois arêtes et que deux
arêtes ne peuvent s'intersecter qu'en un sommet, l'ensemble
$\Gamma_0$ est entièrement déterminé par l'union dans $\RR^2$ des arêtes
de $\Gamma$.  
Pour éviter les détails trop techniques et peu intéressants, et bien
que cela ne soit pas tout à fait honnête, nous confondrons dans la
suite un graphe rectiligne à pentes rationnelles et
l'union de ses arêtes   dans $\RR^2$.

Soit $s$ un sommet adjacent à une arête $a$. Comme l'arête $a$ est
inclue dans une droite $L$ dont la direction est d'équation à
coefficients dans $\ZZ$, il existe 
un vecteur $\vec v$ dans $\ZZ^2$ tel que $s+\vec v$ soit dans $L$. De
plus, si nous demandons que les coordonnées de $\vec v$ soient premières
entre elles, il n'y a plus que deux choix possibles pour $\vec v$ 
suivant sa direction. Nous pouvons fixer cette direction en demandant 
que le point $s+\epsilon \vec v$ soit dans $a$ pour $\epsilon$ un
nombre réel positif assez
petit. Nous avons ainsi défini de manière unique un vecteur $\vec v$ dans
$\ZZ$ à partir de $s$ et de $a$.

\begin{defi}
Le vecteur $\vec v$ est appelé le vecteur primitif sortant du sommet $s$
suivant l'arête $a$.
\end{defi}

\begin{exe}
Si le sommet $s$ du graphe de la Figure \ref{graph}c est le points $(0,0)$,
alors les trois arêtes sont contenues dans les trois droites
d'équations $Y=0$, $X=0$ et $X-Y=0$. Les trois vecteurs primitifs
sortant de $s$ sont donc les vecteurs $(-1,0)$, $(0,-1)$ et $(1,1)$.
\end{exe}

Pour définir une courbe tropicale, il nous faut introduire la notion
de \textit{poids} sur les arêtes.

\begin{defi}
Un graphe pondéré dans $\RR^2$ est un graphe
rectiligne à pentes rationnelles $\Gamma$ muni d'une fonction
  $w : \Gamma_1\to \NN^*$.
\end{defi}

Dans les dessins de graphes pondérés, on ne précise le poids d'une
arête uniquement si celui ci est au moins 2.

\begin{exe}
Les quatre graphes de la Figure \ref{graph} sont donc des graphes
pondérés, chaque arête étant implicitement de poids 1. Les graphes de
la Figure \ref{courbe trop} sont aussi des graphes pondérés, le graphe
pondéré de la Figure 
\ref{courbe trop}b contient une arête de poids 2.
\end{exe}

Nous avons maintenant tout ce qu'il faut pour définir une courbe
tropicale, qui est un graphe  pondéré
et \textit{équilibré}.

\begin{defi}
Une courbe tropicale $C$ dans $\RR^2$ est un graphe rectiligne à
pentes rationnelles $\Gamma$ pondéré par $w$  qui vérifie la condition
d'équilibre en chacun de ses sommets :
pour tout sommet $s$ de $\Gamma$ adjacent aux arêtes $a_1,\ldots,a_k$
de  vecteurs primitifs sortants correspondants $\vec v_1,\ldots,\vec
v_k$, on a
$$\sum_{i=1}^k w(a_i)\vec v_i=\vec 0 $$
\end{defi}

La condition d'équilibre est bien connue en électricité sous le nom de
``Loi de Kirchoff''. 

\begin{exe}
Le graphe pondéré de la Figure \ref{graph}c est une courbe
tropicale. En effet, les trois vecteurs primitifs sortant du sommet
sont $(-1,0)$, $(0,-1)$ et $(1,1)$ et les poids  sont
tous 1, donc la somme est bien le vecteur nul.
\end{exe}

\begin{exe}
En revanche, le graphe pondéré de la Figure \ref{graph}a n'est pas une
courbe tropicale. Les autres graphes pondérés des 
Figures \ref{graph} et
\ref{courbe trop} sont des courbes tropicales.
\end{exe}

\begin{figure}[h]
\begin{center}
\begin{tabular}{ccc}
\includegraphics[height=3.5cm, angle=0]{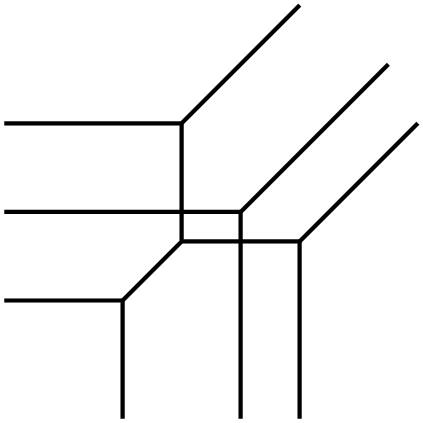}&
\hspace{10ex}  &
\includegraphics[height=3cm, angle=0]{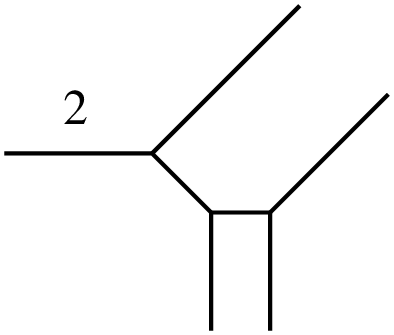}
\\ a) && b) 

\\
\\\includegraphics[height=4cm, angle=0]{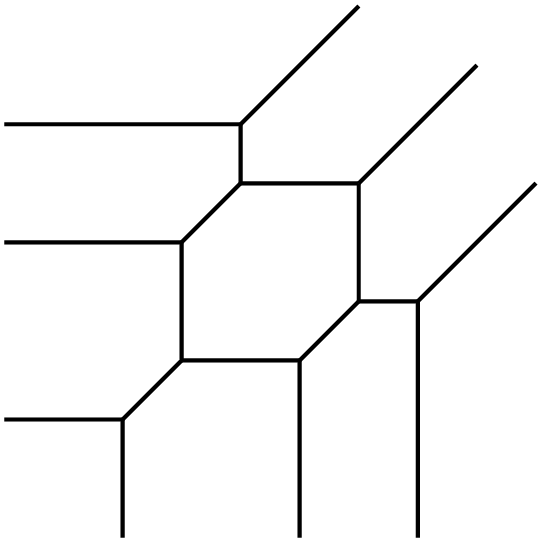}&
\hspace{10ex}  &
\includegraphics[height=4cm, angle=0]{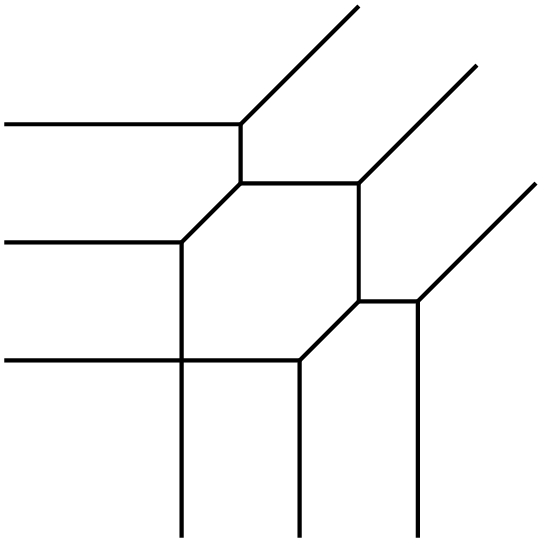}

\\c)&&d)
\end{tabular}
\end{center}
\caption{Quelques courbes tropicales}
\label{courbe trop}
\end{figure}

\begin{rem}
En raison de la condition
d'équilibre, une courbe tropicale n'est \textit{jamais} compacte. 
\end{rem}

Comme dans le cas des courbes algébriques, nous pouvons définir les
courbes tropicales irréductibles, le
degré d'une courbe tropicale et le genre d'une courbe tropicale
nodale.

\begin{defi}
Une courbe tropicale est irréductible si elle n'est pas l'union de
deux courbes tropicales.
\end{defi}

\begin{exe}
Les courbes tropicales des Figures \ref{graph}c, d et \ref{courbe
  trop}b, c et d sont irréductibles. Les courbes tropicales des
Figures \ref{graph}b et  \ref{courbe
  trop}a sont réductibles. Par exemple la courbe tropicale de la
Figure  \ref{courbe
  trop}a est l'union des courbes tropicales des Figures \ref{graph}c
et d. 
\end{exe}

\begin{defi}
Une courbe tropicale $C$ est de degré $d$ si pour chacune des
directions $(-1,0)$, $(0,-1)$ et $(1,1)$, la somme des poids des
arêtes non bornées de $C$ partant à l'infini dans cette direction est
$d$, et si $C$ n'a pas d'autres arêtes non bornées.  
\end{defi}

\begin{exe}
La courbe tropicale de la Figure \ref{graph}c est de degré 1.
 Les courbes tropicales des
Figures \ref{graph}d et  \ref{courbe
  trop}b sont de degré 2. Les courbes tropicales des
 Figures \ref{courbe
  trop}a, c et d sont de degré 3.
\end{exe}

\begin{rem}
Il se peut donc avec cette définition qu'une courbe tropicale n'ait
pas de degré, comme par exemple celle de le Figure \ref{graph}b. Cela
est dû  à la ``mauvaise'' définition de courbe tropicale que
nous adoptons ici. Si l'on voit une courbe tropicale comme l'ensemble
des zéros tropicaux d'un polynôme tropical, alors cette difficulté
disparaît : le degré de la courbe est simplement le degré du polynôme,
comme en géométrie classique. La courbe tropicale de la Figure
\ref{graph}b est alors de degré 4.
\end{rem}

Le nombre maximal de sommets et d'arêtes d'une courbe tropicale
s'exprime facilement en fonction du degré de la courbe. Nous aurons
besoin dans la suite de la proposition suivante.  

\begin{prop}\label{max som}
Une courbe tropicale de degré $d$ a au plus $d^2$ sommets.
\end{prop}

Comme en géométrie complexe, nous appellerons \textit{droite} (resp
\textit{conique}, \textit{cubique}) tropicale une
courbe tropicale de degré 1 (resp. 2, 3).
Une courbe tropicale peut être assez compliquée, et comme à la section
\ref{alg cplx} nous allons
travailler avec les plus simples, les courbes nodales.

\begin{defi}\label{ct nodale}
Une courbe tropicale $C$ est nodale si chaque arête non bornée de $C$
est de poids 1 et si chaque sommet de $C$ est trivalent ou quadrivalent.
 De plus, si $s$ est un sommet quadrivalent adjacent aux
arêtes $a_1,\ldots a_4$ de vecteurs primitifs sortants correspondants
$\vec v_1,\ldots 
,\vec v_4$, alors quitte à renuméroter les $a_i$ et $\vec v_i$ on a
$w(a_1)=w(a_2)$, $w(a_3)=w(a_4)$, $\vec v_1=-\vec v_2$ et $\vec
v_3=-\vec v_4$. 
\end{defi}
\begin{figure}[h]
\begin{center}
\begin{tabular}{c}
\includegraphics[height=3cm, angle=0]{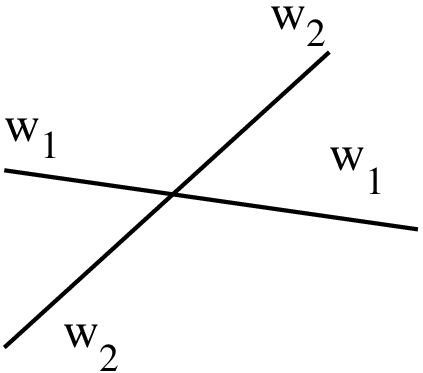}
\end{tabular}
\end{center}
\caption{Un sommet quadrivalent d'une courbe tropicale nodale}
\label{pt2 trop}
\end{figure}

En d'autres termes, une courbe tropicale $C$ est nodale si tous ses
sommets qui sont adjacents 
à au moins 4 arêtes sont en fait adjacents à exactement 4 arêtes et
sur un voisinage de ce sommet, $C$ est l'union de deux intervalles
(voir Figure \ref{pt2 trop}). Nous
retrouvons la notion de point d'intersection de deux branches de $C$
comme à la section \ref{alg cplx}. 

\begin{exe}
La courbe tropicale de la Figure \ref{graph}b n'est pas nodale car
un de ses sommets est adjacent à 6 arêtes. La courbe tropicale de la
Figure \ref{courbe trop}b n'est pas nodale car une de ses arêtes
infinies est de poids 2.
Les courbes tropicales des
Figures \ref{graph}c, d et \ref{courbe trop}a, c, d sont nodales.
\end{exe}

Pour calculer le genre d'une
courbe algébrique complexe, il suffit de compter le nombre de points doubles
de la courbe. En géométrie tropicale, l'idée est toujours la même mais
la définition est un peu 
différente car un sommet quadrivalent peut compter pour
plusieurs points doubles, et un sommet trivalent peut 
aussi ``cacher'' des points doubles.

\begin{defi}\label{ct genre}
Soit $C$ une courbe tropicale irréductible nodale de degré $d$ avec
$\sigma$ sommets trivalents. Le
genre de $C$, noté $g(C)$ est défini par
$$g(C)=\frac{\sigma -3d+2}{2}$$
\end{defi}

Comme dans le cas des courbes algébriques, le genre est un entier
compris entre  0 et $\frac{(d-1)(d-2)}{2}$.

\begin{exe}
Les courbes tropicales des Figures \ref{graph}c, d,  \ref{courbe
  trop}d sont de genre 0. La courbe tropicale de la Figure
\ref{courbe trop}c est de genre 1. 
\end{exe}

Comme en géométrie complexe, nous appellerons \textit{courbe tropicale
  rationnelle} une courbe tropicale de genre 0.

\subsection{Théorèmes de Correspondance}\label{corr}
Nous avons maintenant toutes les définitions requises pour pouvoir
faire de la géométrie énumérative tropicale.
Fixons nous un degré $d\ge 1$, un
genre $g\ge 0$, et 
une configuration générique $\omega=\{p_1,\ldots, p_{3d-1+g}\}$ de
$3d-1+g$ points dans 
$\RR^2$. Considérons 
alors l'ensemble 
$\TT \mathcal C (d,g,\omega)$ de toutes les courbes  tropicales
irréductibles 
nodales\footnote{Comme
  dans le cas algébrique, il existe une notion de genre pour toute
  courbe tropicale, et pour une configuration générique $\omega$,
  toutes les courbes de $\TT \mathcal C (d,g,\omega)$ seront
  nodales. On peut donc encore une fois oublier de préciser
  ``nodales'' dans l'énoncé du problème.} de  degré 
$d$,  de genre  $g$,  passant par  tous les points de $\omega$. 

\begin{prop}[Mikhalkin, \cite{Mik1}]\label{fini trop}
Si la configuration de
points $\omega$ est générique, alors l'ensemble $\TT \mathcal C
(d,g,\omega)$ est fini.
\end{prop}

\begin{exe}\label{en dte}
Si nous prenons $d=1$ (et donc $g=0$), nous regardons alors les
droites tropicales 
passant par deux points  de $\RR^2$. Si les deux points sont
sur une même droite (habituelle!) d'équation $X=\alpha$ ou $Y=\alpha$ ou
$X=Y+\alpha$, alors il existe une infinité de droites tropicales passant
par ces deux points (voir la Figure  \ref{dtes trop}a). Dans ce cas,
les deux points ne 
sont pas en position générique. Par contre, si les deux points ne sont
pas sur une telle droite, alors il existe une unique droite tropicale
passant par ces deux points, comme nous le voyons sur les Figures \ref{dtes
  trop}b, c et d pour
différentes paires de points. Nous retrouvons ainsi le même résultat
que pour les 
droites algébriques! 
\end{exe}

\begin{figure}[h]
\begin{center}
\begin{tabular}{ccccccc}
\includegraphics[height=2.5cm, angle=0]{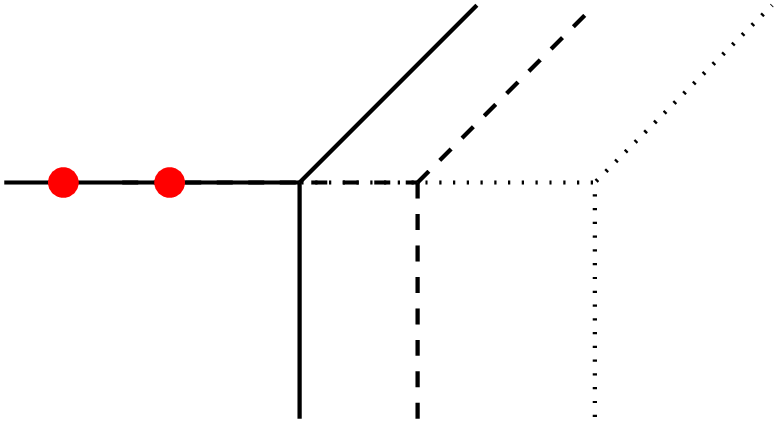}&
\hspace{1ex}  &
\includegraphics[height=2.5cm, angle=0]{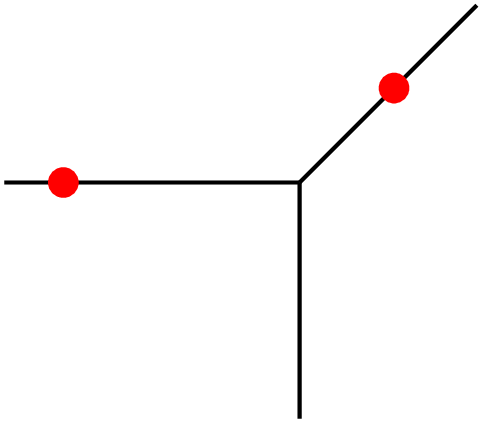}&
\hspace{1ex}  &
\includegraphics[height=2.5cm, angle=0]{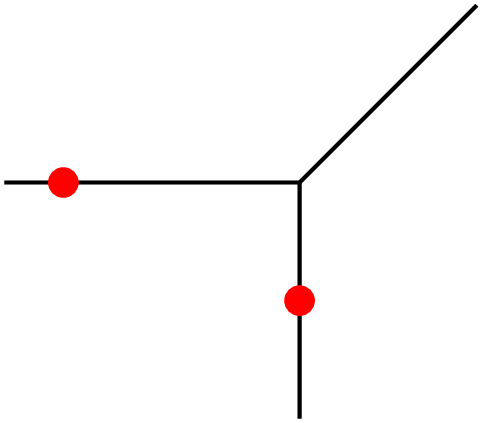}&
\hspace{1ex}  &
\includegraphics[height=2.5cm, angle=0]{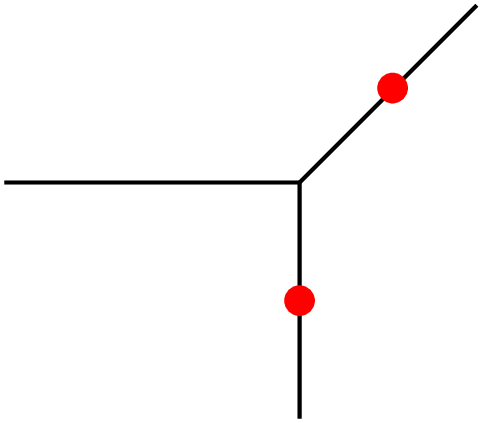}

\\ a) Non générique&& b)  Générique &&
c)  Générique &&d) Générique
\end{tabular}
\end{center}
\caption{Deux points dans $\RR^2$ et les droites les contenant}
\label{dtes trop}
\end{figure}

\begin{exe}\label{en con}
Fixons nous maintenant $d=2$ (et donc $g=0$). Nous comptons  les
coniques tropicales passant par 5 points en position générique, et
dans ce cas encore nous 
en trouvons une unique. Comme pour les courbes algébriques! Nous
pouvons voir cette 
conique tropicale pour un exemple de configuration sur la Figure
\ref{Tconcub}a. La 
preuve que cette conique est la seule sera donnée à la section \ref{etage}.
\end{exe}
\begin{figure}[h]
\begin{center}
\begin{tabular}{ccccccc}
\includegraphics[height=3cm, angle=0]{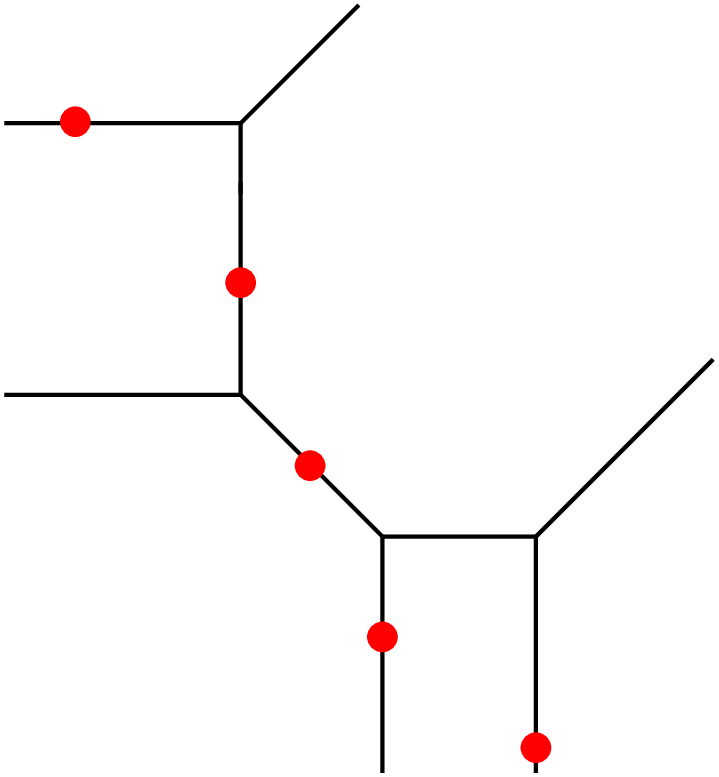}&
\hspace{6ex}  &
\includegraphics[height=4cm, angle=0]{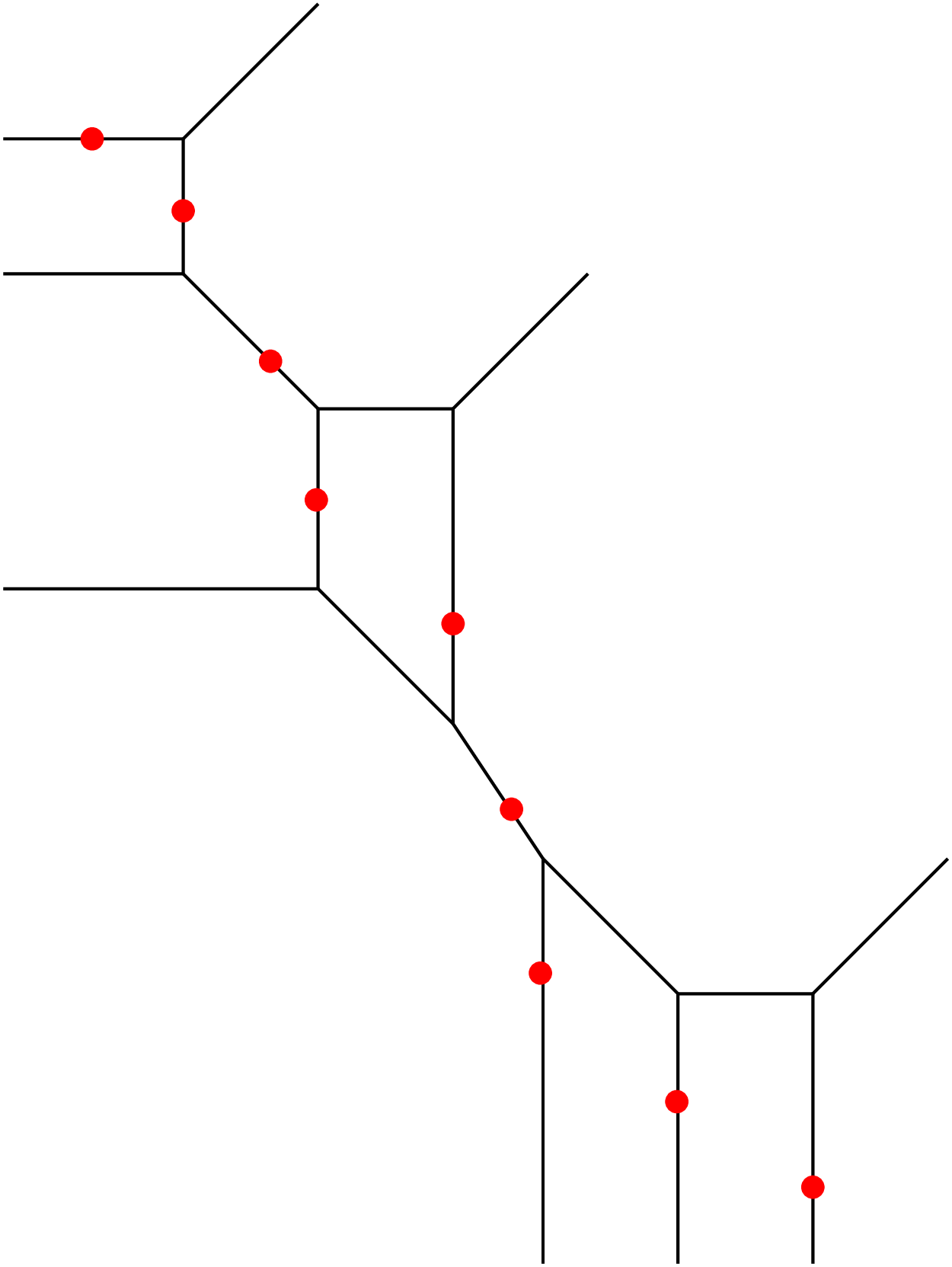}

\\ a) Une unique conique tropicale && b)  Une unique cubique
tropicale de genre 1
\\ passant par 5 points&&  passant par 9 points
\end{tabular}
\end{center}
\caption{Deux solutions de problèmes énumératifs tropicaux}
\label{Tconcub}
\end{figure}

\begin{exe}
Pour $d=3$, le genre peut être 0
ou 1. Si nous fixons $g=1$, alors comme aux Exemples \ref{en dte}
et \ref{en con},
 nous trouvons une
unique cubique de genre 1 passant par 9 points génériques (voir Figure
\ref{Tconcub}b). Regardons maintenant les cubiques tropicales rationnelles
passant par  une configuration  de 8 points. La Figure \ref{Tcub0}
nous donne
toutes les cubiques tropicales  rationnelles passant par une  telle
configuration.  La 
preuve que toutes les cubiques rationnelles passant par les 8 points
choisis sont bien représentées sur la Figure \ref{Tcub0} sera donnée à
la section 
\ref{etage}. Nous avons alors beau compter et recompter, 
nous ne trouvons que 9
cubiques tropicales rationnelles passant par nos 8 points, rien à voir avec
le nombre $N(3,0)$ égal à 12... La similitude entre la géométrie
tropicale et la géométrie complexe
s'arrêterait-elle ici?
\end{exe}

\begin{figure}[h]
\begin{center}
\begin{tabular}{ccccc}
\includegraphics[height=4cm, angle=0]{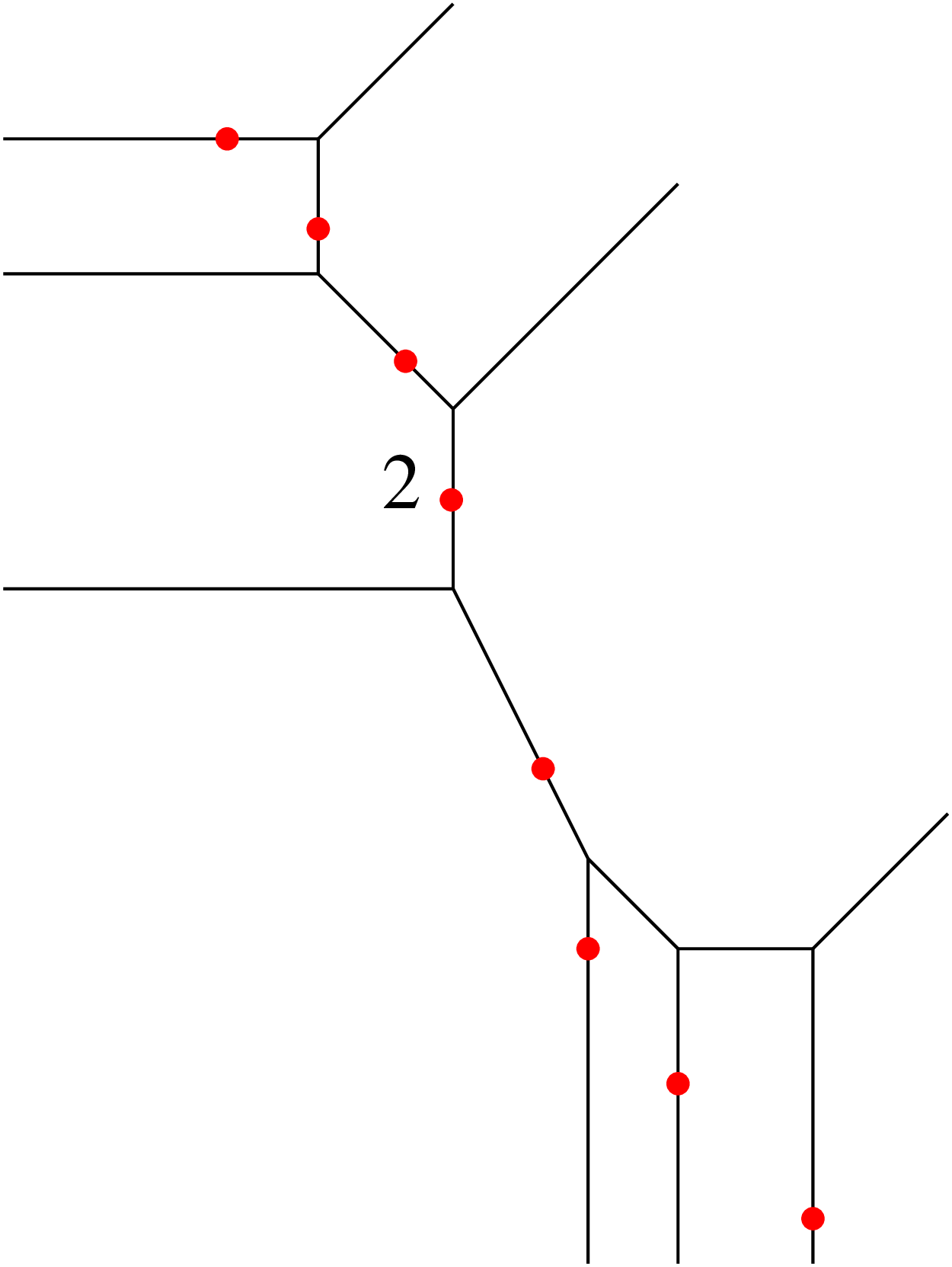}&
\hspace{3ex}  &
\includegraphics[height=4cm, angle=0]{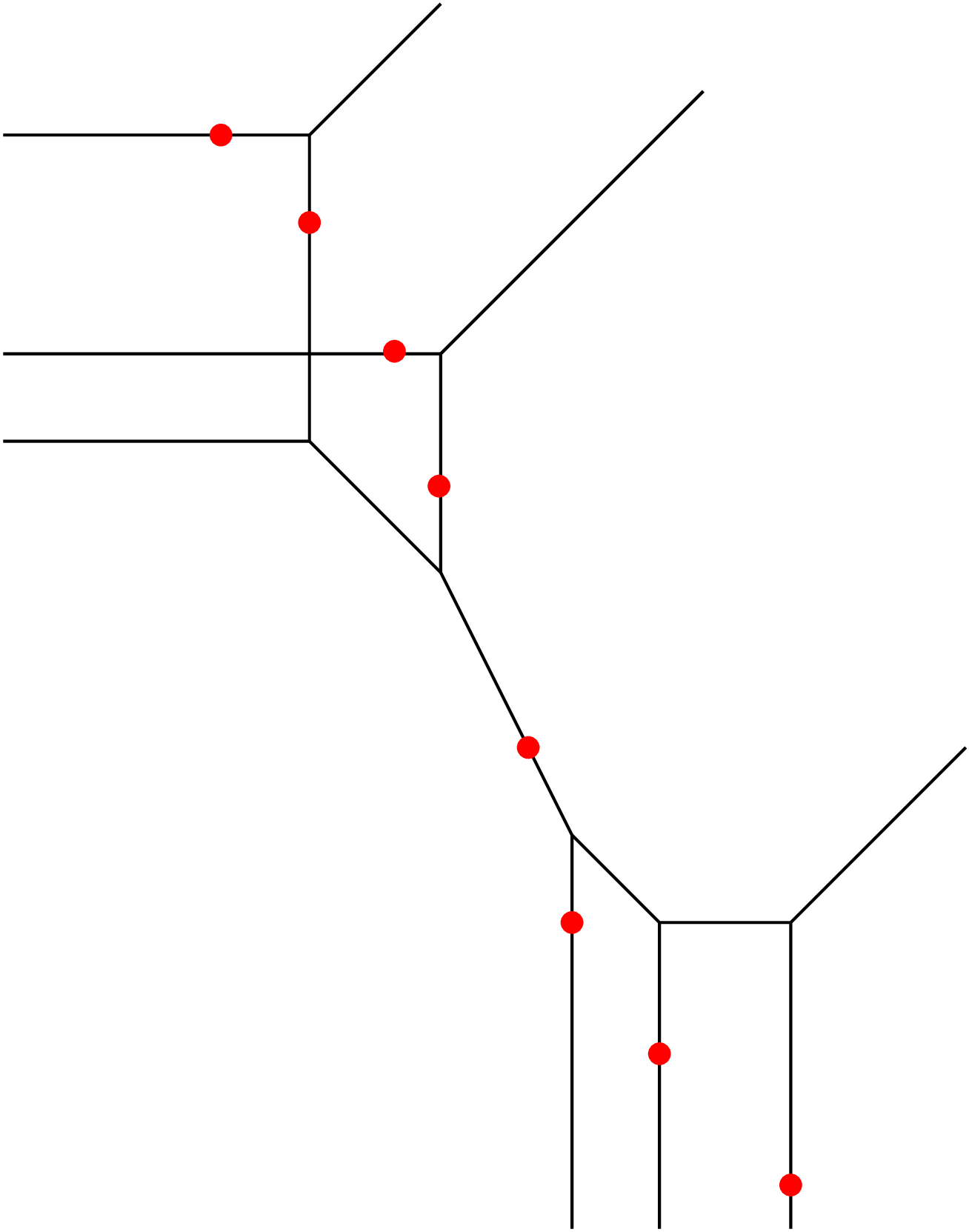}&
\hspace{3ex}  &
\includegraphics[height=4cm, angle=0]{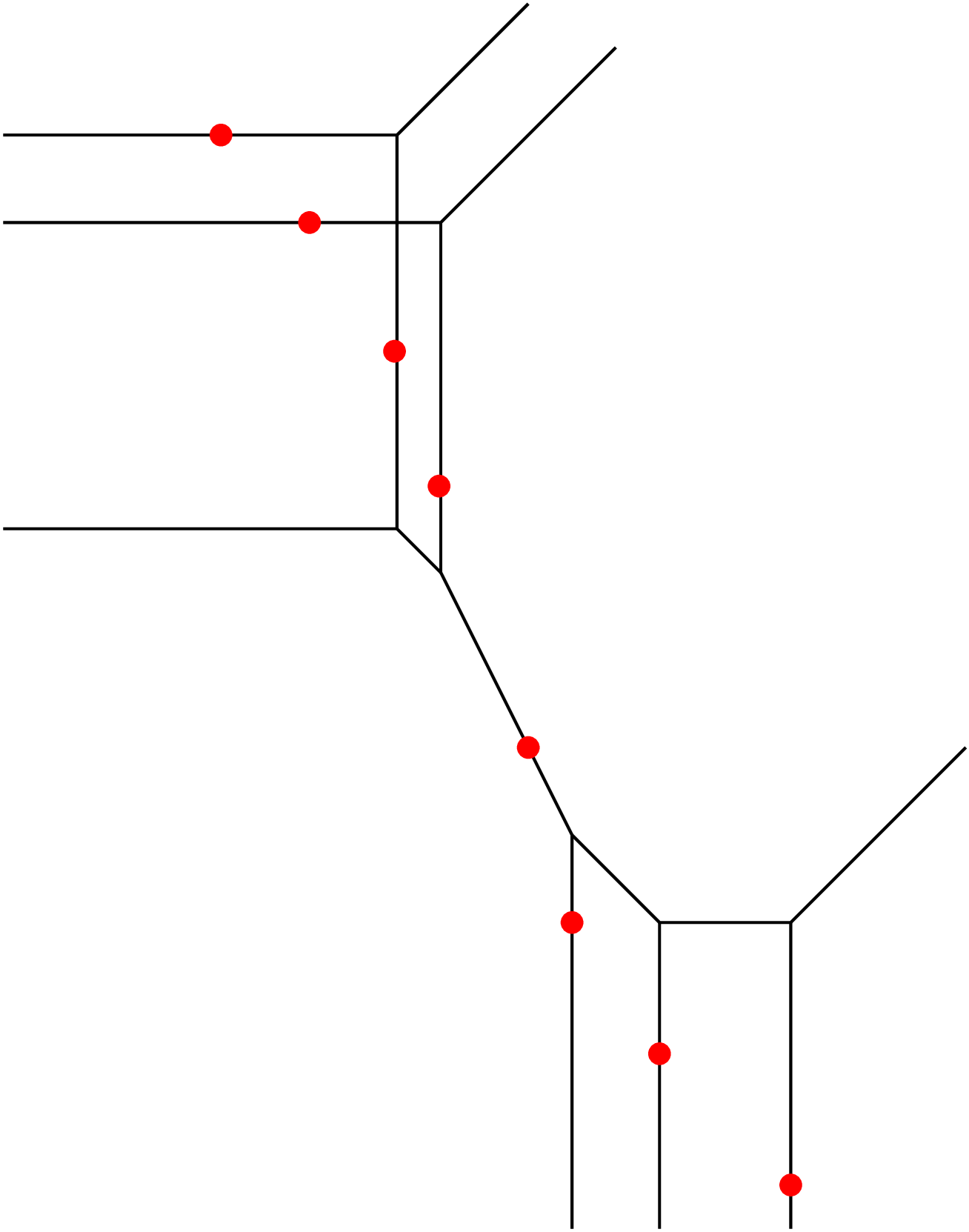}
\\ a) &&b) &&c)
\\
\\
\\\includegraphics[height=4cm, angle=0]{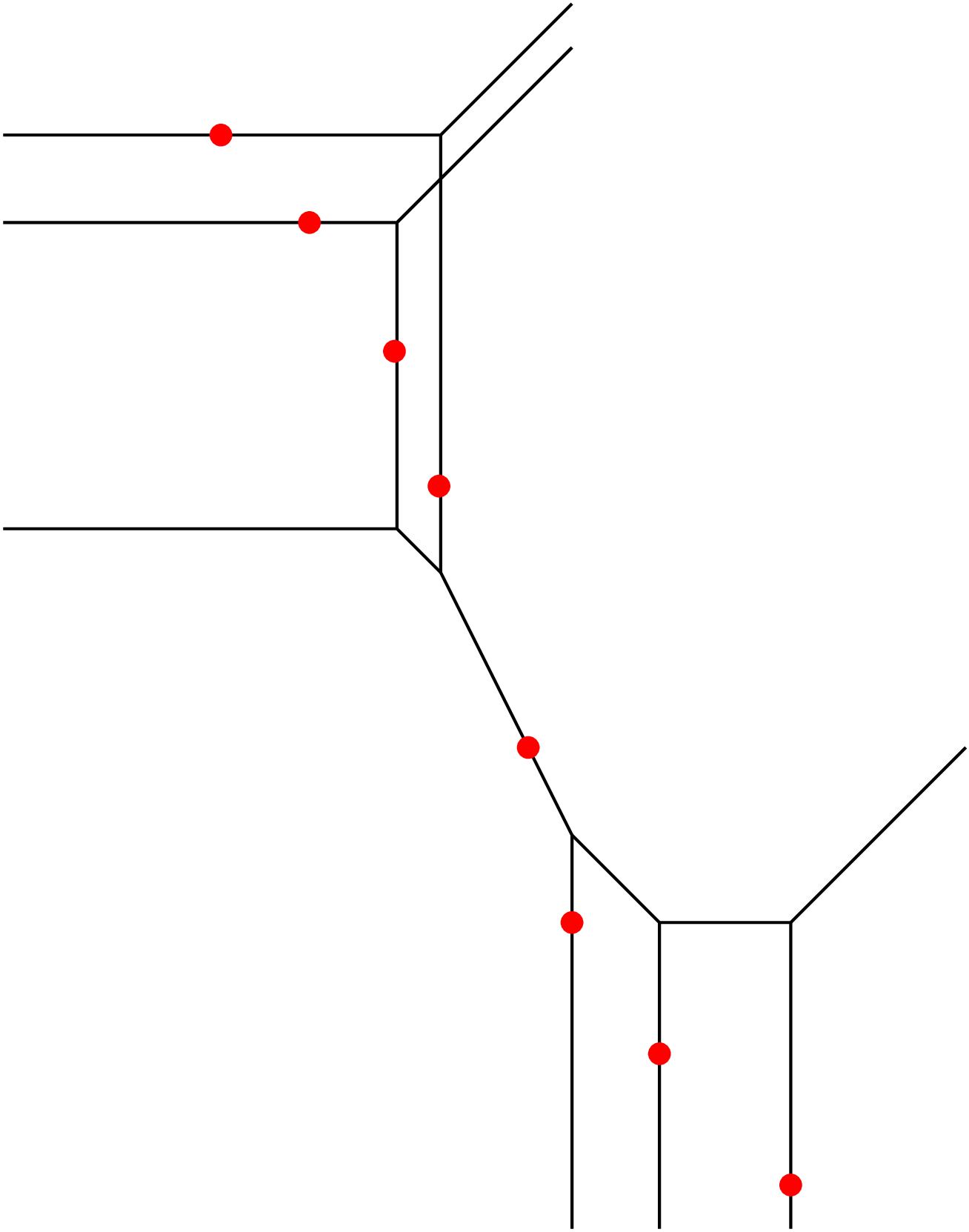}&
\hspace{3ex}  &
\includegraphics[height=4cm, angle=0]{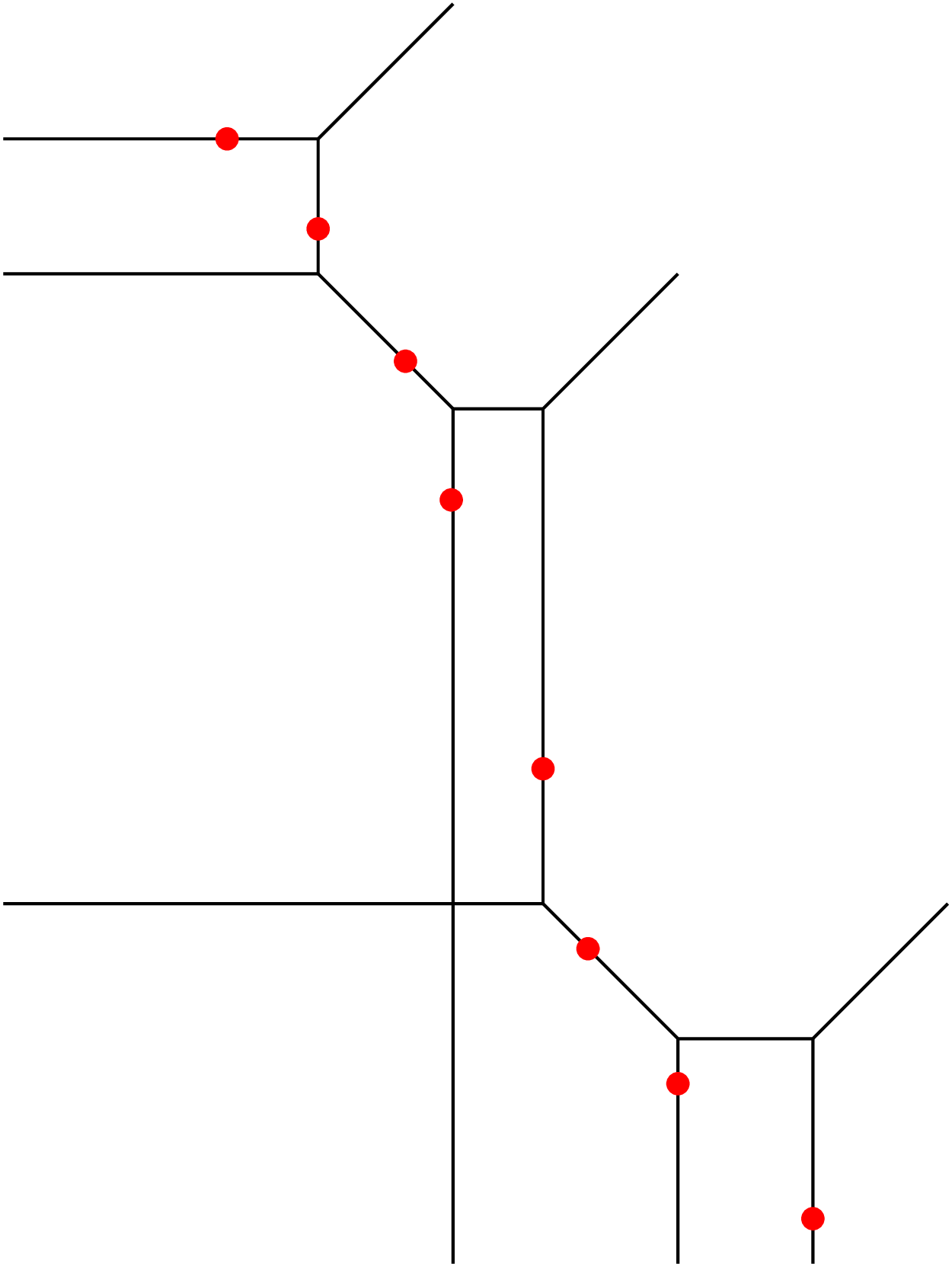}&
\hspace{3ex}  &
\includegraphics[height=4cm, angle=0]{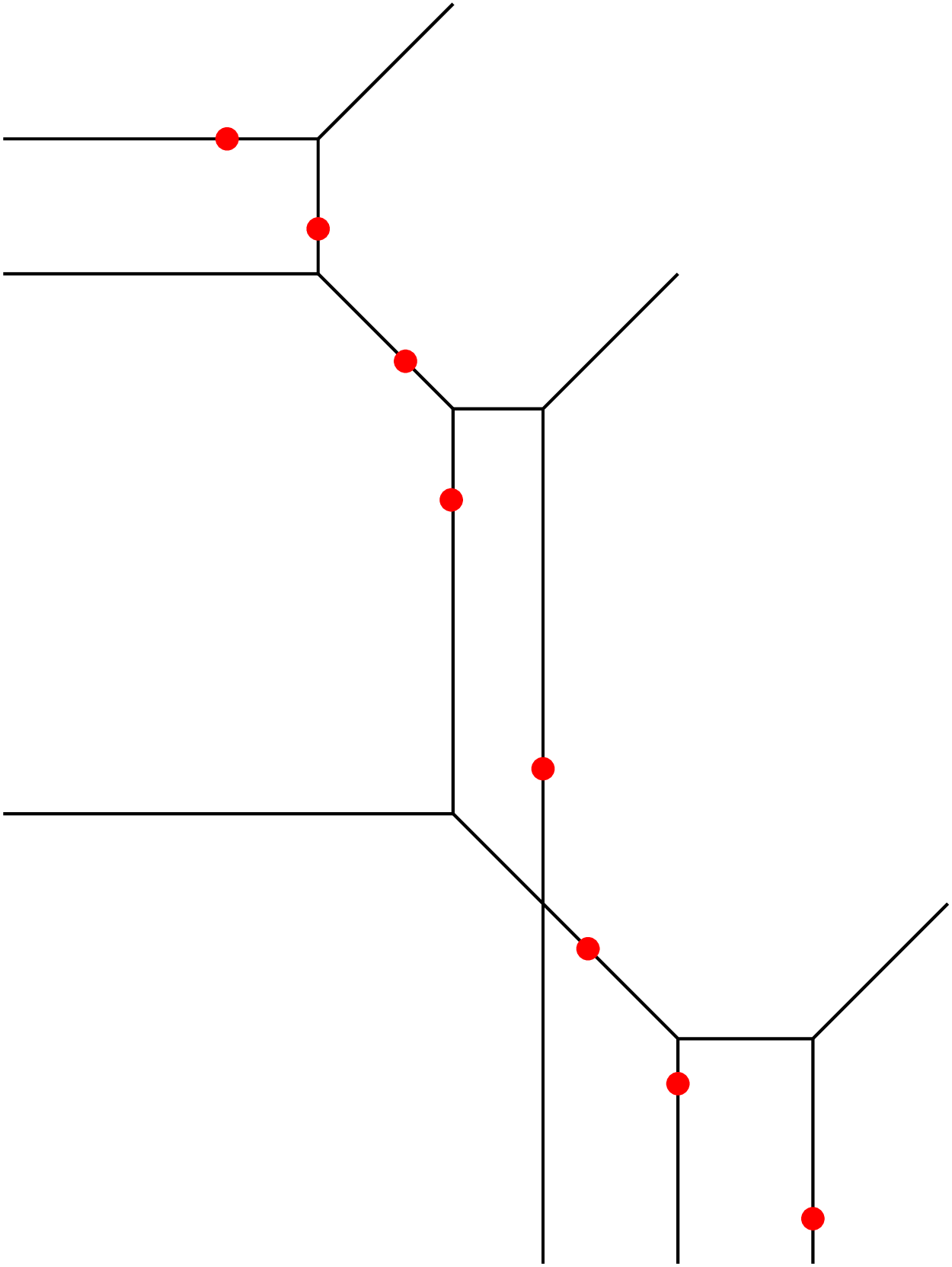}
\\ d) &&e) &&f)

\\
\\
\\\includegraphics[height=4cm, angle=0]{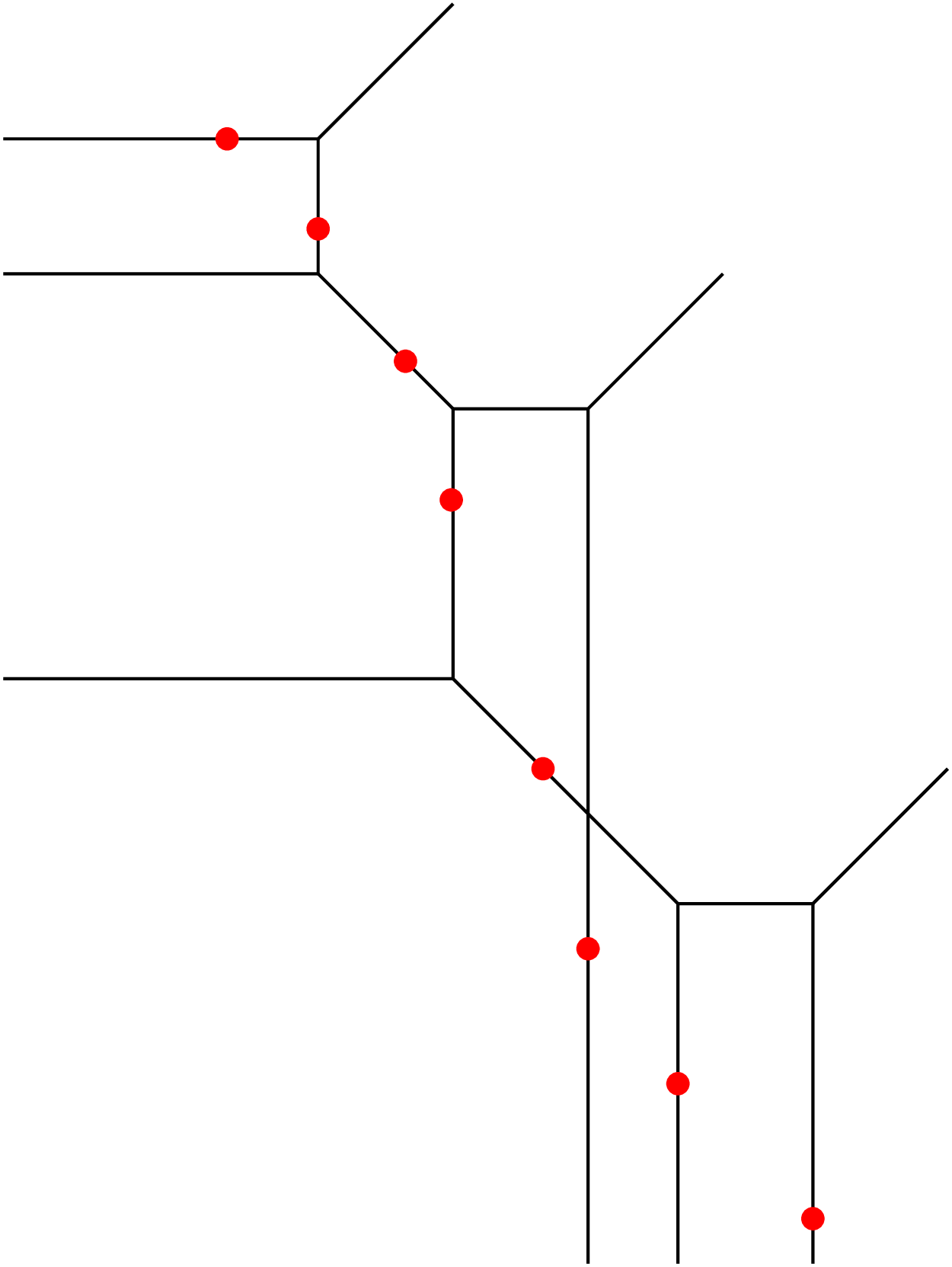}&
\hspace{3ex}  &
\includegraphics[height=4cm, angle=0]{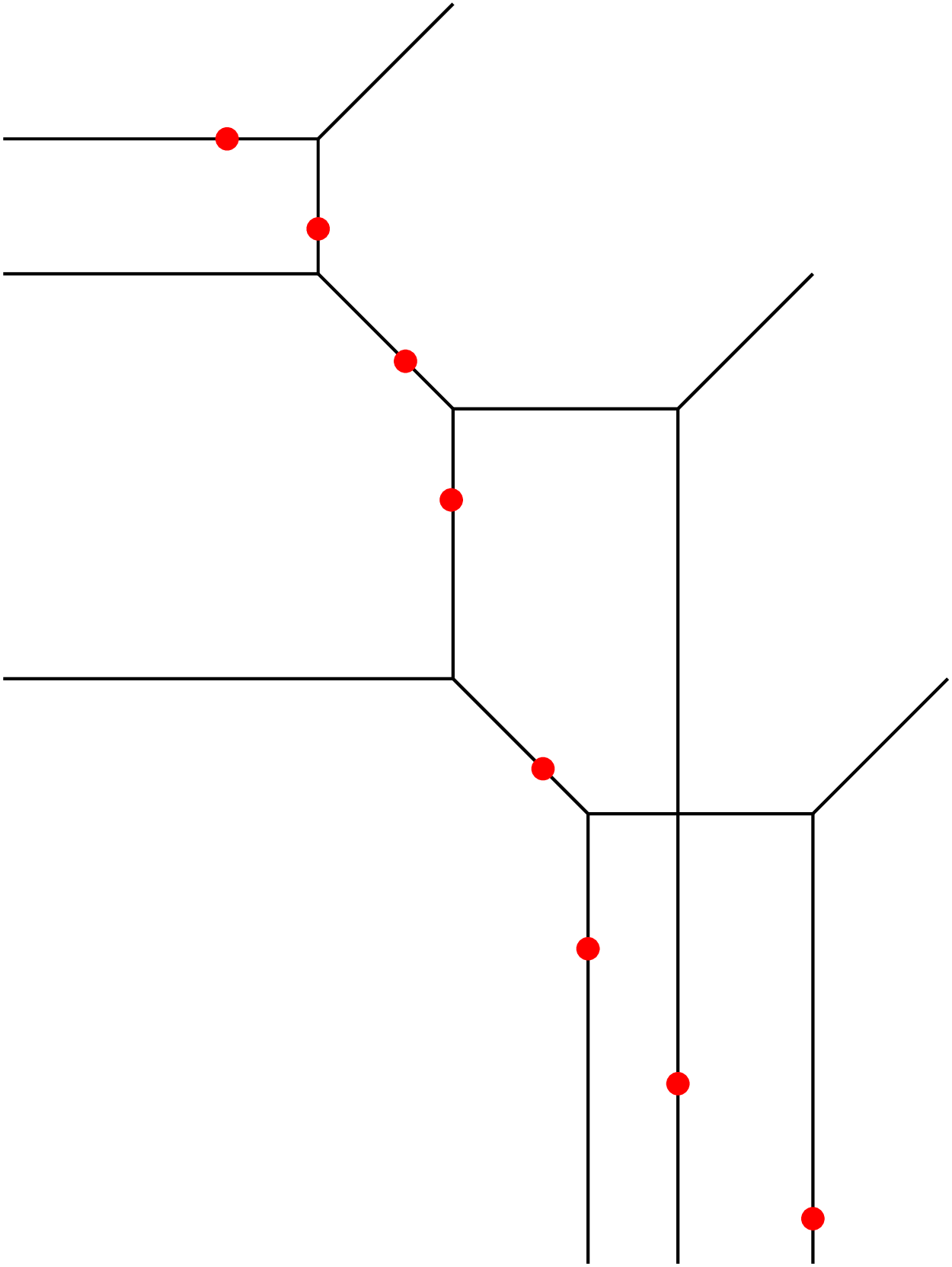}&
\hspace{3ex}  &
\includegraphics[height=4cm, angle=0]{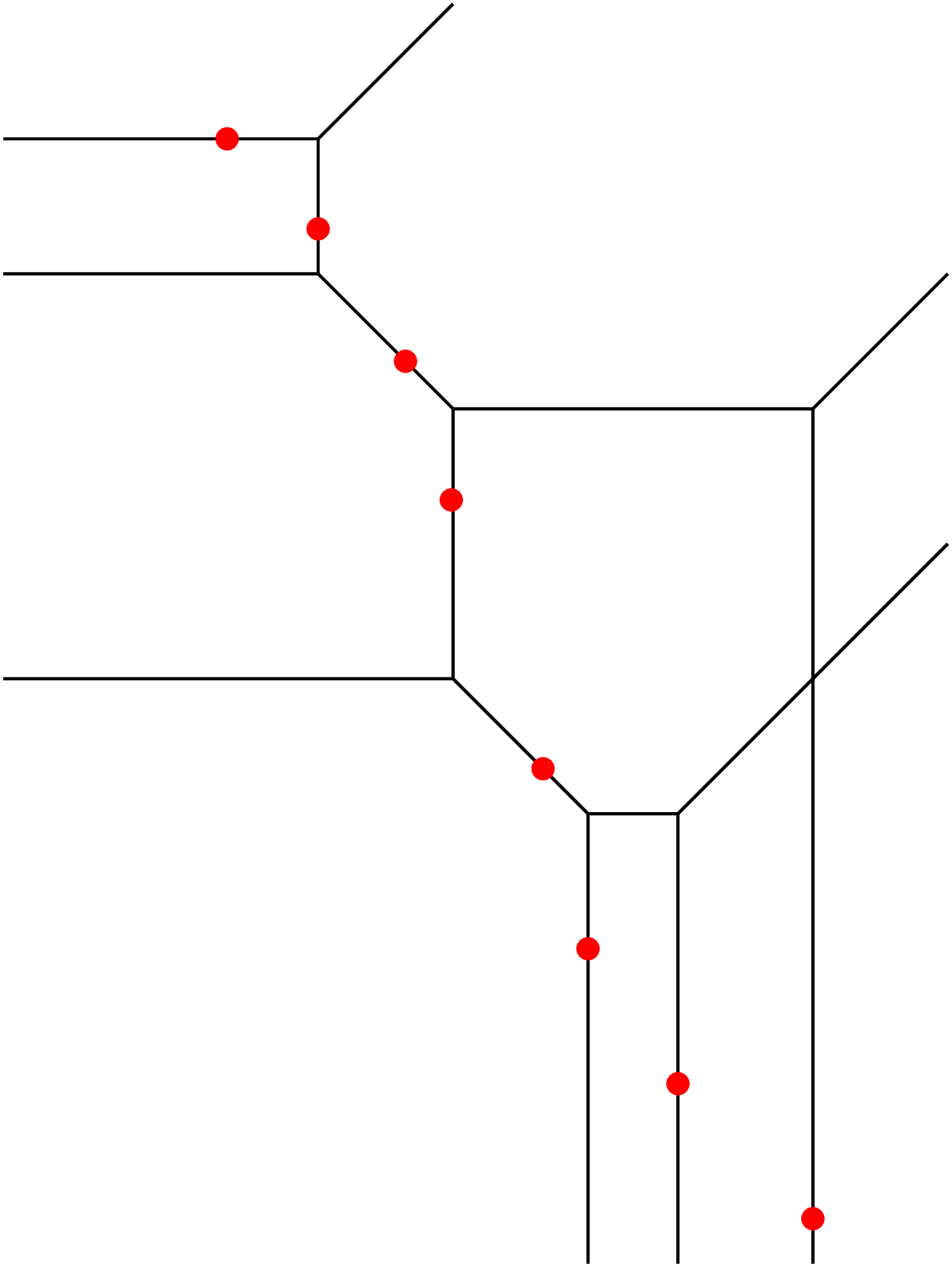}
\\ g) &&h) &&i)

\end{tabular}
\end{center}
\caption{Cubiques tropicales rationnelles passant par 8 points}
\label{Tcub0}
\end{figure}

En fait, nous trouvons 9 dans l'exemple précédent car nous comptons
mal. Nous avons déjà vu à la Définition \ref{ct genre} du genre d'une courbe
tropicale nodale qu'un point d'intersection de deux branches pouvait
en fait compter pour plusieurs points 
doubles. La situation est analogue ici : il peut arriver, et il arrive
souvent, que l'on doive compter une courbe tropicale de $\TT \mathcal
C (d,g,\omega)$  plusieurs 
fois. En d'autres termes, il faut attribuer à chaque courbe une
\textit{multiplicité}, et compter les courbes de $\TT \mathcal C
(d,g,\omega)$ avec cette multiplicité. 

Quelle est  donc cette multiplicité?
Soit $C$ une courbe tropicale nodale, et $s$ un sommet trivalent de $C$ adjacent
aux arêtes $a_1,a_2$ et $a_3$ de vecteurs primitifs sortants
correspondants $\vec
v_1,\vec v_2$ et $\vec v_3$. La multiplicité complexe de $s$ est
définie par
$$\mu_\CC(s)=w(a_1)w(a_2)\Big|\det\big(\vec v_1,\vec v_2\big)\Big| $$
Grâce à la condition d'équilibre, le nombre $\mu_\CC(s)$ ne dépend pas
de la numérotation des $a_i$. Notons $\Gamma_{0,3}$ l'ensemble des
sommets trivalents de $C$.

\begin{defi}
La multiplicité complexe
d'une courbe tropicale nodale $C$, notée $\mu_\CC(C)$, est définie par
$$\mu_{\CC}(C)= \prod_{s\in \Gamma_{0,3}}\mu_\CC(s) $$
\end{defi}

Quelle est la signification de cette multiplicité? Prenons une
configuration générique   $\omega$ de $3d-1+g$ points dans
$\RR^2$. Dans son article \cite{Mik1},  Mikhalkin construit à partir
de $\omega$  une configuration générique $\omega_0$
 de  $3d-1+g$
points  dans $\CC^2$ et montre ensuite qu'il
existe une
application naturelle $\phi :\mathcal C (d,g,\omega_0)\to \TT \mathcal C
(d,g,\omega)$. La multiplicité d'une courbe de $\TT \mathcal C
(d,g,\omega)$ est alors le cardinal de $\phi^{-1}(C)$.

Avec ces précisions, il devient clair  que c'est  avec cette multiplicité qu'il
faut compter les courbes 
tropicales. 

\begin{thm}[Mikhalkin \cite{Mik1}]\label{corr 1}
Pour tous $d\ge 1$, $g\ge 0$, et $\omega$ générique, on a l'égalité
$$N(d,g)=  \sum_{C\in \TT \mathcal C (d,g,\omega)}\mu_{\CC}(C)$$
\end{thm}

Il n'est pas difficile de voir que le cardinal de  $\TT \mathcal C
(d,g,\omega)$ dépend de $\omega$. Un corollaire immédiat du
Théorème \ref{corr 1}
est  que le nombre de courbes
tropicales dans $\TT \mathcal C (d,g,\omega)$ comptées avec
multiplicité ne dépend pas de $\omega$.

\begin{exe}
Toutes les courbes tropicales des Figures \ref{dtes trop}, \ref{Tconcub} et
\ref{Tcub0} sont de multiplicité complexe 1 sauf la courbe
de la Figure \ref{Tcub0}a qui est de multiplicité complexe 4. Nous
retrouvons ainsi 
$N(1,0)=N(2,0)=N(3,1)=1$ et $N(3,0)=12$. 
\end{exe}

Ainsi, la géométrie énumérative tropicale donne les mêmes résultats
que la géométrie énumérative complexe. Qu'en est-il de la géométrie
énumérative réelle? Dans la démonstration du Théorème
\ref{corr 1}, il est possible de choisir $\omega_0\subset \RR^2$ et
donc de considérer l'ensemble $\RR \mathcal C (d,g,\omega_0)$.
Mikhalkin a alors montré que la somme des signes de Welschinger des 
courbes algébriques réelles dans 
$\phi^{-1}(C)$  pouvait se calculer très facilement à partir de la
seule connaissance de $C$ et était égale à $-1$, 0 ou 1. C'est-à-dire
qu'il existe  
 une multiplicité
réelle  d'une courbe tropicale rationnelle  telle 
qu'en comptant les  
courbes
tropicales dans $\TT \mathcal C (d,0,\omega)$ avec cette
multiplicité, on retrouve l'invariant de Welschinger.
Soit $C$ une courbe tropicale rationnelle nodale, et $s$ un sommet
trivalent de
$C$. Si $\mu_\CC(s)$ est pair, alors la multiplicité réelle de
$C$ sera 0. Si $\mu_\CC(s)$ est
impair,  sa valeur modulo 4 est soit 1
soit 3. Posons
$$o(C)=\mathcal Card\left(\{s\textrm{ sommet trivalent de $C$ tel que
}\mu_\CC(s)=3\ mod \ 4\}\right)  $$

\begin{defi}
La multiplicité réelle
d'une courbe tropicale nodale irréductible $C$, notée $\mu_\RR(C)$,
est définie par 
$$\mu_{\RR}(C)=0 $$
si $\mu_{\CC}(C)$ est pair, et par
$$\mu_{\RR}(C)= (-1)^{o(C)} $$
sinon.
\end{defi}

\begin{thm}[Mikhalkin \cite{Mik1}]\label{corr 2}
Pour tous $d\ge 1$ et $\omega$ générique, on a l'égalité
$$W(d)=  \sum_{C\in \TT \mathcal C (d,0,\omega)}\mu_{\RR}(C)$$
\end{thm}

\begin{exe}
Toutes les courbes tropicales rationnelles sur les Figures \ref{dtes trop},
\ref{Tconcub}a et 
\ref{Tcub0} sont de multiplicité réelle 1 sauf la courbe
de la Figure \ref{Tcub0}a qui est multiplicité réelle 0. Nous retrouvons ainsi
$W(1)=W(2)=1$ et $W(3)=8$. 
\end{exe}

\begin{rem}
Nous pouvons très bien attribuer avec les mêmes règles une
multiplicité réelle à une courbe tropicale nodale irréductible de
\textit{n'importe quel} genre. Une surprise nous attend alors (voir
\cite{IKS3}) :
contrairement à ce qui se passe pour les courbes algébriques réelles,
ces nombres tropicaux ne dépendent pas de $\omega$! On ne comprend
toujours pas les raisons profondes de cette
invariance tropicale.
\end{rem}

\section{Décomposition en étages de courbes tropicales}\label{etage}
Les Théorèmes \ref{corr 1} et \ref{corr 2} réduisent déjà nos problèmes
énumératifs en géométrie algébrique à des problèmes combinatoires : il
suffit de compter certains graphes rectilignes dans  $\RR^2$. Les
décompositions 
en étages sont une étape supplémentaire dans la simplification du
calcul des nombres $N(d,g)$ et $W(d)$.
L'idée est de prendre une configuration  quelconque de points, et
d'éloigner les points les uns des autres dans la direction verticale.
 À 
mesure que les points s'éloignent, nos courbes tropicales  se
``cassent'' en morceaux très simples, et le tour est joué. Cette méthode
peut être vue comme une variante tropicale de la méthode de Caporaso et Harris
en géométrie énumérative complexe.

\subsection{Étages d'une courbe tropicale}
Puisque nous distinguons une direction particulière dans $\RR^2$, les
arêtes d'une courbe tropicale se trouvent naturellement divisées en
deux catégories.
Assez logiquement, nous appelons arêtes \textit{verticales} d'une
courbes tropicale  celles qui sont 
parallèles à la droite d'équation 
$X=0$.

\begin{defi}
Un étage d'une courbe tropicale $C$ est une composante connexe de
$C$ privée de ses arêtes verticales.
\end{defi}
Les mathématiques étant basées sur le monde observé, deux étages seront
naturellement reliés par un \textit{ascenseur}.

\begin{defi}
Une composante connexe d'une courbe tropicale $C$ privée de ses étages
est appelée un ascenseur.
\end{defi}

\begin{rem}
Pour être totalement rigoureux, il faudrait écrire \textit{une
  composante connexe de l'adhérence de ...} dans les deux
définitions précédentes. 
\end{rem}

En général, les ascenseurs sont une union d'arêtes verticales, mais ne
sont pas eux même des arêtes verticales comme on peut le voir sur les
Figures \ref{Tcub0}d, e, ..., i. Puisqu'un sommet quadrivalent d'une courbe tropicale nodale est
l'intersection de deux intervalles, la Définition \ref{ct nodale} nous
permet de parler du poids d'un
ascenseur d'une courbe tropicale nodale.

\begin{exe}
Les ascenseurs, qui sont aussi des arêtes verticales ici, des courbes
tropicales de la Figure \ref{floor} 
sont dessinées en pointillés, leurs étages sont dessinés en traits
pleins. Leur poids est à chaque fois 1.
\end{exe}

\begin{figure}[h]
\begin{center}
\begin{tabular}{ccc}
\includegraphics[height=3cm, angle=0]{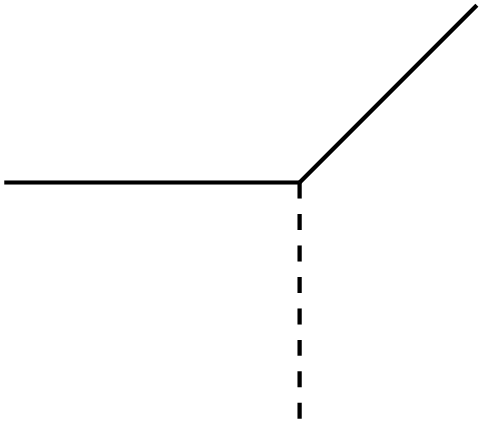}&
\hspace{10ex}  &
\includegraphics[height=4cm, angle=0]{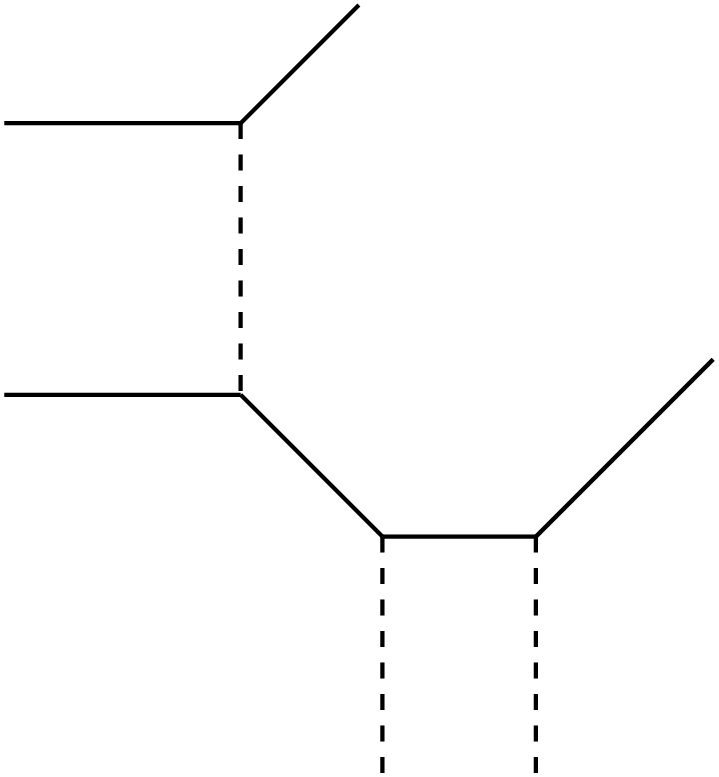}
\\a) Un étage && b) Deux étages

\\
\\ \includegraphics[height=4cm, angle=0]{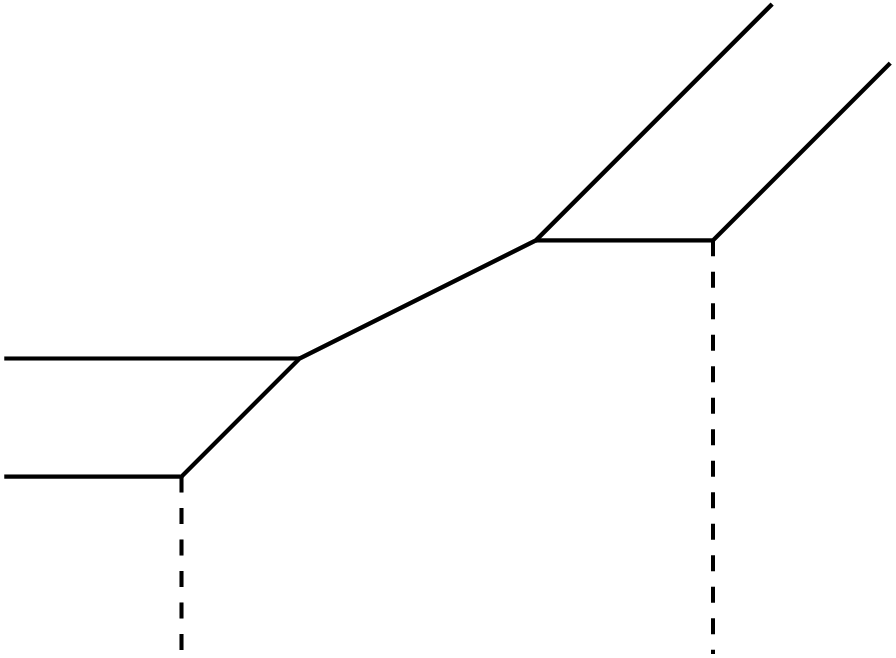}&
\hspace{10ex}  &
\includegraphics[height=6cm, angle=0]{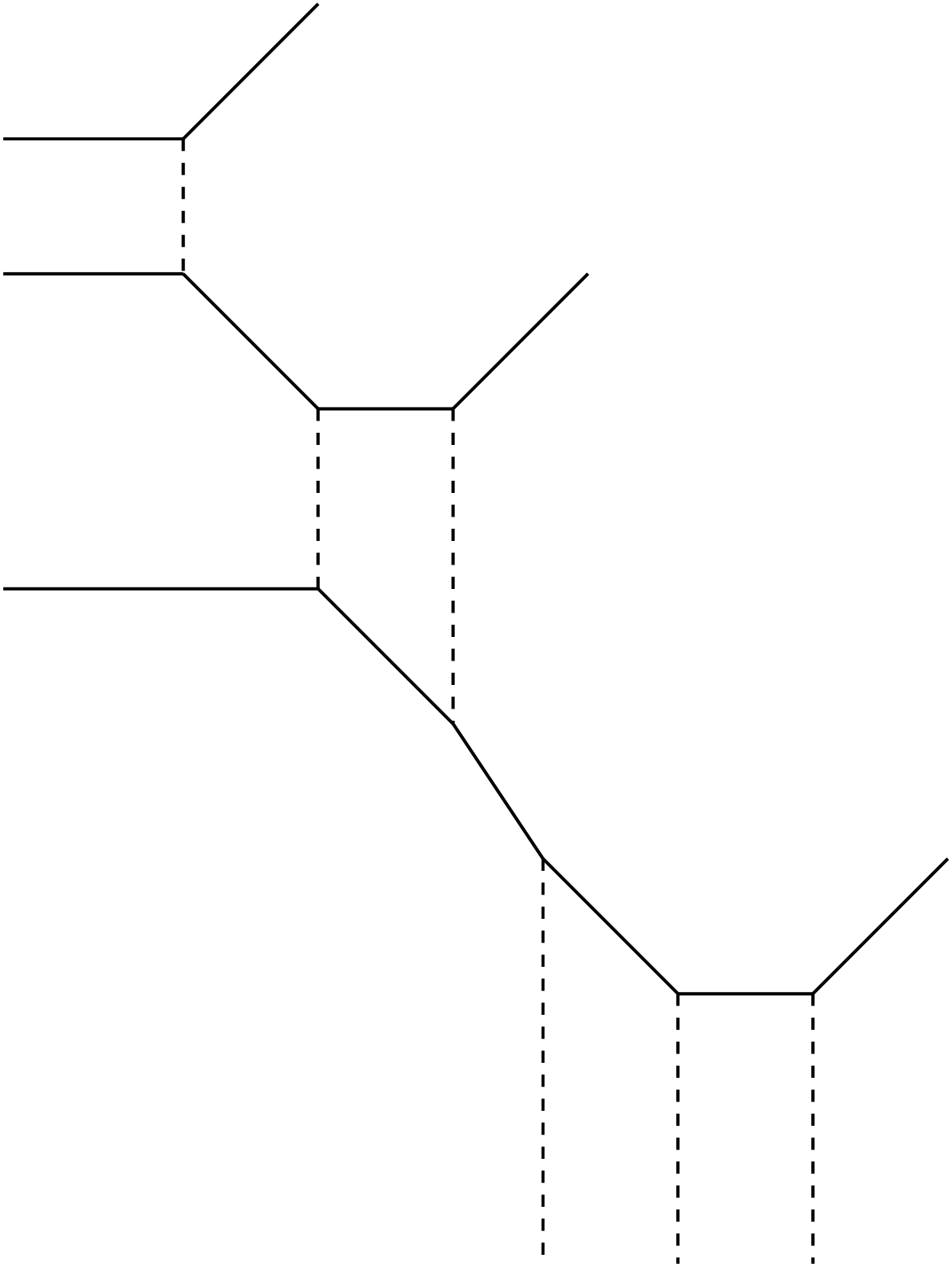}
\\ c) Un étage && d) Trois étages

\end{tabular}
\end{center}
\caption{Étages de courbes tropicales}
\label{floor}
\end{figure}

Comme d'habitude, on se fixe un degré $d$, un
genre $g$, 
 une configuration $\omega=\{p_1,\ldots, p_{3d-1+g}\}$ de
$3d-1+g$ points dans 
$\RR^2$, et on considère 
 l'ensemble 
$\TT \mathcal C (d,g,\omega)$.
Maintenant, envoyons un par un  les points $p_i$
``très haut'', et observons ce qui se passe.

\begin{exe}\label{exf1}
Supposons que $d=1$ et que nous  partons de la configuration de la Figure
\ref{decf 1}a. Lorsque nous déplaçons le point $p_2$ vers le haut, le
sommet de la droite tropicale passant par $p_1$ et $p_2$  se déplace
vers la gauche (Figure \ref{decf 1}b) jusqu'à coïncider avec $p_1$
 (Figure \ref{decf 1}c, notons que la configuration $\{p_1,p_2\}$ 
n'est alors pas générique), puis le sommet de la droite se
déplace vers le haut et $p_1$ se trouve désormais sur l'ascenseur
 (Figure \ref{decf 1}b). Si nous
continuons à déplacer $p_2$ vers le haut, le sommet de la droite
continue à monter, mais rien d'autre ne se passe.
\end{exe}

\begin{figure}[h]
\begin{center}
\begin{tabular}{ccccccc}
\includegraphics[height=3cm, angle=0]{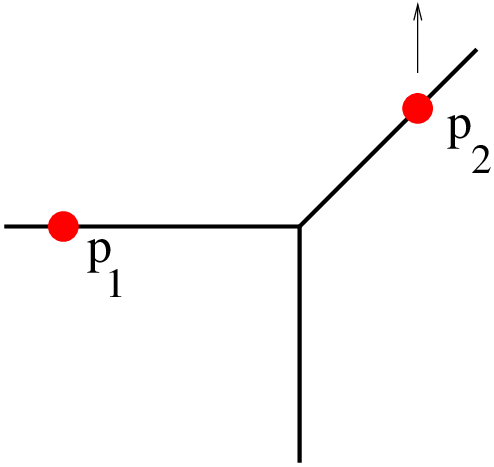}&
\hspace{1ex}  &
\includegraphics[height=3cm, angle=0]{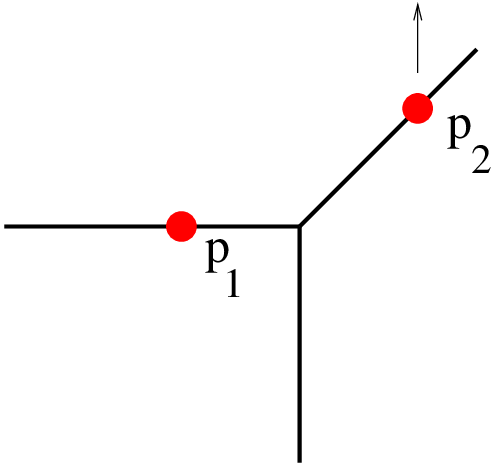}&
\hspace{1ex}  &
 \includegraphics[height=3cm, angle=0]{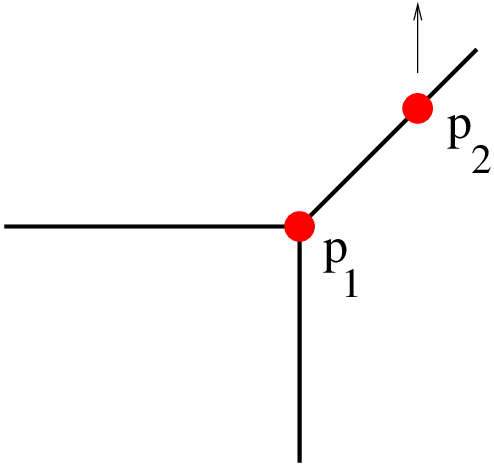}&
\hspace{1ex}  &
\includegraphics[height=2.5cm, angle=0]{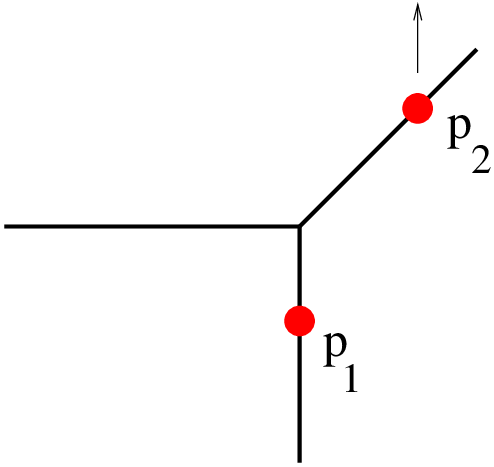}
\\ a) && b) &&c)  && d) 

\end{tabular}
\end{center}
\caption{On monte le point $p_2$}
\label{decf 1}
\end{figure}

\begin{exe}\label{exf2}
Prenons maintenant $d=2$ et  la configuration de la Figure
\ref{decf 2}a. Lorsque nous déplaçons le point $p_5$ vers le haut,
comme à l'Exemple \ref{exf1} nous arrivons à la courbe tropicale de la 
Figure \ref{decf 2}b. Puis, si nous montons le point $p_4$, celui ci se
déplace le long de l'ascenseur, et si nous gardons la distance
entre $p_5$ et $p_4$ très grande, 
rien ne se passe. 
De même, en  montant à présent le point $p_3$, alors nous arrivons au bout
d'un moment à la courbe tropicale de la 
Figure \ref{decf 2}c, et plus rien ne se passe lorsque nous montons les
points $p_2$ puis $p_1$ tout en gardant une distance très grande entre
les points $p_i$.
\end{exe}

\begin{figure}[h]
\begin{center}
\begin{tabular}{ccccc}
\includegraphics[height=5cm, angle=0]{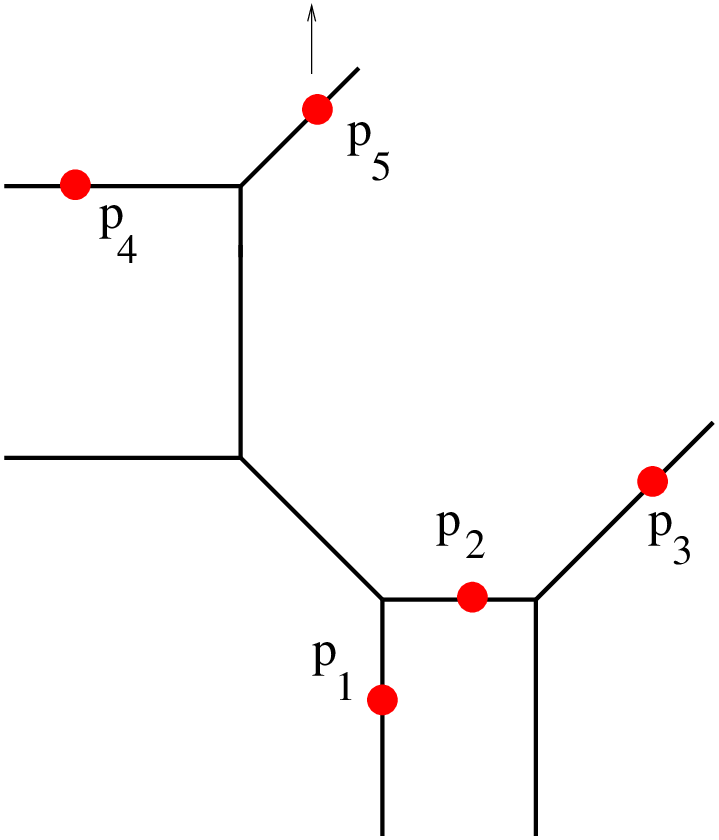}&
\hspace{1ex}  &
\includegraphics[height=5cm, angle=0]{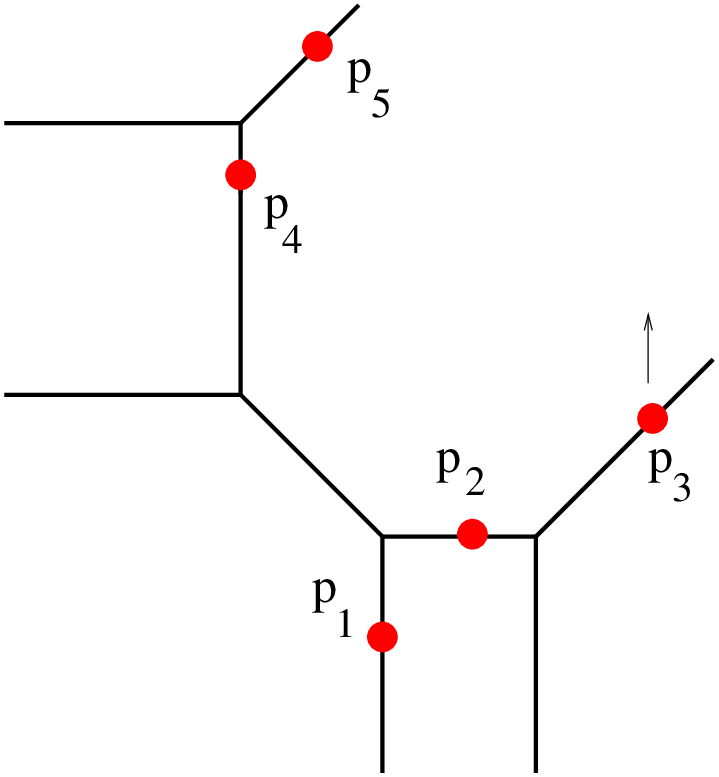}&
\hspace{1ex}  &
\includegraphics[height=5cm, angle=0]{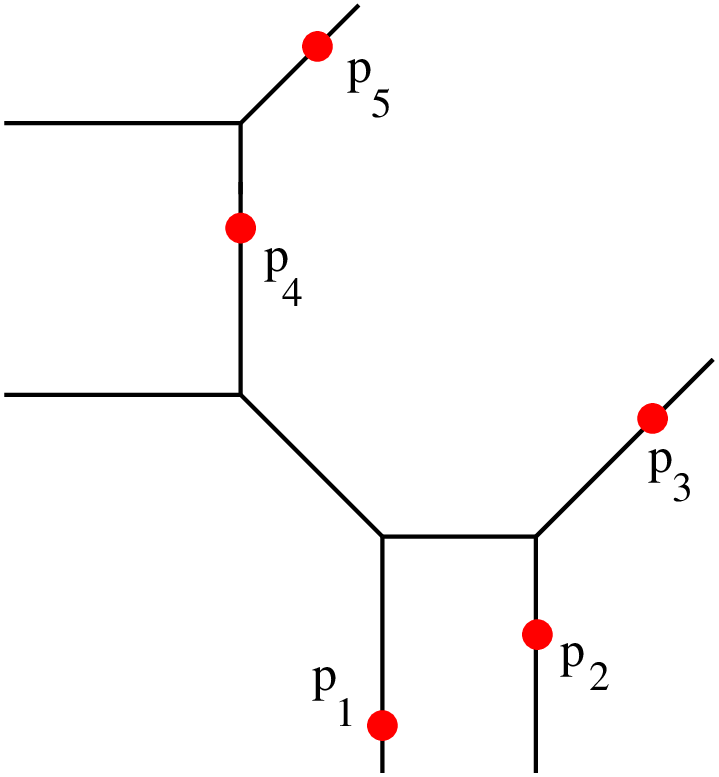}
\\ a) && b)&&c)

\end{tabular}
\end{center}
\caption{On monte les points un par un}
\label{decf 2}
\end{figure}

Dans les deux exemples précédents, on observe le même phénomène. Si
les points $p_i$ sont suffisamment éloignés les uns des autres, alors
chaque étage et chaque ascenseur contient \textit{exactement} un point de
$\omega$. Ainsi, la répartition des points de $\omega$ sur les courbes
tropicales devient très simple! Étudions maintenant ce qui se passe
dans le cas général.

Revenons à une configuration générique $\omega$ quelconque, et soit
$[u;v]$ un segment de $\RR$. La proposition suivante nous dit que si
$\omega\subset [u;v]\times \RR$, alors
toute l'information intéressante à propos des courbes de $\TT \mathcal
C (d,g,\omega)$ se trouve aussi dans le ruban $[u;v]\times \RR$.

\begin{lemma}\label{ruban}
Si tous les points de $\omega$ sont dans $[u;v]\times \RR$, alors tous les
sommets de toutes les courbes tropicales de $\TT \mathcal
C (d,g,\omega)$ sont aussi dans $[u;v]\times \RR$.
\end{lemma}
\begin{proof}
Nous allons raisonner par l'absurde :  s'il
existe une courbe de $\TT \mathcal
C (d,g,\omega)$ ayant un sommet dans le demi plan $\{(x,y)|x<u\}$, nous
allons alors montrer qu'il existe un nombre infini de courbes tropicales
dans $\TT \mathcal
C (d,g,\omega)$ en contradiction avec la Proposition \ref{fini trop}. Le
cas d'une courbe ayant un sommet dans le demi plan $\{(x,y)|x>v\}$ se
traite de la même manière, nous ne le ferons pas ici.

Supposons donc qu'il
existe une courbe tropicale $C$ dans $\TT \mathcal
C (d,g,\omega)$ ayant un sommet dans le demi plan
$\{(x,y)|x<u\}$. Soit $s=(x_0,y_0)$ un sommet  de $C$ dont
l'abscisse est la plus 
petite parmi les sommets de $C$. Par hypothèse, $x_0<u$. 
 Puisqu'il n'existe pas d'autres sommets de
$C$ strictement à gauche de $s$, toutes les arêtes  strictement à
gauche de $s$ sont infinies dans la direction $(-1,0)$,  $s$ est
adjacent à l'une d'entre elles et nous
pouvons choisir  $s$ trivalent. 
De plus, par la
condition d'équilibre, les deux autres arêtes adjacentes à $s$ sont de
vecteurs primitifs sortant $(0,\pm 1)$ et $(1,\alpha)$ (on ne sait rien
sur les poids 
a priori). Puisque aucun point de $\omega$ ne se trouve
à gauche de $s$, nous pouvons alors construire
 facilement un nombre infini de courbes tropicales dans $\TT \mathcal
C (d,g,\omega)$, simplement en translatant l'arête verticale dans la
direction $(-1,0)$. Le cas où l'arête verticale est adjacente à un
autre sommet de $C$ est représenté sur la Figure \ref{deform}.
\end{proof}

\begin{figure}[h]
\begin{center}
\begin{tabular}{c}
\includegraphics[height=6cm, angle=0]{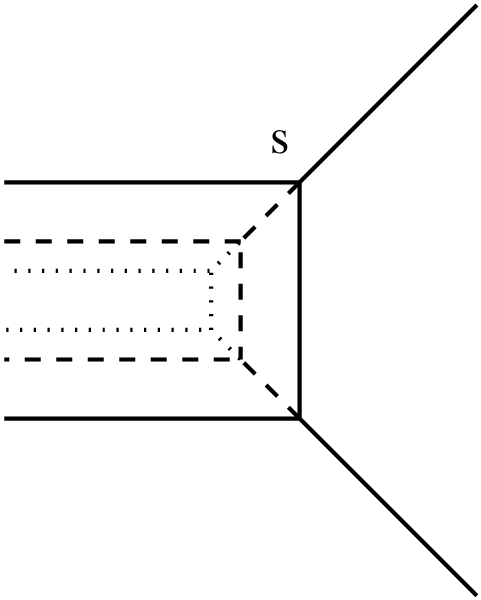}

\end{tabular}
\end{center}
\caption{Une famille infinie de courbes tropicales}
\label{deform}
\end{figure}

Gardons maintenant le segment $[u;v]$ fixe, et éloignons tous les points
les uns des autres comme dans les Exemples \ref{exf1} et
\ref{exf2}. Au bout d'un moment, pour aller d'un point $p_i$ à un
autre en restant sur une courbe tropicale de  $\TT \mathcal
C (d,g,\omega)$, nous sommes obligés de passer par une arête verticale.

\begin{cor}\label{floor dec}
Pour tout segment $[u;v]$, il existe un nombre réel $A$ tel que si
$\omega\subset [u;v]\times\RR$ et 
$|y_i-y_j|>A$ pour tout couple de points distincts $(x_i,y_i)$ et
$(x_j,y_j)$ de $\omega$, alors chaque étage de chaque courbe de $\TT \mathcal
C (d,g,\omega)$ contient exactement un point de $\omega$.
\end{cor}
\begin{proof}
Soit $C$ une courbe de $\TT \mathcal
C (d,g,\omega)$.
Tout d'abord, pour toute configuration $\omega$, chaque étage de $C$
contient au moins un point de $\omega$. Dans le cas contraire, on pourrait alors
translater  un étage vide de point de $\omega$  dans la direction $(0,1)$ et
construire ainsi une famille infinie de courbes dans $\TT \mathcal
C (d,g,\omega)$ en contradiction avec  la Proposition \ref{fini trop}.

En raison de la condition d'équilibre, toute arête $a$ de vecteur
primitif $\vec v$ d'une courbe
tropicale de degré $d$ vérifie $w(a)\vec v=(\alpha,\beta)$ avec
$|\alpha|$ et $|\beta|$ plus petits que $d$. Cela entraîne  qu'il existe un
nombre $M(d)$ qui majore 
la pente de toute arête
non verticale d'une courbe tropicale de degré $d$. 

Prenons maintenant une configuration $\omega\subset [u;v]\times \RR$.
Soient  $C$
une
courbe tropicale de $\TT \mathcal
C (d,g,\omega)$ et $\gamma$ un chemin dans $C$ reliant deux points
$(x_i,y_i)$ et 
$(x_j,y_j)$ de
$\omega$. Supposons que $\gamma$ ne contienne pas d'arête verticale
de $C$. Les deux points sont donc sur le même étage de $C$.
D'après le Lemme \ref{ruban},
tous les sommets de $C$ sont 
dans le ruban $[u;v]\times \RR$.  Or,  d'après la Proposition \ref{max
  som} le nombre de sommets  de $C$
est plus petit que $d^2$, et  les pentes des arêtes non 
verticales de $C$ sont uniformément bornées par $M(d)$, donc la
quantité $|y_i-y_j|$ est bornée par un nombre $A$ qui ne dépend que du
degré $d$. Ainsi, si   $|y_i-y_j|>A$ pour tous
les couples de points de $\omega$, chaque étage de $C$ ne peut
contenir  plus d'un point de
$\omega$.
\end{proof}

Rappelons que si $C$ est une courbe tropicale nodale, alors 
 le nombre de sommets trivalents de $C$ est noté $\sigma$ (Définition
 \ref{ct genre}).

\begin{cor}\label{nb etage}
Si la distance entre les points de $\omega$ est assez grande, alors
chaque courbe de $\TT \mathcal
C (d,g,\omega)$ a exactement $d$ étages et $2d-1+g$ ascenseurs. De
plus, toute arête d'un étage d'une courbe de $\TT \mathcal
C (d,g,\omega)$ est de vecteur directeur primitif $(1,\alpha)$, avec $\alpha$
dans $\ZZ$, et est de poids 1.
\end{cor}
\begin{proof}
Soit $C$ une courbe tropicale dans $\TT \mathcal
C (d,g,\omega)$, soient $e$ le nombre d'étages de $C$ et $v$ son
nombre d'ascenseurs. Tout d'abord, comme $C$ a exactement $d$
arêtes infinies dans la direction $(-1,0)$ et que chaque étage 
contient au moins une de ces arêtes, nous avons $e\le d$. Nous allons
maintenant montrer que $v\le
2d-1+g$. Les extrémités d'un ascenseur sont soit  deux
sommets trivalents de $C$ (si l'ascenseur est borné), soit un sommet
trivalent de $C$ (si
l'ascenseur n'est pas borné). Comme
$C$ a exactement $d$ ascenseurs infinis, nous avons
l'inégalité 
$2v\le \sigma + d$. Mais par définition du genre d'une courbe tropicale
 $\sigma=2g +3d -2$, donc $v\le 2d-1 +g$ comme annoncé. 

Comme $\omega$ est une
configuration générique, un ascenseur de $C$ ne peut pas
contenir deux points de $\omega$. Sinon, cela voudrait dire que deux
points de $\omega$ ont la même abscisse, ce que l'on peut toujours
éviter. Donc, d'après le Corollaire \ref{floor dec}, la courbe $C$ ne peut pas
contenir plus de $3d-1+g$ points de $\omega$. Mais comme par hypothèse
$C$ contient exactement $3d-1+g$ points de $\omega$, nous avons forcément
les égalités $e=d$ et $v=2d-1+g$. 

Le calcul précédent  implique que tout sommet
trivalent de
$C$ est adjacent à un ascenseur. Maintenant, puisque la dernière assertion
du corollaire 
est vraie par hypothèse pour  les arêtes non bornées, 
 la condition d'équilibre nous assure qu'elle est vraie pour toute
 arête de $C$ contenue dans un étage.
\end{proof}

Le Corollaire \ref{nb etage} implique qu'une courbe
tropicale dans $\TT \mathcal
C (d,g,\omega)$ est entièrement déterminée par  la répartition des
points de $\omega$ sur les ascenseurs et les étages 
de $C$! Dans l'énumération des courbes tropicales de $\TT \mathcal
C (d,g,\omega)$, nous pouvons donc finalement oublier les courbes
tropicales, et uniquement nous rappeler comment les ascenseurs
 relient les étages et comment les points de
$\omega$ se répartissent. Le codage de ces informations est appelé un
\textit{diagramme 
  en étages marqué}.

\subsection{Énumération de diagrammes en étages}

Pour  définir les diagrammes en étages, nous avons besoin de
quelques définition abstraites préalables.

\begin{defi}\label{go}
Un graphe $\Gamma$ connexe orienté est la donnée d'un ensemble fini de points
$\Gamma_0$,  d'une liste finie $\Gamma_1$ d'éléments de
$\Gamma_0\times\Gamma_0$ et d'une liste finie $\Gamma_1^{\infty}$
d'éléments de $\Gamma_0$,
tels qu'étant donnés  deux points distincts $t_1$ et $t_2$ de $\Gamma_0$, il
existe une suite d'éléments $s_1=t_1$, $s_2$, $\ldots$, $s_{k-1}$,
$s_k=t_2$ de $\Gamma_0$ vérifiant $(s_i,s_{i+1})\in\Gamma_1$ ou
$(s_{i+1},s_{i})\in\Gamma_1$ pour tout $i$.

 Les éléments de $\Gamma_0$ (resp. $\Gamma_1$,
 $\Gamma_1^{\infty}$) sont appelés les 
sommets (resp. arêtes bornées, arêtes non bornées) de $\Gamma$. 
\end{defi}
Les graphes définis ici sont plus généraux que ceux de la Définition
\ref{rect}. Par exemple, ces  nouveaux graphes sont \textit{abstraits} et
ne vivent pas a priori dans $\RR^2$.
On peut cependant, et c'est ce que nous ferons dans ce texte,  se représenter un
graphe géométriquement : les sommets $s$  sont des
points (en fait des ellipses ici), les arêtes bornées $(s_i,s_j)$ sont
des segments joignant les 
sommets $s_i$ et $s_j$ et orientés de $s_i$ vers $s_j$, et les arêtes $s$
non bornées sont des demi-droites dont l'extrémité est $s$ et orientées vers
$s$.  La condition de la Définition \ref{go} signifie que le dessin
ainsi obtenu
doit être connexe.

\begin{rem}
  $\Gamma_1^{\infty}$ est une liste d'éléments de
$\Gamma_0$, et non pas un sous ensemble de $\Gamma_0$! Certains
sommets de $\Gamma_0$ peuvent très bien se répéter dans
$\Gamma_1^{\infty}$, c'est-à-dire que plusieurs arêtes non bornées
peuvent arriver sur le même sommet. De même, $\Gamma_1$ est une liste
d'éléments de $\Gamma_0\times\Gamma_0$ et pas un sous ensemble de
$\Gamma_0\times \Gamma_0$.
\end{rem}

\begin{exe}
Des exemples de graphes orientés avec 3 sommets sont représentés sur la Figure \ref{graph or}.
\end{exe}

\begin{figure}[h]
\begin{center}
\begin{tabular}{ccccccc}
\includegraphics[height=3cm, angle=0]{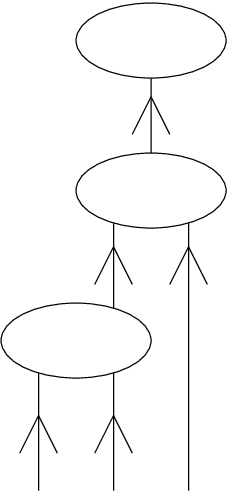}&
\hspace{7ex}  &
\includegraphics[height=3cm, angle=0]{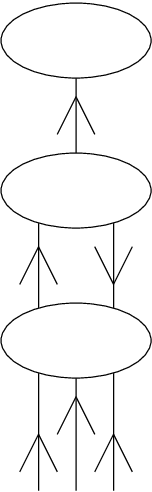}&
\hspace{7ex}  &
\includegraphics[height=3cm, angle=0]{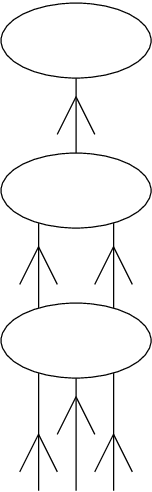}&
\hspace{7ex}  &
\includegraphics[height=3cm, angle=0]{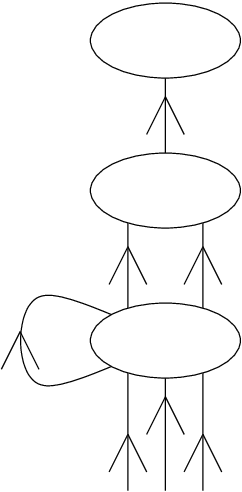}
\\ a) && b)&& c)&& d)

\end{tabular}
\end{center}
\caption{Quelques graphes orientés}
\label{graph or}
\end{figure}

Évidemment, les graphes que nous allons considérer dans la suite vont
représenter comment une courbe tropicale $C$ peut se décomposer en
étages. Les sommets d'un graphe correspondent aux étages de $C$ et les
arêtes correspondent aux ascenseurs.

Certains graphes orientés peuvent avoir des \textit{cycles},
c'est-à-dire qu'il est possible à partir d'un sommet du graphe, de
parcourir 
les arêtes suivant l'orientation et de revenir au même point. Pour la
suite, nous ne devons pas considérer ce type de graphe. 

\begin{defi}
Un graphe orienté est acyclique s'il ne contient pas de suite d'arêtes
$(s_1,s_2)$, $(s_2,s_3)$, $\ldots$, $(s_{k-1},s_k)$, $(s_k,s_1)$.
\end{defi}

\begin{exe}
Les graphes orientés des Figures \ref{graph or}a et c sont
acycliques, mais pas ceux des Figures \ref{graph or}b et d.
\end{exe}

Il ne reste plus qu'à identifier les conditions pour qu'un graphe
acyclique code une décomposition en étages d'une courbe tropicale.

\begin{defi}
Soient $d\ge 1$ et $g\ge 0$. Un diagramme en étages $\mathcal D$ de
degré $d$ et de 
genre $g$ est la donnée d'un graphe acyclique $\Gamma$ et d'une
application appelée poids $w: \Gamma_1\cup \Gamma_1^\infty \to \NN^*$ tels que
\begin{itemize}
\item[$\bullet$] le graphe $\Gamma$ a $d$ sommets, $d-1+g$ arêtes
  bornées et $d$ arêtes non bornées,
 
\item[$\bullet$] pour tout sommet $s$ de $\Gamma$, si $a_1,\ldots,a_k$
  (resp. $b_1,\ldots,b_l$)
  sont les arêtes entrantes (resp. sortantes) adjacentes à $s$, alors
$$\sum_{i=1}^k w(a_i) - \sum_{i=1}^l w(b_i)=1 $$
\end{itemize}

\end{defi}

\begin{rem}
Cette définition implique que toutes les arêtes non bornées d'un
diagramme en étages sont de poids 1.
\end{rem}

Afin d'alléger un peu les notations, nous confondrons dans la suite un
diagramme en étage $\mathcal D$ et le graphe acyclique $\Gamma$ sous-jacent.
Comme pour les courbes tropicale on ne marque le poids d'une arête d'un
diagramme en étages uniquement lorsque celui ci est au moins 2. De plus, comme
les diagrammes en étages sont acycliques, on les oriente implicitement
dans
les dessins du
bas vers le haut.

\begin{exe}
Les graphes orientés des Figures \ref{graph or}a et c
 sont des diagrammes en
étages. Leur  degré est 3 et leur genre est respectivement 0 et 1. D'autres
diagrammes en étages sont représentés à la Figure 
\ref{FD1} (l'orientation est implicite du bas vers le haut).

\end{exe}
\begin{figure}[h]
\begin{center}
\begin{tabular}{ccccccccc}
\includegraphics[height=1.3cm, angle=0]{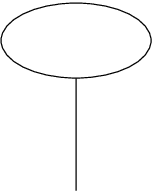}&
\hspace{1ex}  &
\includegraphics[height=2cm, angle=0]{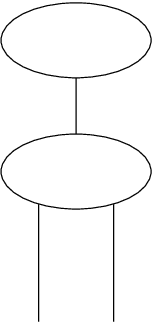}&
\hspace{1ex}  &
\includegraphics[height=3cm, angle=0]{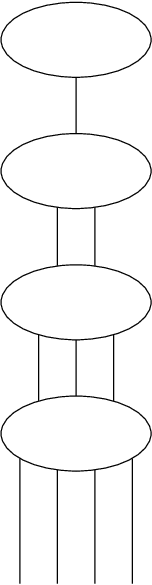}&
\hspace{1ex}  &
\includegraphics[height=3cm, angle=0]{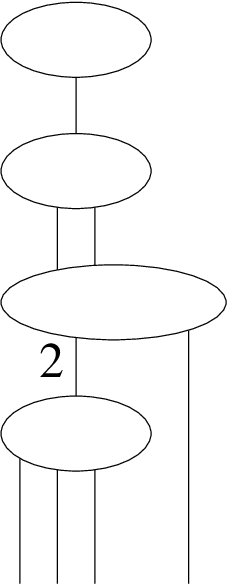}&
\hspace{1ex}  &
\includegraphics[height=3cm, angle=0]{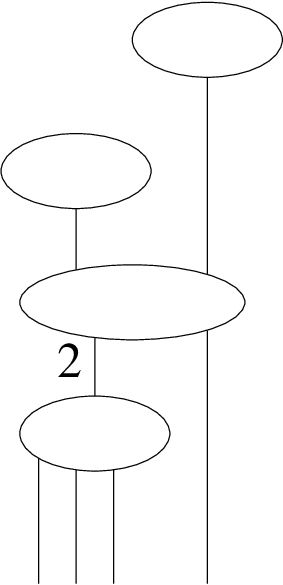}
\\ a)$d=1$, $g=0$ && b) $d=2$, $g=0$ &&c) $d=4$, $g=3 $ && d) $d=4$,
$g=1 $ &&e) $d=4$, $g=0 $

\end{tabular}
\end{center}
\caption{Quelques diagrammes en étages}
\label{FD1}
\end{figure}

Pour terminer, il ne nous reste plus qu'à coder la répartition des
points de $\omega$ sur un diagramme en étages. L'orientation d'un
diagramme en étages $\mathcal D$ nous induit un ordre partiel sur celui
ci : étant donné deux points  $p$ et $q$ de $\mathcal D$, on dit
que $q$ est plus grand que $p$ si on peut aller de $p$ à $q$ en
``suivant les flèches''. Plus précisément, s'il existe un chemin orienté
dans $\mathcal D$ allant de $p$ à $q$.

\begin{exe}
Sur la Figure \ref{FD2}a, le point 2 est plus grand que 1. Sur la
Figure \ref{FD2}b, le point 1 est plus grand que 2. Sur la Figure
\ref{FD2}e, les points 1 et 5 ne sont pas comparables.
\end{exe}

Nous dirons qu'une application  $m: \{1,\ldots,3d-1+g\}\to \mathcal D$
est croissante si $m(i)\ge m(j)$ implique que $i\ge j$.

\begin{defi}
Un diagramme en étages  marqué de degré $d$ et de genre $g$ est la
donnée d'un diagramme en étages $\mathcal D$ de degré $d$ et de genre $g$
munit d'une bijection croissante
$m:\{1,\ldots,3d-1+g\}\to \mathcal D$.
\end{defi}

\begin{exe}
La Figure \ref{FD2}  représente des diagrammes en étages
munis d'une bijection $m:\{1,\ldots,3d-1+g\}\to \mathcal D$. 
La Figure \ref{FD2}b n'est pas un diagramme en étages marqué car
l'application $m$
n'est pas croissante. Les Figures \ref{FD2}a, c, d, e et f sont
des diagrammes en étages marqués.
\end{exe}

\begin{figure}[h]
\begin{center}
\begin{tabular}{ccccccccccc}
\includegraphics[height=1.3cm, angle=0]{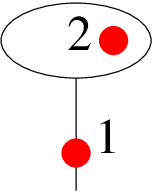}&
\hspace{5ex}  &
\includegraphics[height=1.3cm, angle=0]{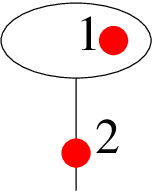}&
\hspace{5ex}  &
\includegraphics[height=2cm, angle=0]{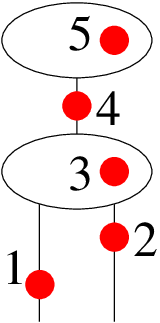}&
\hspace{5ex}  &
\includegraphics[height=2cm, angle=0]{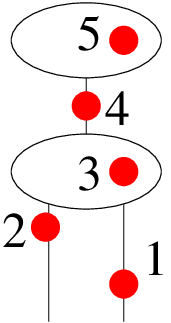}&
\hspace{5ex}  &
\includegraphics[height=3cm, angle=0]{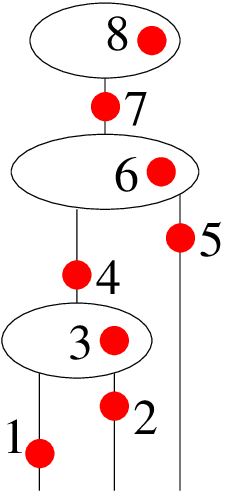}&
\hspace{5ex}  &
\includegraphics[height=3cm, angle=0]{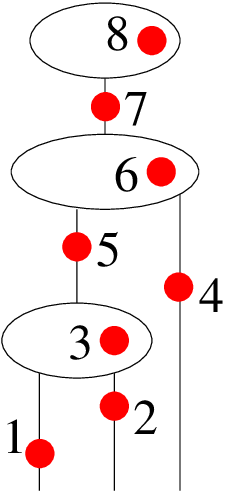}
\\ a) && b)&&c)&& d)&&e)&&f)

\end{tabular}
\end{center}
\caption{Marquages}
\label{FD2}
\end{figure}

Enfin, pour être vraiment rigoureux, il faut considérer  les
diagrammes en étages marqués \textit{à isomorphisme près}.

\begin{defi}
Deux diagrammes en étages marqués $(\mathcal D,m)$ et $(\mathcal D',m')$
sont dits isomorphes s'il existe une bijection $\phi: \Gamma_0\to
\Gamma_0'$ telle que
\begin{itemize}
\item[$\bullet$]l'application $(s_1,s_2)\mapsto (\phi(s_1),\phi(s_2))$
  est une bijection de $\Gamma_1$ dans $\Gamma_1'$,
\item[$\bullet$]l'application $s\mapsto \phi(s)$
  est une bijection de $\Gamma_1^\infty$ dans ${\Gamma_1^{\infty}}'$,
\item[$\bullet$] $w= w'\circ \phi$,

\item[$\bullet$]  $m= m'\circ \phi$.

\end{itemize}

\end{defi}
Malgré son apparente aridité, cette définition sert juste à identifier
les diagrammes en étages marqués que l'on a intuitivement envie d'identifier.

\begin{exe}
Les diagrammes en étages marqués des Figures \ref{FD2}c et d sont
isomorphes, mais pas ceux  des Figures \ref{FD2}e et f.
\end{exe}

On note $\mathcal E(d,g) $ l'ensemble des classes d'équivalence des
diagrammes en 
étages marqués de degré $d$ et de genre $g$. Dans la suite, nous
confondrons un diagramme en étages marqué et sa 
classe d'équivalence. Afin d'alléger les notations, nous noterons aussi
simplement par $\mathcal D$ un diagramme en étages marqué lorsque cela
ne prêtera pas à confusion.

De même que dans le cas des courbes tropicales, il faut associer une
multiplicité réelle et complexe à un diagramme en étages marqué.
\begin{defi}
Soit $\mathcal D$ un diagramme en étages marqué. Alors la multiplicité
complexe de  $\mathcal D$, notée $\mu_\CC(\mathcal D)$, est définie
par
$$\mu_\CC(\mathcal D)=\prod_{a\textrm{ arête de }\mathcal D} w(a)^2$$

La multiplicité
réelle de  $\mathcal D$, notée $\mu_\RR(\mathcal D)$, est définie
par
$$\mu_\RR(\mathcal D)= \mu_\CC(\mathcal D)\ mod\ 2$$
\end{defi}
C'est-à-dire que $\mu_\RR(\mathcal D)=0$ si $\mu_\CC(\mathcal D)$ est
pair, et $\mu_\RR(\mathcal D)=1$ sinon. Notons que les multiplicités 
d'un diagramme en étages marqué ne dépendent en fait que du diagramme en
étages sous-jacent et sont toujours positives ou nulles.

\begin{thm}[Brugallé, Mikhalkin, \cite{Br7}]\label{FD}
Pour tous $d\ge 1$ et  $g\ge 0$, on a
$$N(d,g) =\sum_{\mathcal D\in \mathcal E(d,g) }\mu_\CC(\mathcal D)$$ 
et 
$$W(d) =\sum_{\mathcal D\in \mathcal E(d,0) }\mu_\RR(\mathcal D)$$ 
\end{thm}
\begin{proof}
Soit $\omega$ une configuration générique de $3d-1+g$ points de
$\RR^2$. 
D'après les Corollaires \ref{floor dec} et \ref{nb etage}, si les
distances entre les points de $\omega$ sont assez grandes
il y a une
bijection naturelle entre les éléments de $ \mathcal E(d,g)$ et les courbes
tropicales de  $\TT\mathcal C(d,g,\omega)$. Il reste  à déterminer la
multiplicité d'une courbe tropicale décomposée en étages. Soient $C$
une courbe tropicale de $\TT\mathcal C(d,g,\omega)$
et $s$
un
sommet trivalent de $C$. D'après le Corollaire \ref{nb etage}, le
sommet  $s$ se situe sur un étage de
$C$ et est 
adjacent à un ascenseur de poids $w_1$ et à une arête $a$ de
poids 1 et de 
vecteur primitif sortant $(1,\alpha)$. On a donc $\mu_\CC(s)=w_1$. Comme
tous les ascenseurs bornés sont adjacents à deux
sommets, et que tous les ascenseurs non bornés sont
de poids 1, on a bien que $\mu_\CC(C)$ est égale au produit des carrés
des poids
des ascenseurs.

Si un des ascenseurs de $C$ est de poids pair, alors
$\mu_\RR(C)=0$. Sinon, pour la même raison que dans le calcul de
$\mu_\CC(C)$, la multiplicité réelle $\mu_\RR(C)$ est positive.

Le théorème découle maintenant des Théorèmes  \ref{corr 1} et
\ref{corr 2}.
\end{proof}

\begin{exe}
Les seuls diagrammes en étages marqués en degrés 1 et 2 sont ceux des
Figures \ref{FD2}a et c. Comme ils sont de multiplicité complexe 1,
nous retrouvons bien $N(1,0)=N(2,0)=W(1)=W(2)=1$.
\end{exe}

\begin{exe}
Tous les diagrammes en étages de degré 3 sont représentés à la Figure
\ref{FD3}.
Celui de la Figure \ref{FD3}a est de genre 1 et il n'en existe qu'un
marquage. Nous retrouvons bien 
$N(3,1)=1$. Tous les autres diagrammes en 
étages de  la Figure \ref{FD3} sont de genre 0. Il existe
respectivement 1, 5 et 3 marquages possibles pour les diagrammes des
Figures \ref{FD3}b, c et d. Nous retrouvons bien $N(3,0)=12$ et $W(3)=8$.
\end{exe}

\begin{figure}[h]
\begin{center}
\begin{tabular}{ccccccc}
\includegraphics[height=3cm, angle=0]{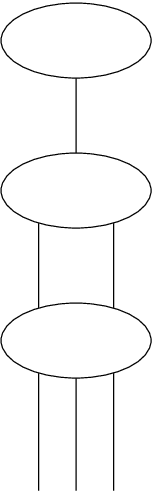}&
\hspace{5ex}  &
\includegraphics[height=3cm, angle=0]{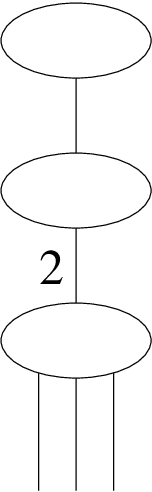}&
\hspace{5ex}  &
\includegraphics[height=3cm, angle=0]{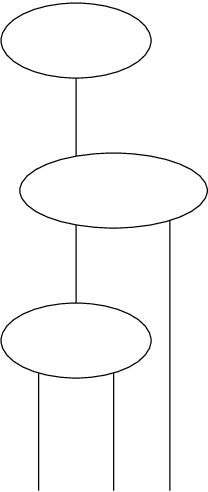}&
\hspace{5ex}  &
\includegraphics[height=2.5cm, angle=0]{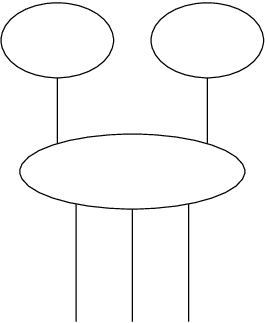}
\\ a) $\mu_\CC=1$ && b) $\mu_\CC=4$, $\mu_\RR=0$
  &&c) $\mu_\CC=1$, $\mu_\RR=1$&& d) $\mu_\CC=1$, $\mu_\RR=1$

\end{tabular}
\end{center}
\caption{Diagrammes en étages de degré 3}
\label{FD3}
\end{figure}

\begin{exo}
Montrer que pour tout $d\ge 1$, le diagramme en étages de la Figure
\ref{FD4} est l'unique diagramme en étages de degré $d$ et de genre
$\frac{(d-1)(d-2)}{2}$ et qu'il en existe un unique marquage. En
déduire
$$N(d, \frac{(d-1)(d-2)}{2})=1$$
\end{exo}

\begin{figure}[h]
\begin{center}
\begin{tabular}{c}
\includegraphics[height=6cm, angle=0]{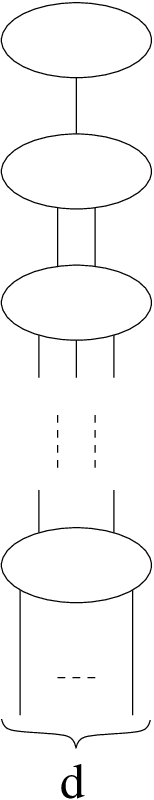}
\end{tabular}
\end{center}
\caption{Diagramme en étages de degré $d$ de genre maximal}
\label{FD4}
\end{figure}

\begin{exo}
Calculer les nombres $N(4,2)=27$, $N(4,1)=225$, $N(4,0)=620$ et
$W(4)=240$ à l'aide des diagrammes en étages.
\end{exo}

\section{Applications}\label{appl}

Nous allons maintenant appliquer le Théorème \ref{FD} pour démontrer
de manière combinatoire quelques résultats sur les nombres $N(d,g)$ et $W(d)$.

\subsection{En géométrie énumérative complexe}\label{appl cplx}

Démontrons la formule non triviale la plus simple sur les nombres $N(d,g)$.

\begin{prop}
Pour tout $d\ge 3$, on a
$$N(d,\frac{(d-1)(d-2)}{2}-1)= 3(d-1)^2 $$
\end{prop}
\begin{proof}
Le point de départ est l'unique diagramme en étages de degré $d$ et de
genre maximal. Pour baisser le genre de ce diagramme en étages de 1,
nous avons deux possibilités représentées sur la Figure \ref{baisse
  genre}. De plus, il n'est pas difficile de voir que tous les
diagrammes en étages de genre $\frac{(d-1)(d-2)}{2}-1$ s'obtiennent de
cette manière. Il ne nous reste plus qu'à compter de combien de
manières nous pouvons 
marquer ces nouveaux diagrammes en étages. 

Nous avons $i-1$ possibilités pour le marquage dans le cas de la
Figure \ref{baisse 
  genre}a, et  $2i+1$
possibilités  dans le cas de la Figure \ref{baisse 
  genre}b. De plus, dans les deux cas $i$  varie de 1 à $d-1$ et
il ne nous reste plus qu'à faire un petit calcul :
$$ \begin{array}{lll}
N(d,\frac{(d-1)(d-2)}{2}-1)&=& \sum_{i=1}^{d-1}4(i-1)
+\sum_{i=1}^{d-1}(2i+1)
\\
\\ &=& 6\sum_{i=1}^{d-1}(i-1) + 3(d-1) 
\\
\\ &=& 3(d-1)(d-2) + 3(d-1) 
\\
\\&=& 3(d-1)^2
\end{array}$$
\end{proof}

\begin{figure}[h]
\begin{center}
\begin{tabular}{ccc}
\includegraphics[height=2cm, angle=0]{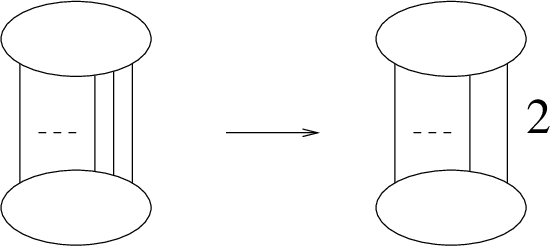}&
\hspace{10ex}  &
\includegraphics[height=1.5cm, angle=0]{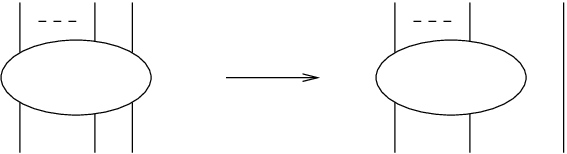}
\\ a) $i$ arêtes $\to$ $i-1$ arêtes &&b)  $i$ arêtes sortantes $\to$
$i-1$ arêtes sortantes
\\ $\mu_\CC=4$ && $\mu_\CC=1$
\end{tabular}
\end{center}
\caption{Diminuer de 1 le genre du diagramme en étages de genre maximal}
\label{baisse genre}
\end{figure}

\begin{exo}\label{exo CH}
Il est possible de démontrer la formule de Caporaso et Harris grâce
aux diagrammes en étages de la manière suivante.
Tout d'abord, il est nécessaire de considérer des diagrammes en étages
plus généraux. Les arêtes infinies peuvent être de n'importe quel
poids, et l'application $m$ peut envoyer des points à l'infini. Ensuite,
il faut associer une multiplicité complexe à de tels diagrammes en
étages marqués de
telle sorte qu'en envoyant un par un les points à l'infini, une
formule de récurrence apparaisse. 

Pour s'inspirer, on pourra se reporter à la preuve tropicale  de la
formule de Caporaso et Harris par
Gathmann et Markwig dans \cite{GM2}. 
\end{exo}

\subsection{Comportement des invariants de  Welschinger}\label{appl reel}

De par leur définition,  les nombres
$N(d,0)$ et $W(d)$ sont trivialement égaux modulo 2. En fait,
Mikhalkin a observé 
que cette égalité est 
vraie modulo 4.

\begin{prop}[Mikhalkin]
Pour tout $d\ge 1$ on a 
$$W(d)= N(d,0) \ mod \ 4 $$
\end{prop}
\begin{proof}
C'est la conséquence immédiate de  
$$a=b \ mod \ 2\Longrightarrow  a^2=b^2\ mod \ 4$$
appliqué aux poids des arêtes d'un diagramme en étages et $b=0$ ou $1$.
\end{proof}

Terminons par quelques résultats sur  la suite $(W(d))_{d\ge 1}$.

\begin{thm}[Itenberg, Kharlamov, Shustin \cite{IKS1} \cite{IKS2}]\label{thm:IKS}
La suite $(W(d))_{d\ge 1}$ des invariants de Welschinger vérifie les
propriétés suivantes :

\begin{itemize}
\item[$\bullet$] c'est une suite de nombres strictement positifs,

\item[$\bullet$] la suite est croissante, et strictement croissante à
  partir de $d=2$,

\item[$\bullet$] $\ln  W(d) \sim \ln N(d,0) \sim 3d \ln  d$ lorsque $d$
  tend vers l'infini. 

\end{itemize}
\end{thm}
\begin{proof}
Comme la multiplicité réelle d'un diagramme en étages est positive ou
nulle, les nombres $W(d)$ sont tous positifs ou nuls.  De plus, comme
$W(1)=1$, la positivité des nombres $W(d)$ découlera de la croissance
de la suite  $(W(d))_{d\ge 1}$.

Soit $(\mathcal D_0,m_0)$ un diagramme en étages marqué de degré $d$. Pour
rendre notre raisonnement  plus transparent,
supposons que le marquage $m_0$  va non pas de
$\{1,\ldots,3d-1\}$ dans $\mathcal D_0$, mais de  $\{4,\ldots,3d+2\}$
dans $\mathcal D_0$. Notons que le point $4$ est forcément envoyé sur
une arête infinie de $\mathcal D_0$. À partir de $\mathcal D_0$, nous pouvons
construire un nouveau diagramme en étages marqué $\mathcal D$ de degré
$d+1$ comme
indiqué à la Figure \ref{FDW}a. Les multiplicités réelles de
$\mathcal D_0$ et $\mathcal D$ sont les mêmes et deux diagrammes en étages marqués
différents $\mathcal D_0$  et $\mathcal D_0'$ donnent évidemment deux
diagrammes en étages marqués 
différents $\mathcal D$  et $\mathcal D'$. Nous avons donc montré que
$W(d+1)\ge W(d)$. De plus, si $d\ge 2$, alors le point $5$ est aussi
forcément sur une arête infinie de $\mathcal D_0$. Nous voyons qu'il existe
alors un diagramme en étages marqué de degré $d+1$ qui n'est pas obtenu
à partir d'un diagramme en étages marqué de degré $d$ (voir les Figures
\ref{FDW}b et c), et donc $W(d+1)>W(d)$.

Regardons maintenant l'asymptotique logarithmique de la suite
$(W(d))_{d\ge 1}$. 
Soit $(\mathcal D_k)$ la suite de diagrammes en étages construite de
la manière suivante : $\mathcal D_1$ est le diagramme en étages de
degré 1 et $\mathcal D_k$ est obtenu à partir de $\mathcal D_{k-1}$ en
recollant à chaque arête infinie de $\mathcal D_{k-1}$ le morceau
représenté à la Figure \ref{FDWa}a. Les diagrammes en étages
$\mathcal D_1$, $\mathcal D_2$, $\mathcal D_3$ et $\mathcal D_4$ sont
représentés sur les Figures  \ref{FDWa}b, c, d et e. Le
diagramme en étages $\mathcal D_k$ est de degré $2^{k-1}$ et est de
multiplicité 1. Notons $\nu(\mathcal D_k)$ le nombre de marquages
possibles de $\mathcal D_k$. Si nous oublions les 3 points les plus haut,
nous voyons apparaître la relation
$$\forall k \ge 2 \ \ \ \nu(\mathcal D_k)=\frac{\nu(\mathcal
  D_{k-1})^2}{2}C_{3\times 2^{k-1}-4}^{3\times 2^{k-2}-2}  $$
et il ne nous reste plus qu'à calculer
$$ \begin{array}{lll}
\nu(\mathcal D_k) &=& \frac{\nu(\mathcal
  D_{k-1})^2(3\times 2^{k-1}-4)!}{2\left((3\times 2^{k-2}-2)!\right)^2}
\\
\\ &=& \frac{(3\times 2^{k-1}-4)!}{2^{2^{k-1}-1}}\prod_{i=2}^{k}\frac{1}{\left((3\times 2^{k-i}-2)(3\times 2^{k-i}-3)\right)^{2^{i-1}}}

\end{array}$$
Nous avons donc l'encadrement
$$\frac{(3\times 2^{k-1}-4)!}{2^{2^k }\prod_{i=1}^{k}(3\times 2^{k-i})^{2^{i}}} \le \nu(\mathcal D_k) \le  (3\times 2^{k-1}-4)! $$
D'après la formule de Stirling, nous avons l'équivalence 
 $\ln k!\sim k\ln k$ et donc $\ln (3\times 2^{k-1})! \sim 3\times
2^{k-1}\ln (3\times 2^{k-1})$. Nous avons aussi
$$\begin{array}{lll}
\ln(2^{2^k }\prod_{i=1}^{k}(3\times 2^{k-i})^{2^{i}})&=& 2^k\ln 2 +
(2^{k+1}-1)\ln 3 + \ln 2\ \sum _{i=1}^{k} 2^i(k-i) 
\\
\\ &\le & 2^k\ln 2 +
(2^{k+1}-1)\ln 3 + k \ln 2\ \sum _{i=1}^{k} 2^i 
\\
\\ &\le & 2^k\ln 2 +
(2^{k+1}-1)\ln 3 + k(2^{k+1}-1) \ln 2 
\\
\\ & = & o\left( 3\times
2^{k-1}\ln (3\times 2^{k-1})\right)
\end{array}$$ 
Nous en déduisons immédiatement que $\ln \nu(\mathcal D_k)\sim 3\times
2^{k-1}\ln (3\times 2^{k-1})$. Comme $\ln \nu(\mathcal D_k)\le \ln W(2^{k-1})
\le \ln N(2^{k-1},0)$ et que la suite $(W(d))_{d\ge 1}$ est croissante,
le théorème découle maintenant de la 
Proposition \ref{asy N}
\end{proof}

\begin{figure}[h]
\begin{center}
\begin{tabular}{ccccc}
\includegraphics[height=4cm, angle=0]{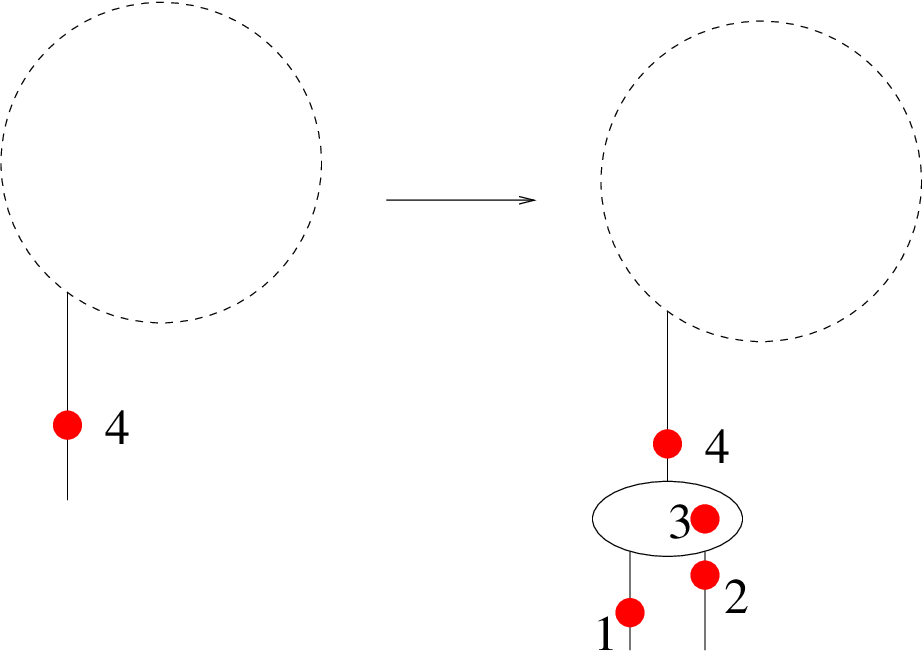}&
\hspace{0ex}  &
\includegraphics[height=4cm, angle=0]{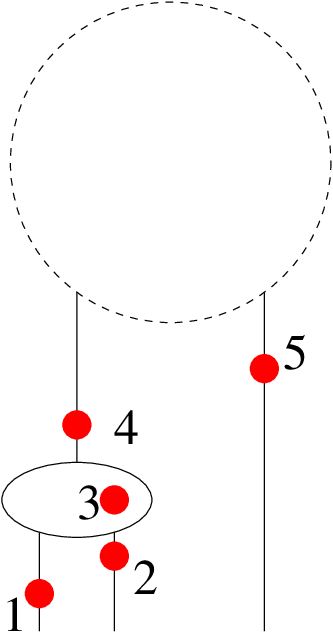}&
\hspace{0ex}  &
\includegraphics[height=4cm, angle=0]{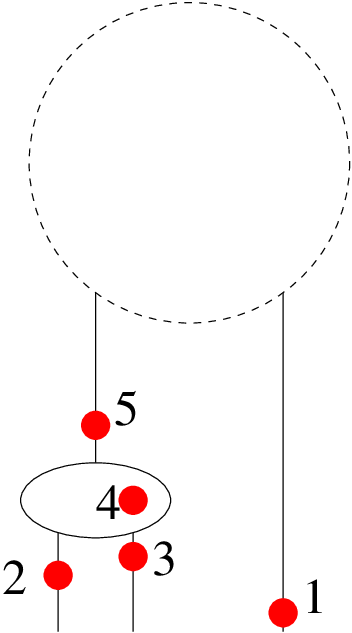}
\\ a) Du degré $d$ au degré $d+1$ &&b) Construit à partir && 
c) Non construit à partir
\\ &&du degré
$d$ && du degré $d$
\end{tabular}
\end{center}
\caption{Croissance des nombres $W(d)$}
\label{FDW}
\end{figure}

\begin{figure}[h]
\begin{center}
\begin{tabular}{ccccccccc}
\includegraphics[height=1.5cm, angle=0]{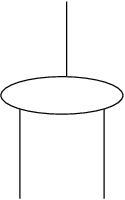}&
\hspace{2ex}  &
\includegraphics[height=1.3cm, angle=0]{Figures/FD1.eps}&
\hspace{2ex}  &
\includegraphics[height=2cm, angle=0]{Figures/FD2.eps}&
\hspace{2ex}  &
\includegraphics[height=3cm, angle=0]{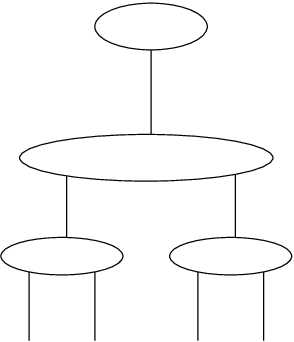}&
\hspace{2ex}  &
\includegraphics[height=4cm, angle=0]{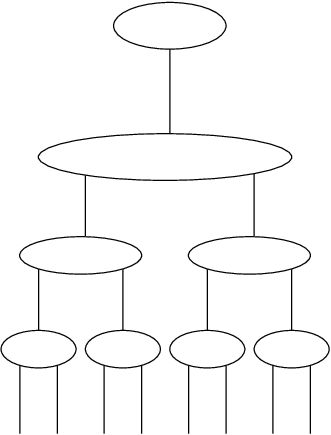}
\\ a) &&b) $\mathcal D_1$ && 
c) $\mathcal D_2$&&d) $\mathcal D_3$&& e) $\mathcal D_4$
\end{tabular}
\end{center}
\caption{Étude asymptotique des nombres $W(d)$}
\label{FDWa}
\end{figure}

\begin{cor}
Par $3d-1$ points de $\RR^2$ passe toujours une courbe algébrique
réelle rationnelle de degré $d$.
\end{cor}

\begin{exo}\label{exo CH reel}
Il est possible de démontrer une formule de type Caporaso et Harris
pour les nombres $W(d)$ grâce aux diagrammes en étages. 
Cette formule a été démontrée pour la première fois par Itenberg,
Kharlamov et Shustin dans \cite{IKS3}.
La méthode est
la même qu'à l'Exercice 
\ref{exo CH}.
Il faut  considérer les mêmes diagrammes en étages marqués qu'à l'Exercice 
\ref{exo CH} puis leur attribuer une multiplicité réelle de
telles sorte qu'en envoyant un par un les points à l'infini, une
nouvelle formule de récurrence apparaisse. 

Pour s'inspirer, on pourra se reporter à l'article correspondant d'Itenberg,
Kharlamov et Shustin.

\end{exo}

\section{Exercices dont je ne connais pas la solution}\label{exos}

Pour finir, voici trois exercices dont j'aimerais connaître la solution.
N'hésitez pas à m'écrire si vous résolvez l'un d'entre eux!

\begin{exo}
Montrer, en utilisant les digrammes en étages, que pour tout $d$
$$N(d,\frac{(d-1)(d-2)}{2}-2)= \frac{3}{2}(d-1)(d-2)(3d^2-3d-11) $$
\end{exo}

\begin{exo}
Montrer la Formule de Kontsevich en utilisant les diagrammes en étages.
\end{exo}

\begin{exo}
En utilisant la Formule de Kontsevich, on voit facilement que $2^{\left[\frac{d-1}{2}\right]}$
divise $N(d,0)$, où $[x]$ désigne la partie entière de $x$. Qu'en
est-il des nombres $W(d)$?
\end{exo}

\section{Pour aller  plus loin}

Il est possible  de généraliser 
dans de nombreuses directions les problèmes énumératifs que nous avons
discuté
dans ce texte. Par exemple, on peut compter les 
courbes algébriques dans $\CC^n$, avec $n\ge 3$, passant par une certaine
configuration de points donnée. On peut même
remplacer les points par des droites, des plans, ...
Par exemple,  combien de droites dans
$\CC^3$ intersectent 4 autres droites données? Combien de coniques
dans $\CC^3$ passent par 2 points et intersectent 4 droites données?

Comme dans le cas du plan, il existe des formules récursives reliant
certains de ces nombres. De telles formules ont été obtenues
 dans
les années 90 par 
Kontsevich (voir \cite{KonMan1}) et Vakil (voir \cite{Vak1}). Pour une
introduction à ce sujet, nous renvoyons au très bon livre \cite{Vain2}.
Ces mêmes problèmes énumératifs se posent aussi en géométrie
tropicale, et 
 les diagrammes en étages s'avère
encore une fois utiles (voir \cite{Br7}, \cite{Br6}).

On peut aussi ajouter des conditions de tangences, c'est-à-dire
regarder les courbes passant par des points et tangentes à des
droites. Il existe des formules récursives  en
géométrie complexe pour les courbes de genre plus petit que 2,
obtenues par Pandharipande, Vakil, puis Graber, Kock 
et Pandharipande (voir \cite{Pan1}).
Il est possible de définir une notion de tangence entre deux courbes
tropicales, et dans un travail en cours avec Bertrand et Mikhalkin
(voir le futur \cite{Br9}),
nous cherchons à utiliser les diagrammes en étages pour calculer ces
nombres. 

Les invariants de Welschinger existent aussi dans des contextes plus
généraux. Dans son article original \cite{Wel1}, Welschinger ne les
a pas uniquement 
défini  pour des courbes 
rationnelles 
passant par des configurations de points réels, mais par
des \textit{configurations réelles de points}. C'est-à-dire que l'on
se fixe des points réels et  des paires de points complexes conjugués. Plus
tard, Welschinger a aussi défini des invariants pour les courbes
rationnelles dans $\RR^3$ (voir \cite{Wel2}). Encore une fois, la
géométrie tropicale et les
diagrammes en étages
sont des moyens pratiques de calculer ces invariants (voir par exemple
\cite{Br7},
\cite{Br6}, \cite{Br8}, \cite{IKS2}, \cite{IKS3}, \cite{Sh8}).

Pour terminer, précisons que la géométrie tropicale n'est pas le seul
moyen connu pour 
calculer les invariants de Welschinger. En utilisant la géométrie
symplectique, Welschinger (\cite{Wel4}) et Solomon (\cite{Sol1}) ont
pu aussi calculer certains de ces invariants.

\section*{Appendice: Preuve du Théorème \ref{thm:IKS}, errata (en collaboration avec
  Gurvan Mével)}

La preuve de l'asymptotique logarithmique du Théorème \ref{thm:IKS}
 donnée plus haut n'est pas satisfaisante car elle ne traite
 en fait que la sous-suite $(W(2^k))_{k\ge 0}$. Nous corrigeons ce
 défaut dans ce qui suit.

\begin{proof}[Preuve de l'asymptotique logarithmique de la suite
    $(W(d))_{d\ge 1}$.] 
  On considère $(A_d)_{d\ge 1}$
  la suite d'ensembles de diagrammes en étages de degré $d$
construite récursivement de la façon
suivante :
\begin{itemize}
\item $A_1$ est constitué de l’unique diagramme en étages de degré 1,
\item les diagrammes de $A_{d+1}$ sont obtenus à partir de ceux de
  $A_d=\{\mathcal D_1,\ldots, \mathcal D_k\}$
  en recollant le tripode représenté
à la Figure \ref{FDWa}a
  à une arête infinie de $\mathcal D_i$, et ce pour
  chaque $\mathcal D_i$ et chaque arête infinie de $\mathcal D_i$. 
\end{itemize}
L'ensemble $A_2$ est constitué de l’unique diagramme en étages de
degré 2, l'ensemble $A_3$ est constitué du diagramme en étages
représenté à la Figure \ref{FD3}c,  et les diagrammes en étages
constituant l'ensemble $A_4$ sont représentés à la Figure \ref{FDA4}.
\begin{figure}[h]
\begin{center}
\begin{tabular}{ccc}
\includegraphics[height=3cm, angle=0]{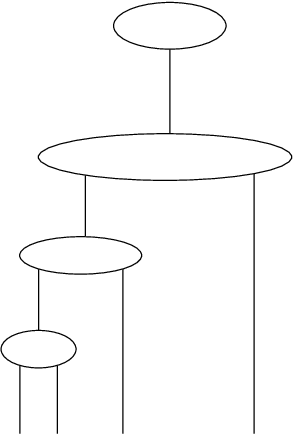}&
\hspace{2ex}  &
\includegraphics[height=3cm, angle=0]{Figures/FDa2.eps}
\end{tabular}
\end{center}
\caption{Diagrammes en étages constituant l'ensemble $A_4$}
\label{FDA4}
\end{figure}
On note $\lambda_d$
le nombre de marquages possibles pour l’ensemble des diagrammes de
$A_d$.
Si nous oublions les 3 points les plus haut,
nous voyons apparaître la relation
\[
\forall d\ge 2 \qquad
\lambda_d=\frac{1}{2}\sum_{\begin{array}{c}d_1+d_2=d\\ d_1,d_2\ge
    1\end{array}}
{3d-4\choose 3d_1-2} \lambda_{d_1}\lambda_{d_2}.
\]
De la même manière que nous avons montré la croissance de la suite
$(W(d))_{d\ge 1}$, nous obtenons que
la suite $(\lambda_d)_{d\ge 1}$ est aussi croissante. 
En ne considérant que le terme de droite pour $d_1=\left[\frac{d}{2}\right]$
dans l'égalité ci-dessus,
nous obtenons alors l'inégalité
\[
\lambda_d \ \ge \ \frac{1}{2}
       {3d-4\choose 3\left[\frac{d}{2}\right]-2}
       \lambda_{\left[\frac{d}{2}\right]}\lambda_{d-\left[\frac{d}{2}\right]}
     \ \ge \    \frac{1}{2}
       {3d-4\choose 3\left[\frac{d}{2}\right]-2}
       \lambda_{\left[\frac{d}{2}\right]}^2.
\]
On pose $l=[\log_2 d]$, et on
définit les entiers $k_0,k_1,\cdots,k_l$ par
$k_0=d$ et
  $k_{i+1}=\left[\frac{k_i}{2}\right]$. Remarquons que $k_l=1$.
On obtient alors de proche en proche
\[
\lambda_d \ \ge \ \prod_{i=1}^l
\frac{1}{2^{2^{i-1}}}
     {3k_{i-1}-4\choose 3k_i-2}^{2^{i-1}}
      \ = \
\frac{1}{2^{2^l-1}}
     \prod_{i=1}^l
     {3k_{i-1}-4\choose 3k_i-2}^{2^{i-1}}.
\]
Comme $k_i\le k_{i-1} -k_i \le k_i+1$ et $2^l\le d$, on a
\begin{align*}
 \lambda_d & \ \ge \  \frac{1}{2^d}
     \prod_{i=1}^l
     \left(\frac{(3k_{i-1}-4)!}{(3k_i-2)!(3(k_{i-1}
       -k_i)-2)!}\right)^{2^{i-1}}
     \\
     & \ \ge \  \frac{1}{2^d}
     \prod_{i=1}^l
     \left(\frac{(3k_{i-1}-4)!}{(3k_i+1)!^2}\right)^{2^{i-1}}
     \\
     & \ \ge \  \frac{(3k_0-4)!}{2^{d}}
     \prod_{i=1}^{l-1}
     \left(\frac{(3k_{i}-4)!}{(3k_i+1)!}\right)^{2^{i}}
     \times \frac{1}{(3k_l+1)!^{2^l}}
     \\
     & \ \ge \  \frac{(3d-4)!}{48^{d}}
     \prod_{i=1}^{l-1}
     \frac{1}{(3k_i+1)^{5\times 2^{i}}}
    \\
     & \ \ge \  \frac{(3d-4)!}{48^{d}}
     \prod_{i=1}^{l-1}
     \frac{1}{(3\frac{n}{2^i}+1)^{5\times 2^{i}}}.
\end{align*}
En passant au logarithme, on obtient
\[
\ln \lambda_d \ge \ln (3d-4)! - d \ln 48 -
5\sum_{i=1}^{l-1} 2^i \ln (3\frac{n}{2^i}+1).
\]
D'après la formule de Stirling, nous avons l'équivalence 
$\ln k!\sim k\ln k$, et donc
\[
\ln (3 d-4)! \sim (3d-4)\ln(3d-4)\sim 3d \ln d.
\]
Comme $n\le 2^{l+1}$, nous avons aussi
\begin{align*}
  \sum_{i=1}^{l-1} 2^i \ln (3\frac{n}{2^i}+1) & \ \le \
\sum_{i=1}^{l-1} 2^i \ln (3\times 2^{l+1-i}+1)
   \\ &\ \le \
   \sum_{i=1}^{l-1} 2^i \ln (3\times 2^{l+2-i})
     \\ &\ \le \
     \ln 6 \times  \sum_{i=1}^{l-1} 2^i
     + \ln 2\times \sum_{i=1}^{l-1} (l+1-i) 2^i 
     \\ &\ \le \
     \ln 6 \times  2^l
     + \ln 2\times 2^{l}\sum_{i=1}^{l-1} \frac{(l-i+1)}{ 2^{l-i}}
          \\ &\ \le \
     d\ln 6 
     + d\ln 2\times \sum_{j=0}^{+\infty} \frac{(j+1)}{ 2^{j}}
            \\ &\ \le \
     d\ln 6 
     + d\ln 2\times \frac{1}{(1-\frac{1}{2})^2}
                 \\ &\ \le \
     d(\ln 6 +4\ln 2).
\end{align*}
On en déduit alors 
\[
d\ln 48 +5 \times\sum_{i=1}^{l-1} 2^i \ln (3\frac{n}{2^i}+1)=o(3d\ln
d),
\]
et donc 
\[
 \ln (3d-4)! - d \ln 48 -
5\sum_{i=1}^{l-1} 2^i \ln (3\frac{n}{2^i}+1)\sim 3d\ln d.
\]
D'après le Théorème \ref{FD}, on a
\[
\forall d\ge 1\quad N(d,0)\ge W(d) \ge \lambda_d,
\]
et donc
\[
\forall d\ge 1\quad \ln N(d,0)\ge \ln W(d) \ge \ln (3d-4)! - d \ln 48 -
5\sum_{i=1}^{l-1} 2^i \ln (3\frac{n}{2^i}+1).
\]
Le théorème découle maintenant de la 
Proposition \ref{asy N}.
\end{proof}

\bibliographystyle{alpha}
\bibliography{../../FrBiblio.bib}

\end{document}